\documentclass[11pt]{amsart}

\usepackage{amsmath,amssymb,amsthm}
\usepackage{amsmath,amssymb,amsthm,amscd}
\usepackage[frame,cmtip,arrow,matrix,line,graph,curve]{xy}
\usepackage{graphpap, color}
\usepackage[mathscr]{eucal}
\usepackage{color}

\def\dfrac{\displaystyle\frac}
\def\dsum{\displaystyle\sum}

\def\dmin{\displaystyle\min}

\newtheorem{prop}{Proposition}
\newtheorem{theo}[prop]{Theorem}
\newtheorem{lemm}[prop]{Lemma}
\newtheorem{coro}[prop]{Corollary}
\newtheorem{rema}[prop]{Remark}

\newtheorem{defi}[prop]{Definition}
\newtheorem{conj}[prop]{Conjecture}

\newcommand{\al}{\alpha}
\newcommand{\abc}[1]{\left( #1 \right)}
\renewcommand{\leq}{\leqslant}
\renewcommand{\geq}{\geqslant}

\newcommand{\p}{\partial}

\numberwithin{equation}{section}

\title{The global curvature estimate for the $n-2$ Hessian  equation}

\begin{document}

\author{Changyu Ren}
\address{School of Mathematical Science\\
Jilin University\\ Changchun\\ China}
\email{rency@jlu.edu.cn}
\author{Zhizhang Wang}
\address{School of Mathematical Science\\ Fudan University \\ Shanghai, China}
\email{zzwang@fudan.edu.cn}
\thanks{Research of the first author is supported by NSFC Grant No. 11871243 and the second author is supported  by NSFC Grants No.11871161 and 11771103.}
\begin{abstract}
This paper establishes the global curvature estimate for the $n-2$ curvature equation with the general right hand side which partially solves this longstanding problem.
\end{abstract}
\maketitle

\section{introduction}

The present paper continues to  study the longstanding problem that how to establish the global $C^2$ or curvature estimate for the $k$-Hessian equation or the $k$-curvature equation with the right hand side function depending on the gradient term or the normal direction.

Suppose $\Omega\subset\mathbb{R}^n$ is a bounded domain and $u(x)$ is an unknown function defined on $\Omega$, namely $x\in\Omega$. Denote $Du$ and $D^2u$ to be the gradient and the Hessian of $u$. We let $\psi(x,\xi,p)$ be a function defined on $(x,\xi,p)\in \bar{\Omega}\times \mathbb{R}\times \mathbb{R}^n$.
The $k$-Hessian equation means
\begin{eqnarray}\label{s1.03}
\sigma_k(D^2 u)=\psi(x,u,Du),
\end{eqnarray}
for $1\leq k\leq n$, where $\sigma_k$ are the $k$-th elementary symmetric function
defined by, for $\kappa=(\kappa_1,\kappa_2,\cdots,\kappa_n)\in\mathbb{R}^n$ and $1\leq k\leq n$,
\begin{eqnarray}\label{a2.1}
\sigma_k(\kappa)=\sum_{1\leq i_1<\cdots<i_k\leq
n}\kappa_{i_1}\cdots \kappa_{i_k}.
\end{eqnarray}
Thus, $\sigma_k(D^2 u)$ is defined to be the summation of the  $k$-th order principal minors of the Hessian matrix $D^2u$.

Following \cite{CNS3}, we need to introduce a set (admissible set): the  G\r{a}rding's cone, which guarantees  the ellipticity of the $k$-Hessian equation \eqref{s1.03}.
\begin{defi}\label{k-convex} For a domain $\Omega\subset \mathbb R^n$, a function $v\in C^2(\Omega)$ is called $k$-convex if the eigenvalues $\kappa (x)=(\kappa_1(x), \cdots, \kappa_n(x))$ of the hessian $\nabla^2 v(x)$ is in $\Gamma_k$ for all $x\in \Omega$, where $\Gamma_k$ is the G\r{a}rding's cone
\begin{equation}\label{def G2}
\Gamma_k=\{\kappa \in \mathbb R^n \ | \quad \sigma_m(\kappa)>0,
\quad  m=1,\cdots,k\}.\nonumber
\end{equation}
\end{defi}

Corresponding to the $k$-Hessian equation, we can propose the $k$-curvature equation. Suppose $M$ is an $n$ dimensional compact hypersurface in the $n+1$ dimensional Euclidean space $\mathbb{R}^{n+1}$ and we use $X$ to denote its position vector.
We let
$\nu(X), \kappa(X)$ be the unit outer-normal vector and principal
curvature vector of $M\subset R^{n+1}$ at position vector $X$
respectively. The prescribed $k$-Hessian curvature equation is
\begin{eqnarray}\label{s1.01}
\sigma_k(\kappa(X))=\psi(X,\nu).
\end{eqnarray}
Here, $\psi$ is a function defined on $(X,\nu)\in\mathbb{R}^{n+1}\times\mathbb{S}^{n}$.
A $C^2$ regular hypersurface $M\subset R^{n+1}$ is called $k-$convex if its principal curvature vector
$\kappa(X)\in \Gamma_k$ for all $X\in M$.

To establish  the global $C^2$ estimate for the  $k$-Hessian equation \eqref{s1.03} or the global curvature estimate  for the $k$-curvature equation \eqref{s1.01} is some longstanding problem, which was clearly posed by Guan-Li-Li in \cite{GLL} at first. Moreover, it is  very nature  to  consider
the equation \eqref{s1.01} with the right hand side function containing the gradient term or
the normal vector in view of the papers \cite{CNS3,CNS4, CNS5}.

Let's brief review the research history on the global $C^2$ estimate for the Hessian equation.
 For $\psi$ being independent of the normal vector or the gradient term, the $C^2$-estimate
was obtained by Caffarelli-Nirenberg-Spruck \cite{CNS3} for a general class of fully nonlinear operators, including the Hessian type and the quotient Hessian type.
The Pogorelov type interior $C^2$ estimate for the Hessian equation have been obtained by Chou-Wang \cite{CW}.
Sheng-Urbas-Wang \cite{SUW} obtained the Pogorelov type interior $C^2$ estimate for the curvature equation  of the graphic hypersurface.
$C^2$ estimates for the complex Hessian equations defined on K\"ahler manifolds have been obtain by Hou-Ma-Wu  \cite{HMW}.
The curvature estimate was also established for the equation of the prescribing curvature measure problem by Guan-Li-Li \cite{GLL} and Guan-Lin-Ma \cite{GLM}.
If the function $\psi$ is  convex with respect to the normal or the gradient, it is well
known  that the global $C^2$ estimate has been obtained by Guan \cite{B}. Recently, Guan \cite{B2}  obtained an important result of  $C^2$ estimates for some  fully nonlinear equations defined on Riemannian manifolds.

If  the right hand side function $\psi$ depending  on the gradient term or the normal, let's give a review on some recent obtained results.
Guan-Ren-Wang \cite{GRW} obtained the global curvature estimate of the closed convex hypersurface and the star-shaped $2$-convex hypersurface.
The complex corresponded theorem has been established by Phong-Picard-Zhang \cite{PPZ3} on the Hemitian manifold. Li-Ren-Wang \cite{LRW} substituted the convex condition  to $k+1$- convex condition for any Hessian equations and derived the Pogorelov type interior $C^2$ estimate. For the case $k=n-1$, Ren-Wang \cite{RW} completely solved the longstanding problem, that they obtained the global curvature estimates of $n-1$ convex solutions for $n-1$ Hessian equations.
Chen-Li-Wang \cite{CLW} established the global curvature estimate for the prescribed curvature problem in arbitrary warped product spaces. Li-Ren-Wang \cite{LiRW} considered the global curvature estimate of convex solutions for a class of general  Hessian equations.

Let's  review some geometric applications of the Hessian equation.
The  famous  Minkowski problem, namely, the prescribed Gauss-Kronecker curvature on the outer normal, has been widely discussed by Nirenberg \cite{N}, Pogorelov \cite{P3}, Cheng-Yau \cite{CY}.
 Alexandrov also posed the problem of prescribing general Weingarten curvature on the outer normal \cite{A2, gg}.
The prescribing curvature measure problem in convex geometry also
has been extensively studied by Alexandrov \cite{A1}, Pogorelov
\cite{P1}, Guan-Lin-Ma \cite{GLM}, Guan-Li-Li \cite{GLL}. The
prescribing mean curvature problem and Weingarten curvature problem
also have been considered and obtained fruitful results by
Bakelman-Kantor \cite{BK}, Treibergs-Wei \cite{TW}, Oliker \cite{O},
Caffarelli-Nirenberg-Spruck \cite{CNS4, CNS5}. The prescribed
curvature problems in Riemannian manifolds have been considered  by
Li-Oliker \cite{LO}, Barbosa-deLira-Oliker \cite{BL}  and
Andrade-Barbosa-de Lira \cite{AB}. Spruck-Xiao \cite{SX} established the
curvature estimate for the prescribed scalar curvature problem in space forms and
gave a simple proof of our curvature estimate for the same equation in the Euclidean space.
Phong-Picard-Zhang \cite{PPZ1,PPZ2} has considered Fu-Yau equation in
high dimensional spaces, which is a complex $2$-Hessian equation
 with the right hand side function depending on the gradient term.
 Guan-Lu \cite{GL} considered the curvature estimate for the isometric embedding system in general Riemannian manifold, which can derive a $2$-Hessian curvature equation with the right hand side function depending on the
 normal.
 Our estimate established in \cite{GRW} is also applied by Xiao \cite{X} and Bryan-Ivaki-Scheuer \cite{BIS} respectively  in some geometric flow problems. Very recently, Ren-Wang-Xiao \cite{RWXiao} considers entire space like hypersurfaces with constant $\sigma_{n-1}$ curvature in Minkowski space, using some techniques developing in \cite{GRW, LiRW} and \cite{RW}.

Now, we state our main theorem.
\begin{theo}\label{theo2}
Suppose $M\subset \mathbb R^{n+1}$ is a closed ${n-2}$-convex
hypersurface satisfying the curvature equation (\ref{s1.01}) with
$k=n-2$ for some positive function $\psi(X, \nu)\in C^{2}(\Gamma)$,
where $\Gamma$ is an open neighborhood of the unit normal bundle of
$M$ in $\mathbb R^{n+1} \times \mathbb S^n$, then there is a
constant $C$ depending only on $n, k$, $\|M\|_{C^1}$, $\inf \psi$
and $\|\psi\|_{C^2}$, such that
 \begin{equation}\label{Mc2}
 \max_{X\in M, i=1,\cdots, n} \kappa_i(X) \le C.\end{equation}
\end{theo}
\par

A straightforward corollary of the above theorem is the following $C^2$ estimate for the Dirichlet problem of the $n-2$ Hessian equation.
\begin{coro} \label{Coro}
For the Dirichlet problem (\ref{s1.03}) of the $\sigma_{n-2}$
equation defined on some bounded domain $\Omega\subset
\mathbb{R}^n$,
there exits some constant $C$ depending on $\psi$ and $\nabla u,
u$ and the domain $\Omega$, such that we have the global $C^2$ estimate
$$\|u\|_{C^2(\bar{\Omega})}\leq C+\max_{\p\Omega}|\nabla^2 u|.$$
\end{coro}

\begin{rema} We believe that Theorem \ref{theo2} holds for $2k>n$. In section 3, we will establish the estimate \eqref{Mc2}, when we assume that the Conjecture \ref{con} holds.
\end{rema}
In section 4, we will prove that the left hand side of the inequality \eqref{s3.01}
is bigger than the summation of four quadratic forms
$\textbf{A}_{k;i},\textbf{B}_{k;i},\textbf{C}_{k;i},
\textbf{D}_{k;i}$ defined there. In view of the four forms,
 $k=n-1$ and $k=n-2$ are two special cases, namely
that some terms disappear. The case $k=n-3$ in fact owns the general
expression. However, $k=n-1$ and $k=n-2$ are two essential cases we need
to solve at first. Thus, the present paper is a very important step to solve the whole conjecture.
Comparing the case $k=n-1$, the difficulty of the case $k=n-2$
is that $\textbf{C}_{k;i}$
becomes more complicated. The novelty of our
new proof lies that we discover the method to compete square of the
four quadratic forms, which is very different from the argument of the
case $k=n-1$ provided in \cite{RW}.

As in \cite{RW}, our new estimate can give two geometric applications. The first one is that we can solve the prescribed $n-2$
curvature equation (\ref{s1.01}) in the cone $\Gamma_{n-2}$. Suppose $r_1,r_2$ and Condition (1.4), Condition (1.5) are same as \cite{RW}. We have
\begin{theo}\label{exist} Suppose $k=n-2$ and the positive function $\psi\in C^2(\bar B_{r_2}\setminus B_{r_1}\times \mathbb S^n)$ satisfies conditions (1.4) and (1.5) in \cite{RW}, then the equation (\ref{s1.01}) has an unique $C^{3,\alpha}$ starshaped solution $M$ in $\{r_1\le |X|\le r_2\}$.
\end{theo}

\par
Here, $M$ is called a star-shaped hypersurface, if it can be viewed as a radial
graph defined on $\mathbb{S}^n$ and its radial function is positive. Thus, the support function of a closed star-shaped hypersurface with respect to its  outer normal direction is positive.
\par
The second application is to solve the prescribed $n-2$ curvature
problem of the $n$ dimensional spacelike graphic hypersurface in
Minkowski space $\mathbb{R}^{n,1}$. The Minkowski space
$\mathbb{R}^{n,1}$ is the set $\mathbb{R}^{n+1}$ endowed with the
indefinite metric
$$ds^2=dx_1^2+\cdots +dx_{n}^2-dx_{n+1}^2.$$ The terminology "spacelike" means the tangent space of the hypersurface lies in the outside of the light cone. We have
\begin{theo}\label{theo5}
 Let $\Omega$ be some bounded domain in $\mathbb{R}^n$ with smooth boundary and $\psi(x,\xi,p)\in C^2(\bar{\Omega}\times \mathbb{R}\times \mathbb{R}^n)$ be a positive function with $\psi_{\xi}\geq 0$. Let  $\varphi\in C^4(\bar{\Omega})$ be space like. Consider  the following Dirichlet problem
\begin{eqnarray}\label{s1.06}
\left\{\begin{matrix}\sigma_{n-2}(\kappa_1,\cdots,\kappa_n)=\psi(X,\nu),
& \text{ in } \Omega\\ \qquad u=\varphi,   &\text{ on }
\partial \Omega\end{matrix}\right.,
\end{eqnarray}
where $\kappa_1,\kappa_2,\cdots,\kappa_n$ and $X,\nu$ are the principal
curvatures and the position vector, the unit normal of a space like graphic hypersurface defined by $(x,
u(x))$ in Minkowski space $\mathbb{R}^{n,1}$. If the  problem
\eqref{s1.06} has a sub solution, then it has an unique space like
solution $u$ in $\Gamma_{n-2}$ belonging to
$C^{3,\alpha}(\bar{\Omega})$ for any $\alpha\in (0,1)$.
\end{theo}
The prescribing curvature problem of the spacelike graphic hypersurface
over some bounded domain in Minkowski space has been studied by various authors, such as Bartnik-Simon \cite{BS}, Delan\"o \cite{De}, Bayard \cite{Ba0,Ba}, Urbas \cite{U}, etc. In \cite{RW}, Ren-Wang solves the prescribed $n-1$ curvature problem. Here, we solves the prescribed $n-2$ curvature problem. The other cases $2<k<n-2$ are still open.  Hypersurfaces with prescribed curvatures
in Lorentzian manifolds also attract some attentions, seeing papers \cite{Ger1,Ger2,Sch} and reference therein.

Here, we omit the proofs of Corollary \ref{Coro}, Theorem \ref{exist} and Theorem \ref{theo5} in this paper, since they are same as \cite{RW}, if we have Theorem \ref{theo2}.

The organization of our paper is as follows. In section 2, we explain more notations and list facts and lemmas repeatedly used in the paper. In section 3, when we assume a key inequality holds and $2k>n$, we can establish the curvature estimate.
In section 4, we obtain that the proof of the key inequality can be deduced from four quadratic forms $\textbf{A}_{k;i},\textbf{B}_{k;i},\textbf{C}_{k;i},
\textbf{D}_{k;i}$. Further, to derive the key inequality for the case $k=n-2$, we need to divide into five cases, which will be proved in the last two sections.
In section 5, we will prove more algebraic lemmas, which is prepared to prove the key inequality for $k=n-2$.  In section 6, by competing square, we can establish the key inequality for the first three cases. In the last section, we prove the key inequality for the last two cases, applying the idea of handing the convex solutions developing in \cite{GRW,LRW}.

\section{Preliminary}
The operator $\sigma_k(\kappa)$ for $\kappa=(\kappa_1,\kappa_2,\cdots,\kappa_n)\in\mathbb{R}^n$ has been defined by \eqref{a2.1}. For convenience, we further set $\sigma_0(\kappa)=1$ and $\sigma_k(\kappa)=0$ for $k>n$
or $k<0$. Following \cite{CNS3}, the G\r{a}rding's cone $\Gamma_k$ is an open,
convex, symmetric (invariant under the interchange of any two
$\kappa_i$) cone with vertex at the origin,  containing the positive
cone: $\{\kappa\in \mathbb{R}^n ;$ each component $\kappa_i>0,1\leq
i\leq n\}$. Korevaar \cite{Kor} has shown that the cone $\Gamma_k$
also can be characterized as
\begin{eqnarray}\label{Gamma}
&\left\{\kappa\in
\mathbb{R}^n;\sigma_k(\kappa)>0,\frac{\partial\sigma_k(\kappa)}{\partial\kappa_{i_1}}>0,
\cdots,\frac{\partial^k\sigma_k(\kappa)}{\partial\kappa_{i_1}\cdots\partial
\kappa_{i_k}}>0, \text{ for all } 1\leq i_{1}<\cdots<i_{k}\leq
n\right\}.\nonumber
\end{eqnarray}
Suppose $\kappa_1\geq\cdots\geq\kappa_n$, then using the above fact, we have
\begin{equation}\label{Gamma1}
\kappa_k+\kappa_{k+1}+\cdots+\kappa_n> 0\quad {\rm for}\quad
\kappa\in\Gamma_k.
\end{equation}
Thus, if $\kappa\in\Gamma_k$, the number of possible negative entries of $\kappa$ is at most $n-k$.

For convenience, in this paper, we define another set:
\begin{equation*}
\bar\Gamma_k=\{\kappa \in \mathbb R^n \ | \quad \sigma_m(\kappa)>0,
\quad m=1,\cdots,k-1 ~~{\rm and}~~ \sigma_k(\kappa)\geq 0\}.
\end{equation*}
Note that the meaning of $\bar{\Gamma}_k$ is not the closure of $\Gamma_k$ in $\mathbb{R}^n$.
Thus, we clearly have
$$
\Gamma_n\subset\bar\Gamma_n\subset\cdots\subset\Gamma_k\subset\bar\Gamma_k\subset\cdots\subset\Gamma_1\subset\bar\Gamma_1.
$$
Same as \cite{Kor}, using the fact that
$\dfrac{\sigma_k}{\sigma_{k-1}}$ is degenerated elliptic in
$\Gamma_{k-1}$, $\bar\Gamma_k$ also can be characterized as
\begin{eqnarray}
&\left\{\kappa\in \mathbb{R}^n;\sigma_k(\kappa)\geq
0,\frac{\partial\sigma_k(\kappa)}{\partial\kappa_{i_1}}\geq 0,
\cdots,\frac{\partial^k\sigma_k(\kappa)}{\partial\kappa_{i_1}\cdots\partial
\kappa_{i_k}}\geq 0, \text{ for all } 1\leq i_{1}<\cdots<i_{k}\leq
n\right\}.\nonumber
\end{eqnarray}

Let $\kappa(A)$ be
the eigenvalue vector of a matrix $A=(a_{ij})$.
Suppose $F$ is a function defined on the set of symmetric matrices. We let
$$f\left(\kappa (A)\right)=F(A).$$ Thus, we denote
$$F^{pq}=\frac{\p F}{\p a_{pq}}, \text{ and  } F^{pq,rs}=\frac{\p^2 F}{\p a_{pq}\p a_{rs}}.$$
For a local orthonormal frame, if $A$ is diagonal at a point, then
at this point, we have
$$F^{pp}=\frac{\p f}{\p \kappa_p}=f_p, \text{ and }  F^{pp,qq}=\frac{\p^2 f}{\p \kappa_p\p \kappa_q}=f_{pq}.$$

Thus the definition of the $k$-th elementary symmetric function can be extended to symmetric matrices. Suppose $W$ is an $n\times n$ symmetric matrix and $\kappa(W)$ is its eigenvalue vector. We define
$$\sigma_k(W)=\sigma_k(\kappa(W)),$$ which is the summation of the $k$-th principal minors of the matrix $W$.

Now we will list some algebraic identities and properties of
$\sigma_k$. In this paper, we will denote
$(\kappa|a)=(\kappa_1,\cdots,\kappa_{a-1},\kappa_{a+1},\cdots,\kappa_n)$.
For any $1\leq l\leq n$, the notation
$\sigma_l(\kappa|ab\cdots)$ means $\sigma_l((\kappa|ab\cdots))$.
Thus, we define

\par
\noindent (i) $\sigma^{pp}_k(\kappa):=\frac{\partial
\sigma_k(\kappa)}{\partial \kappa_p}=\sigma_{k-1}(\kappa|p)$ for any given index
$p=1,\cdots,n$;
\par
\noindent (ii) $\sigma^{pp,qq}_k(\kappa):=\frac{\partial^2
\sigma_k(\kappa)}{\partial \kappa_p\partial
\kappa_q}=\sigma_{k-2}(\kappa|pq)$  for any given indices
$p,q=1,\cdots,n$ and $\sigma^{pp,pp}_k(\kappa)=0$.\\
Using the above definitions, we have
\par
\noindent (iii)
$\sigma_k(\kappa)=\kappa_i\sigma_{k-1}(\kappa|i)+\sigma_k(\kappa|i)$ for any given index $i$;
\par
\par
\noindent (iv)
$\dsum_{i=1}^n\sigma_{k}(\kappa|i)=(n-k)\sigma_k(\kappa)$;\\
\par
\noindent (v)
$\dsum_{i=1}^n\kappa_i\sigma_{k-1}(\kappa|i)=k\sigma_k(\kappa)$.\\

\noindent Thus, for a Codazzi  tensor $W=(w_{ij})$, we have

\noindent (vi)
$-\sum_{p,q,r,s}\sigma^{pq,rs}_k(W)w_{pql}w_{rsl}=\sum_{p,q}\sigma_k^{pp,qq}(W)w_{pql}^2-\sum_{p,q}\sigma^{pp,qq}_k(W)w_{ppl}w_{qql}$,\\
\noindent where $w_{pql}$ means the covariant derivative of $w_{pq}$ with respect to $l$ and $\sigma_k^{pq,rs}(W)=\frac{\p^2 \sigma_k(W)}{\p w_{pq}\p w_{rs}}$.  The meaning of Codazzi tensors can be found in \cite{GRW}.

\noindent For $\kappa\in \Gamma_k$, suppose
$\kappa_1\geq\cdots\geq\kappa_n$, then we have
\par
\noindent (vii) $\sigma_{k-1}(\kappa|n)\geq \cdots\geq
\sigma_{k-1}(\kappa|1)>0$.\\
More
details about the proof of these formulas can be found in \cite{HS}
and \cite{Wxj}. \\
For $\kappa\in\mathbb{R}^n$, we have the famous Newton's
inequality and the Maclaurin's inequality.
\par
\noindent (viii) $
\dfrac{\sigma_{k-1}^2(\kappa)}{C_n^{k-1}C_n^{k-1}}\geq
\dfrac{\sigma_k(\kappa)\sigma_{k-2}(\kappa)}{C_n^kC_n^{k-2}}, $ for
$ k\geq 2, \kappa\in \mathbb{R}^n$. (Newton's inequality)
\par
\noindent (ix)
$
\Big[\dfrac{\sigma_k(\kappa)}{C_n^k}\Big]^{1/k}\leq\Big[\dfrac{\sigma_l(\kappa)}{C_n^l}\Big]^{1/l}
$
for $k\geq l\geq 1, \kappa\in\Gamma_k$. (Maclaurin's inequality)\\
Using Newton's inequality, it is not difficult to prove\\
\noindent (x) $
\dfrac{\sigma_s(\kappa)\sigma_k(\kappa)}{C_n^sC_n^k}\geq\dfrac{\sigma_{s-r}(\kappa)\sigma_{k+r}(\kappa)}{C_n^{s-r}C_n^{k+r}},
$ for $1\leq r\leq s\leq k$, $\kappa\in \Gamma_k$. \par\noindent
Here $C^k_n$ is the combinational number, namely
$C^k_n=\frac{n!}{k!(n-k)!}$.

We need to refer two
important concavities. The functions $\sigma_k^{1/k}(\kappa)$ and
$\left(\frac{\sigma_k(\kappa)}{\sigma_l(\kappa)}\right)^{1/(k-l)}$
for $l<k$ are concave functions in $\Gamma_k$ which is showed in
\cite{CNS3} and \cite{Tru1}.

\par
Now, we list several lemmas frequently used in the other sections.
\begin{lemm} \label{Guan}
Assume that $k>l$, $W=(w_{ij})$ is a Codazzi tensor which is in
$\Gamma_k$. Denote $\al=\dfrac{1}{k-l}$.  Then, for $h=1,\cdots, n$
and any $\delta>0$, we have the following inequality
\begin{eqnarray}\label{s2.03}
&&-\sum_{p,q}\sigma_k^{pp,qq}(W)w_{pph}w_{qqh} +\left(1-\al+\dfrac{\al}{\delta}\right)\dfrac{(\sigma_k(W))_h^2}{\sigma_k(W)}\\
&\geq &\sigma_k(W)(\al+1-\delta\al)
\abc{\dfrac{(\sigma_l(W))_h}{\sigma_l(W)}}^2
-\dfrac{\sigma_k}{\sigma_l}(W)\sum_{p,q}\sigma_l^{pp,qq}(W)w_{pph}w_{qqh}.\nonumber
\end{eqnarray}
\end{lemm}
The proof can be found in \cite{GLL} and \cite{GRW}. Now we give another Lemma whose proof is same as the proof of the inequality (12)
in \cite{LT}.
\begin{lemm} \label{lemm7}
Assume that $\kappa=(\kappa_1,\cdots,\kappa_n)\in\Gamma_{k}$. Then
for any given indices $1\leq i,j\leq n$, if $\kappa_i\geq\kappa_j$, we have
$$
|\sigma_{k-1}(\kappa|ij)|\leq \Theta\sigma_{k-1}(\kappa|j), \text{ where } \Theta=\sqrt{\dfrac{k(n-k)}{n-1}}.
$$
\end{lemm}

We also have
\begin{lemm}\label{lemm8}
Assume that $\kappa=(\kappa_1,\cdots,\kappa_n)\in \Gamma_k$ and $\kappa_1\geq  \cdots\geq \kappa_n$. Then for any
$0\leq s\leq k\leq n$, we have
\begin{eqnarray}\label{n2.5}
\dfrac{\kappa_1^s\sigma_{k-s}(\kappa)}{\sigma_k(\kappa)}\geq\dfrac{C_n^{k-s}}{C_n^k}.
\end{eqnarray}

\end{lemm}

\begin{proof}
Obviously, we have $\kappa_1>0$. Define
$\tilde\kappa=\dfrac{\kappa}{\kappa_1}=\left(1,\cdots,\dfrac{\kappa_n}{\kappa_1}\right)$.
Thus, we have $\dfrac{\sigma_k(\tilde\kappa)}{C_n^k}\leq 1$ and
$\tilde{\kappa}\in\Gamma_k$. By Maclaurin's inequality, we get
\begin{align*}
\dfrac{\sigma_{k-s}(\tilde\kappa)}{C_n^{k-s}}\geq
\Big[\dfrac{\sigma_k(\tilde\kappa)}{C_n^k}\Big]^{\frac{k-s}{k}}\geq
\dfrac{\sigma_k(\tilde\kappa)}{C_n^k},
\end{align*}
which implies
\begin{align*}
\dfrac{\kappa_1^s\sigma_{k-s}(\kappa)}{\sigma_k(\kappa)}=\dfrac{\sigma_{k-s}(\tilde\kappa)}{\sigma_{k}(\tilde\kappa)}\geq\dfrac{C_n^{k-s}}{C_n^k}.
\end{align*}

\end{proof}

Using the above lemma, we can prove
\begin{lemm}\label{lemm9}
Assume that $\kappa=(\kappa_1,\cdots,\kappa_n)\in\Gamma_{k}$ and $\kappa_1\geq  \cdots\geq \kappa_n$. Suppose any given indices $i,j$ satisfy $1\leq i,j\leq n$ and $i\neq j$.

(a) If $\kappa_i\leq 0$, then $-\kappa_i<\dfrac{(n-k)\kappa_1}{k}.$

(b) If $\kappa_i\leq\kappa_j\leq0$, then
$-(\kappa_i+\kappa_j)<\dfrac{2\sigma_{k}(\kappa|ij)}{\sigma_{k-1}(\kappa|ij)}.$
\end{lemm}
\begin{proof}
(a) Since $\kappa=(\kappa_1,\cdots,\kappa_n)\in\Gamma_{k}$, by
$$
\sigma_{k}(\kappa)=\kappa_i\sigma_{k-1}(\kappa|i)+\sigma_{k}(\kappa|i)>0,
\text{ and } \kappa_i\leq 0,
$$
we know that $\sigma_{k}(\kappa|i)>0$, which implies
$(\kappa|i)\in\Gamma_k$.
Applying Lemma \ref{lemm8} to $(\kappa|i)$ and using the above inequality, we get
$$
-\kappa_i<\dfrac{\sigma_{k}(\kappa|i)}{\sigma_{k-1}(\kappa|i)}\leq\dfrac{C_{n-1}^{k}\kappa_1}{C_{n-1}^{k-1}}=\frac{(n-k)\kappa_1}{k}.
$$

(b) Same as (a), using $\kappa_j\leq 0$, we know
$\sigma_{k}(\kappa|j)>0$. Thus, it is clear that
$$
\sigma_{k}(\kappa|j)=\kappa_i\sigma_{k-1}(\kappa|ij)+\sigma_{k}(\kappa|ij)>0,
$$
then, rewriting the above inequality, we have $
-\kappa_i<\dfrac{\sigma_{k}(\kappa|ij)}{\sigma_{k-1}(\kappa|ij)}.
$
Since $\kappa_i\leq\kappa_j\leq 0$, the above inequality implies
$-(\kappa_i+\kappa_j)<\dfrac{2\sigma_{k}(\kappa|ij)}{\sigma_{k-1}(\kappa|ij)}.$

\end{proof}

\begin{lemm}\label{lemm10}
Assume that $\kappa=(\kappa_1,\cdots,\kappa_n)\in\bar\Gamma_k$,
$1\leq k\leq n$, and
$\kappa_1\geq\cdots\geq\kappa_n$. Then for any $1\leq s<k$, we have
 $$\sigma_{s}(\kappa)\geq\kappa_1\cdots\kappa_{s}.$$
\end{lemm}
\begin{proof}
Using $\kappa\in\bar\Gamma_k\subset\Gamma_s$, we have
$$\kappa_1>0, \kappa_{2}>0,\cdots,\kappa_{s}>0,$$ and
$$\sigma_s(\kappa|1)\geq 0, \sigma_{s-1}(\kappa|12)\geq 0, \cdots, \sigma_1(\kappa|12\cdots s)\geq 0 .$$
Using the above inequalities, we get
\begin{align*}
\sigma_{s}(\kappa)=&\kappa_1\sigma_{s-1}(\kappa|1)+\sigma_s(\kappa|1)\\
\geq &\kappa_1\sigma_{s-1}(\kappa|1)=\kappa_1\kappa_{2}\sigma_{s-2}(\kappa|12)+\kappa_1\sigma_{s-1}(\kappa|12)\\
\geq &\kappa_1\kappa_{2}\sigma_{s-2}(\kappa|12)=\cdots\\
\geq &\kappa_1\cdots\kappa_{s}.
\end{align*}
\end{proof}

\begin{lemm}\label{lemm11}
Assume that $\kappa=(\kappa_1,\cdots,\kappa_n)\in\Gamma_k$, $1\leq
k\leq n$, and $\kappa_1\geq\cdots\geq\kappa_n$. For
any given indices $1\leq j\leq k$, there exists a positive constant $\theta$
only depending on $n,k$ such that
\begin{align*}
\sigma_{k}^{jj}(\kappa)\geq\dfrac{\theta\sigma_k(\kappa)}{\kappa_j}.
\end{align*}
Especially, we have $\kappa_1\sigma_{k}(\kappa|1)\geq \theta\sigma_k(\kappa)$.
\end{lemm}
\begin{proof}
We note that $\kappa_j>0$. We divide into two cases to prove our
Lemma.
\par
(a) If we have $\sigma_k(\kappa|j)\leq 0$, we easily see that
\begin{align*}
\sigma_{k}^{jj}(\kappa)=\dfrac{\sigma_k(\kappa)-\sigma_k(\kappa|j)}{\kappa_j}\geq\dfrac{\sigma_k(\kappa)}{\kappa_j}.
\end{align*}
\par
(b) If we have  $\sigma_k(\kappa|j)> 0$, using
$\kappa\in\Gamma_k$, we have $(\kappa|j)\in\Gamma_k$.
Thus, applying Lemma \ref{lemm10} to $(\kappa|j)$, we get
\begin{align*}
\sigma_{k}^{jj}(\kappa)=\sigma_{k-1}(\kappa|j)\geq
\dfrac{\kappa_{1}\cdots\kappa_k}{\kappa_j}.
\end{align*}
In view of \eqref{Gamma1}, for any $j>k$, we have $|\kappa_j|\leq n\kappa_k$. Thus, there exists some constant $\theta$ only depending on $n,k$ such that
$\theta\sigma_k(\kappa)\leq \kappa_1\cdots\kappa_k$. Therefore, we obtain our lemma.

\end{proof}

\section{The curvature estimate}
In this section, following the steps of \cite{RW}, we establish the global curvature
estimate. At first, we pose a conjecture.
\begin{conj} \label{con}
Assume that $\kappa=(\kappa_1,\cdots,\kappa_n)\in\Gamma_{k}$,
$2k>n$, $\kappa_1$ is the maximum entry of $\kappa$, and
$\sigma_k(\kappa)$ has the absolutely positive lower bound and upper
bound, $N_0\leq \sigma_k(\kappa)\leq N$. For any given index $1\leq i\leq n$, if
$\kappa_i>\kappa_1-\sqrt{\kappa_1}/n$, the following quadratic form
is non negative,
\begin{align}\label{s3.01}
&\kappa_i\Big[K\Big(\sum_j\sigma_{k}^{jj}(\kappa)\xi_j\Big)^2-\sigma_{k}^{pp,qq}(\kappa)\xi_{p}\xi_{q}\Big]
-\sigma^{ii}_{k}(\kappa)\xi_{i}^2+\sum_{j\neq i}a_j\xi_{j}^2\geq 0,
\end{align}
for any $n$ dimensional vector $\xi=(\xi_1,\xi_2,\cdots,\xi_n)\in\mathbb{R}^n$, when $\kappa_1$ and the constant $K$ are sufficiently large. Here $a_j$ is defined by
\begin{eqnarray}\label{aj}
a_j=\sigma_{k}^{jj}(\kappa)+(\kappa_i+\kappa_j)\sigma_k^{ii,jj}(\kappa).
\end{eqnarray}

\end{conj}
The sections 4-7 will prove the above Conjecture for $k=n-2$. In \cite{RW}, we have proved the Conjecture for $k=n-1$. If $k=n$, the inequality \eqref{s3.01} is well known. Thus, until to now, the above Conjecture holds when $k\geq n-2$.

If we assume the Conjecture \ref{con} holds, we can establish the following global curvature estimate.
\begin{theo}\label{theok}
Suppose $M\subset \mathbb R^{n+1}$ is a closed $k$-convex
hypersurface satisfying the curvature equation (\ref{s1.01}) and
$2k>n$ for some positive function $\psi(X, \nu)\in C^{2}(\Gamma)$,
where $\Gamma$ is an open neighborhood of the unit normal bundle of
$M$ in $\mathbb R^{n+1} \times \mathbb S^n$, then there is a
constant $C$ depending only on $n, k$, $\|M\|_{C^1}$, $\inf \psi$
and $\|\psi\|_{C^2}$, such that
 \begin{equation}
 \max_{X\in M, i=1,\cdots, n} \kappa_i(X) \le C.\nonumber
 \end{equation}
\end{theo}
Since the support function $u$ is positive, we use the following test function
which has been used in \cite{GRW,RW},
$$
\phi=\log\log P-N_1\log u.
$$
Here $N_1$ is an undetermined positive constant and the function $P$
is defined by $$P=\dsum_le^{\kappa_l}. $$ For the given index $1\leq i\leq n$, we denote
\begin{eqnarray}
&&A_i=e^{\kappa_i}\Big[K(\sigma_{k})_i^2-\sum_{p\neq
q}\sigma_{k}^{pp,qq}h_{ppi}h_{qqi}\Big], \ \  B_i=2\sum_{l\neq
i}\sigma_{k}^{ii,ll}e^{\kappa_l}h_{lli}^2, \nonumber \\
&&C_i=\sigma_{k}^{ii}\sum_le^{\kappa_l}h_{lli}^2; \  \
D_i=2\sum_{l\neq
i}\sigma_{k}^{ll}\frac{e^{\kappa_l}-e^{\kappa_i}}{\kappa_l-\kappa_i}h_{lli}^2,
\ \ E_i=\frac{1+\log P}{P\log P}\sigma_{k}^{ii}P_i^2\nonumber.
\end{eqnarray}

Following the arguments in \cite{RW}, using our test
function $\phi$, we can obtain the desired curvature estimate if we can
prove
\begin{align}\label{s3.02}
A_i+B_i+C_i+D_i-E_i\geq 0
\end{align}
for all $i=1,\cdots,n$.

Before proving (\ref{s3.02}), we need the following lemma.
\begin{lemm} \label{lemm13}    
Assume $\kappa=(\kappa_1,\cdots,\kappa_n)\in\Gamma_k$, $2k>n$,
and $\kappa_1$ is the maximum entry of $\kappa$. We let
$\varepsilon_{n,k}=\dfrac{1}{3k}$, then we have
\begin{align}\label{s3.03}
(2-\varepsilon_{n,k})e^{\kappa_l}\sigma_{k-2}(\kappa|il)+(2-\varepsilon_{n,k})\dfrac{e^{\kappa_l}-e^{\kappa_i}}{\kappa_l-\kappa_i}
\sigma_{k-1}(\kappa|l)\geq
\dfrac{e^{\kappa_l}}{\kappa_1}\sigma_{k-1}(\kappa|i)
\end{align}
for all indices $i,l$ satisfying $l\neq i$, if $\kappa_1$ is sufficiently large.
\end{lemm}
\begin{proof}  Obviously we have the following identity,
$$
\sigma_{k-1}(\kappa|l)=\sigma_{k-1}(\kappa|i)+(\kappa_i-\kappa_l)\sigma_{k-2}(\kappa|il).
$$
Multiplying $\dfrac{e^{\kappa_l}-e^{\kappa_i}}{\kappa_l-\kappa_i}$
in  both sides of the above identity, we have
\begin{align}\label{s3.04}
e^{\kappa_l}\sigma_{k-2}(\kappa|il)+\dfrac{e^{\kappa_l}-e^{\kappa_i}}{\kappa_l-\kappa_i}
\sigma_{k-1}(\kappa|l)
=e^{\kappa_i}\sigma_{k-2}(\kappa|il)+\dfrac{e^{\kappa_l}-e^{\kappa_i}}{\kappa_l-\kappa_i}
\sigma_{k-1}(\kappa|i).
\end{align}
Using (\ref{s3.04}), in order to prove \eqref{s3.03}, we only need
to show
\begin{eqnarray}\label{n3.5}
(2-\varepsilon_{n,k})\dfrac{e^{\kappa_l}-e^{\kappa_i}}{\kappa_l-\kappa_i}
\geq\dfrac{e^{\kappa_l}}{\kappa_1},
\end{eqnarray}
which we will divide into four cases to prove.
\par
\noindent Case (a): Suppose $\kappa_l\leq \kappa_i$. We have
\begin{align*}
\dfrac{e^{\kappa_l}-e^{\kappa_i}}{\kappa_l-\kappa_i}
=e^{\kappa_l}\dfrac{e^{\kappa_i-\kappa_l}-1}{\kappa_i-\kappa_l} \geq
e^{\kappa_l}\geq\dfrac{e^{\kappa_l}}{\kappa_1},
\end{align*}
if $\kappa_1$ is sufficiently large.
Here we have used the inequality $e^t>1+t$ for $t>0$.
\par
\noindent Case (b): Suppose $0<\kappa_l-\kappa_i\leq 1$.  By the mean value
theorem, there exists some constant $\xi$ satisfying $\kappa_i<\xi<\kappa_l$, such
that
\begin{align*}
\dfrac{e^{\kappa_l}-e^{\kappa_i}}{\kappa_l-\kappa_i} =e^{\xi}\geq
e^{\kappa_i}\geq e^{\kappa_l-1} \geq\dfrac{e^{\kappa_l}}{\kappa_1},
\end{align*}
if $\kappa_1$ is sufficiently large.
\par
\noindent Case (c): Suppose $1<\kappa_l-\kappa_i\leq \kappa_1$. We have
\begin{align*}
(2-\varepsilon_{n,k})\dfrac{e^{\kappa_l}-e^{\kappa_i}}{\kappa_l-\kappa_i}
  \geq&
(2-\varepsilon_{n,k})e^{\kappa_l}\dfrac{1-e^{-1}}{\kappa_l-\kappa_i}
 \geq
(2-\varepsilon_{n,k})(1-e^{-1})\dfrac{e^{\kappa_l}}{\kappa_1}
 \geq\dfrac{e^{\kappa_l}}{\kappa_1}.
\end{align*}
Here, in the above inequalities, we have used
$(2-\varepsilon_{n,k})(1-e^{-1})>1$.
\par
\noindent Case (d): Suppose $\kappa_l-\kappa_i>\kappa_1$. In this case, our
condition implies  $\kappa_i<0$. By Lemma \ref{lemm9} and $2k>n$, we know that
$-\kappa_i<\dfrac{n-k}{k}\kappa_1\leq\dfrac{k-1}{k}\kappa_1$, then
we have
$$\kappa_l-\kappa_i\leq\kappa_1-\kappa_i<\dfrac{2k-1}{k}\kappa_1.$$
Thus, in this case,
\begin{align*}
(2-\varepsilon_{n,k})\dfrac{e^{\kappa_l}-e^{\kappa_i}}{\kappa_l-\kappa_i}
  \geq&
(2-\varepsilon_{n,k})e^{\kappa_l}\dfrac{1-e^{-\kappa_1}}{\kappa_l-\kappa_i}
 \geq
\dfrac{(2-\varepsilon_{n,k})(1-e^{-\kappa_1})}{(2k-1)/k}\dfrac{e^{\kappa_l}}{\kappa_1}.
\end{align*}
We obviously have
$$\dfrac{(2-\varepsilon_{n,k})}{(2k-1)/k}> 1+\dfrac{1}{3k}.$$
Thus we get
 $\dfrac{(2-\varepsilon_{n,k})(1-e^{-\kappa_1})}{(2k-1)/k}\geq 1$ if
$\kappa_1$ is sufficiently large, which gives the desired inequality.

\end{proof}

We  divide into two cases to establish \eqref{s3.02}, for
$i=1,2\cdots,n$,

\par

\noindent (I) $\kappa_i\leq \kappa_1-\sqrt{\kappa_1}/n $;

\noindent (II) $\kappa_i>  \kappa_1-\sqrt{\kappa_1}/n$. \\

\par
The following lemma deals with the Case (I).
\begin{lemm}\label{lemm14}
Assume $\kappa=(\kappa_1,\kappa_2,\cdots,\kappa_n)\in\Gamma_k$, $n<2k$ and $\kappa_1$ is the maximum
entry of $\kappa$. For any given index $1\leq i\leq n$, if $
 \kappa_i\leq \kappa_1-\sqrt{\kappa_1}/n$, we have
\begin{align*}
A_i+B_i+C_i+D_i-E_i\geq 0,
\end{align*}
when  the constant $K$ and $\kappa_1$ both are sufficiently large.
\end{lemm}
The proof is same as Lemma 13 in \cite{RW}. The difference lies that
we use the inequality \eqref{s3.03} to replace the inequality (4.19)
in \cite{RW}.

For the Case (II), we first prove that
\begin{lemm}\label{lemm16}
Assume $\kappa=(\kappa_1,\kappa_2,\cdots,\kappa_n)\in\Gamma_k$, $n<2k$, $\kappa_1$ is the maximum
entry of $\kappa$ and $\sigma_k(\kappa)$ has a lower bound
$\sigma_k(\kappa)\geq N_0>0$. Then for any given indices $i,j$ satisfying $1\leq i,j\leq n$ and $j\neq i$, if
 $\kappa_i>\kappa_1-\sqrt{\kappa_1}/n$, we have
\begin{align}\label{a3.6}
\dfrac{2\kappa_i(1-e^{\kappa_j-\kappa_i})}{\kappa_i-\kappa_j}\sigma_{k}^{jj}(\kappa)\geq
\sigma_{k}^{jj}(\kappa)+(\kappa_i+\kappa_j)\sigma_k^{ii,jj}(\kappa),
\end{align}
when $\kappa_1$ is sufficiently large.
\end{lemm}
\begin{proof}
If $\kappa_i=\kappa_j$, the left hand side of (\ref{a3.6}) should be
viewed as a limitation when $\kappa_j$ converging to $\kappa_i$, about which we refer \cite{Ball} for more explanation.
It is easy to see that the limitation  is $2\kappa_i\sigma_k^{jj}(\kappa)$. Thus, a straightforward calculation shows
\begin{align*}
&2\kappa_i\sigma_{k-1}(\kappa|j)-\sigma_{k-1}(\kappa|j)-(\kappa_i+\kappa_j)\sigma_{k-2}(\kappa|ij)\\
=&2\kappa_i\sigma_{k-1}(\kappa|j)-2\sigma_{k-1}(\kappa|j)-\sigma_{k-1}(\kappa|i)+2\sigma_{k-1}(\kappa|ij).
\end{align*}
Using Lemma \ref{lemm7}, $|\sigma_{k-1}(\kappa|ij)|$ can be bounded
by $\sqrt{k}\sigma_{k-1}(\kappa|j)$. Thus, since we have
$\sigma_{k-1}(\kappa|i)=\sigma_{k-1}(\kappa|j)$, the above formula
is positive if $\kappa_1$ is sufficiently large.

 If $\kappa_i\neq\kappa_j$, we have the following identity,
\begin{align}\label{n3.7}
&\sigma_{k}^{jj}(\kappa)+(\kappa_i+\kappa_j)\sigma_k^{ii,jj}(\kappa)
=\sigma_{k}^{jj}(\kappa)+(\kappa_i+\kappa_j)\frac{\sigma_k^{jj}(\kappa)-\sigma_k^{ii}(\kappa)}{\kappa_i-\kappa_j}\\
=&\dfrac{2\kappa_i}{\kappa_i-\kappa_j}\sigma_k^{jj}(\kappa)-\dfrac{\kappa_i+\kappa_j}{\kappa_i-\kappa_j}\sigma_k^{ii}(\kappa).\nonumber
\end{align}
In view of \eqref{n3.7}, in order to prove \eqref{a3.6}, it suffices to show
\begin{align}\label{nn}
\dfrac{-2\kappa_ie^{\kappa_j-\kappa_i}}{\kappa_i-\kappa_j}\sigma_{k}^{jj}(\kappa)\geq
-\dfrac{\kappa_i+\kappa_j}{\kappa_i-\kappa_j}\sigma_k^{ii}(\kappa).
\end{align}
Let's define some function:
\begin{eqnarray}\label{n3.8}
L=\left\{\begin{matrix}(\kappa_i+\kappa_j)e^{\kappa_i-\kappa_j}\sigma_k^{ii}(\kappa)-2\kappa_i\sigma_k^{jj}(\kappa)&
\kappa_i>\kappa_j,\\
2\kappa_ie^{\kappa_j-\kappa_i}\sigma_k^{jj}(\kappa)-(\kappa_i+\kappa_j)\sigma_k^{ii}(\kappa)&
\kappa_i<\kappa_j. \end{matrix}\right.
\end{eqnarray}
Obviously, $L\geq 0$ implies \eqref{nn}. Thus, let's prove $L\geq 0$ in the following for $\kappa_i>\kappa_j$ and $\kappa_i<\kappa_j$ respectively.

\bigskip

If $\kappa_i>\kappa_j$, we let $t=\kappa_i-\kappa_j$. Thus we have
$t> 0$. We divide into two cases to prove $L$ is non negative for
$\kappa_i>\kappa_j$.

Case (a): Suppose $t\geq\sqrt{\kappa_1}$. In this case, our
assumption gives $e^t\geq
e^{\sqrt{\kappa_1}}>\dfrac{(\sqrt{\kappa_1})^{2k+1}}{(2k+1)!}$. Here
we have used the Taylor expansion in the second inequality.

If $\kappa_j\leq 0$, using $n\leq 2k-1$ and Lemma \ref{lemm9}, we
have $-\kappa_j<\dfrac{(n-k)\kappa_1}{k}\leq
\dfrac{(k-1)\kappa_1}{k}$. Thus, since
$\kappa_i>\kappa_1-\sqrt{\kappa_1}/n$, we have
$\kappa_i+\kappa_j>\dfrac{\kappa_1}{2k}$ if $\kappa_1>10$. If
$\kappa_j>0$, it is easy to see
$\kappa_i+\kappa_j>\dfrac{\kappa_1}{2k}$. Thus, in any cases, we
have
$$L\geq \frac{\kappa_1^{k+\frac{3}{2}}}{2k (2k+1)!}\sigma_k^{ii}(\kappa)-2\kappa_1\sigma_k^{jj}(\kappa)\geq \kappa_1\left(\frac{\kappa_1^{k-1}\sqrt{\kappa_1}\theta N_0}{2k (2k+1)!}-2\sigma_k^{jj}(\kappa)\right)\geq 0,$$
if $\kappa_1$ is sufficiently large. Here we have used
$\kappa_1\sigma_k^{ii}(\kappa)\geq \theta\sigma_k(\kappa)$.

Case (b): Suppose $t<\sqrt{\kappa_1}$. Using
$\kappa_i>\kappa_1-\sqrt{\kappa_1}/n$, we have $\kappa_j>\kappa_1/2$
if $\kappa_1>10$. For simplification purpose, denote
$\tilde\sigma_m=\sigma_m(\kappa|ij)$. We have
\begin{align}\label{s3.07}
L=&(\kappa_i+\kappa_j)e^{t}\sigma_{k-1}(\kappa|i)-2\kappa_i\sigma_{k-1}(\kappa|j)\\
=&(\kappa_i+\kappa_j)e^{t}(\kappa_j\tilde\sigma_{k-2}+\tilde\sigma_{k-1})-2\kappa_i(\kappa_i\tilde\sigma_{k-2}+\tilde\sigma_{k-1})\nonumber\\
=&[\kappa_j(\kappa_i+\kappa_j)e^{t}-2\kappa_i^2]\tilde\sigma_{k-2}+[(\kappa_i+\kappa_j)e^t-2\kappa_i]\tilde\sigma_{k-1}\nonumber\\
=&[\kappa_j(\kappa_j+\kappa_i)(e^{t}-1)-t(\kappa_j+\kappa_i)-t\kappa_i]\tilde\sigma_{k-2}+[(\kappa_i+\kappa_j)(e^t-1)-t]\tilde\sigma_{k-1},\nonumber
\end{align}
where in the last equality, we have used $t=\kappa_i-\kappa_j$. We
further divide into two sub-cases to prove the nonnegativity of
$L$.

Subcase (b1): Suppose $\tilde\sigma_{k-1}\geq 0$. Note that
$e^t>1+t$. By (\ref{s3.07}) and
$\kappa_i>\kappa_1-\sqrt{\kappa_1}/n, \kappa_j>\kappa_1/2$, we get
$L\geq0$.

Subcase (b2): Suppose $\tilde\sigma_{k-1}< 0$. Inserting the
identity
\begin{align}\label{n311}
(\kappa_i+\kappa_j)\tilde\sigma_{k-1}=\sigma_k(\kappa)-\kappa_i\kappa_j\tilde\sigma_{k-2}-\tilde\sigma_{k}
\end{align}
into the last equality of (\ref{s3.07}), we get
\begin{align}\label{s3.08}
L=&[(\kappa_j^2+\kappa_j\kappa_i)(e^{t}-1)-t(\kappa_j+\kappa_i)-t\kappa_i]\tilde\sigma_{k-2}-t\tilde\sigma_{k-1}\\
&+(e^t-1)[\sigma_k(\kappa)-\kappa_i\kappa_j\tilde\sigma_{k-2}-\tilde\sigma_{k}]\nonumber\\
\geq&[\kappa_j^2(e^t-1)-3t\kappa_i]\tilde\sigma_{k-2}+(e^t-1)[\sigma_k(\kappa)-\tilde\sigma_{k}],\nonumber
\end{align}
where in the last inequality, we have used $\kappa_i\geq
\kappa_j>0,t>0$, and $\tilde\sigma_{k-1}<0$.

For $\sigma_{k-1}(\kappa|i)$ and $\sigma_{k-1}(\kappa|j)$, we have
following estimate
\begin{align}\label{n3.12}
\sigma_{k-1}(\kappa|i)=&\sigma_{k-1}(\kappa|j)-t\sigma_{k-2}(\kappa|ij)\\
=&\sigma_{k-1}(\kappa|j)-\dfrac{t}{\kappa_i}[\sigma_{k-1}(\kappa|j)-\sigma_{k-1}(\kappa|ij)]\nonumber\\
=&\left(1-\dfrac{t}{\kappa_i}\right)\sigma_{k-1}(\kappa|j)+\dfrac{t}{\kappa_i}\sigma_{k-1}(\kappa|ij)\nonumber\\
\geq&\left(1-\dfrac{t(1+\Theta)}{\kappa_i}\right)\sigma_{k-1}(\kappa|j)\geq\dfrac{1}{2}\sigma_{k-1}(\kappa|j),\nonumber
\end{align}
if $\kappa_1$ is sufficiently large. Here in the fourth inequality,
we have used Lemma \ref{lemm7} to estimate the term
$\sigma_{k-1}(\kappa|ij)$. We also have
\begin{align}\label{n3.13}
\kappa_j^2\tilde\sigma_{k-2}+\sigma_k(\kappa)-\tilde\sigma_{k}=&\kappa_j[\sigma_{k-1}(\kappa|i)-\tilde\sigma_{k-1}]+\sigma_k(\kappa)-\tilde\sigma_{k}\\
=&\kappa_j\sigma_{k-1}(\kappa|i)+\sigma_k(\kappa)-\sigma_{k}(\kappa|i)\nonumber\\
=&(\kappa_j+\kappa_i)\sigma_{k-1}(\kappa|i)\nonumber
\end{align}
and
\begin{align}\label{n315}
3\kappa_i=3\kappa_j+3t\leq 3\kappa_j+3\sqrt{\kappa_1}\leq 4\kappa_j,
\end{align}
where we have used $\kappa_j>\kappa_1/2$ and $\kappa_1$ is sufficiently large. Thus, using \eqref{n3.13}
and \eqref{n315}, (\ref{s3.08}) becomes
\begin{align*}
L \geq&(e^t-1)(\kappa_j+\kappa_i)\sigma_{k-1}(\kappa|i)
-3t\kappa_i\tilde\sigma_{k-2}\nonumber\\
\geq&t(\kappa_j+\kappa_i)\sigma_{k-1}(\kappa|i)
-4t\kappa_j\tilde\sigma_{k-2}\nonumber\\
=&t(\kappa_j+\kappa_i-4)\sigma_{k-1}(\kappa|i)+4t\sigma_{k-1}(\kappa|ij)\\
\geq&t(\kappa_j+\kappa_i-4)\sigma_{k-1}(\kappa|j)/2+4t\sigma_{k-1}(\kappa|ij)\\
\geq&\frac{t}{2}(\kappa_j+\kappa_i-4-8\Theta)\sigma_{k-1}(\kappa|j)\geq
0
\end{align*}
if $\kappa_1$ is sufficiently large. Here in the forth inequality,
we have used \eqref{n3.12} and in the last inequality, we have used
Lemma \ref{lemm7} to give the lower bound of
$\sigma_{k-1}(\kappa|ij)$.

\bigskip

If $\kappa_i<\kappa_j$, we let $t=\kappa_j-\kappa_i$, which yields
$0<t\leq \kappa_1-\kappa_i<\sqrt{\kappa_1}/n$. For simplification purpose, we still denote
$\tilde\sigma_m=\sigma_m(\kappa|ij)$. Thus, using
$t=\kappa_j-\kappa_i$ we have
\begin{align}\label{s3.09}
L=&2\kappa_ie^{t}\sigma_{k-1}(\kappa|j)-(\kappa_i+\kappa_j)\sigma_{k-1}(\kappa|i)\\
=&2\kappa_ie^{t}(\kappa_i\tilde\sigma_{k-2}+\tilde\sigma_{k-1})-(\kappa_i+\kappa_j)(\kappa_j\tilde\sigma_{k-2}+\tilde\sigma_{k-1})\nonumber\\
=&[2\kappa_i^2e^{t}-(\kappa_i+\kappa_j)\kappa_j]\tilde\sigma_{k-2}+(2\kappa_ie^t-\kappa_i-\kappa_j)\tilde\sigma_{k-1}\nonumber\\
=&[2\kappa_i^2(e^{t}-1)-3t\kappa_i-t^2]\tilde\sigma_{k-2}+[2\kappa_i(e^t-1)-t]\tilde\sigma_{k-1}.\nonumber
\end{align}
We divide into two cases to prove $L\geq 0$.

Case (c1): Suppose $\tilde\sigma_{k-1}\geq 0$. Since we have
$e^t>1+t$, $t<\sqrt{\kappa_1}/n$ and $\kappa_j>\kappa_i>\kappa_1/2$,
in view of (\ref{s3.09}), we get $L\geq 0$.

Case (c2): Suppose $\tilde\sigma_{k-1}< 0$. Inserting the
identity \eqref{n311} into the last formula of (\ref{s3.09}) and using $2\kappa_i=\kappa_i+\kappa_j-t$, we get
\begin{align}\label{n3.16}
L=&[2\kappa_i^2(e^{t}-1)-3t\kappa_i-t^2]\tilde\sigma_{k-2}-te^t\tilde\sigma_{k-1}\\
&+(e^t-1)[\sigma_k(\kappa)-\kappa_i\kappa_j\tilde\sigma_{k-2}-\tilde\sigma_{k}]\nonumber\\
\geq&[(\kappa_i^2-\kappa_i
t)(e^t-1)-4t\kappa_i]\tilde\sigma_{k-2}+(e^t-1)[\sigma_k(\kappa)-\tilde\sigma_{k}],\nonumber
\end{align}
where in the last inequality, we have used $\tilde\sigma_{k-1}<0$, $t<\sqrt{\kappa_1}/n<\kappa_i$ and $\kappa_1$ is sufficiently large.
Note that comparing the previous case $\kappa_i>\kappa_j$, this case exchanges the position of $i$ and $j$.
Thus, exchanging the indices $i$ and $j$
in \eqref{n3.13} and \eqref{n3.12}, we get the formulae,
\begin{align}\label{n3.17}
\kappa_i^2\tilde\sigma_{k-2}+\sigma_k(\kappa)-\tilde\sigma_{k}=&(\kappa_i+\kappa_j)\sigma_{k-1}(\kappa|j),
\ \ \sigma_{k-1}(\kappa|j)\geq \frac{\sigma_{k-1}(\kappa|i)}{2}.
\end{align}
Combing \eqref{n3.16} with \eqref{n3.17}, we get
\begin{align*}
L \geq&(e^t-1)(\kappa_i+\kappa_j)\sigma_{k-1}(\kappa|j)
-[(e^t-1)t+4t]\kappa_i\tilde\sigma_{k-2}\nonumber\\
=&[(e^t-1)(\kappa_i+\kappa_j)-(e^t+3)t]\sigma_{k-1}(\kappa|j)
+(e^t+3)t\tilde\sigma_{k-1}\\
\geq&[(e^t-1)(\kappa_i+\kappa_j)-(e^t+3)t]\frac{\sigma_{k-1}(\kappa|i)}{2}
+(e^t+3)t\tilde\sigma_{k-1}\\
\geq&[(e^t-1)(\kappa_i+\kappa_j)-(1+2\Theta)(e^t+3)t]\frac{\sigma_{k-1}(\kappa|i)}{2}\\
\geq&0,
\end{align*}
if $\kappa_1$ is sufficiently large. Here, in the fourth inequality, we  have
used Lemma \ref{lemm7} to give the lower bound of
$\tilde\sigma_{k-1}$,  and in the last inequality, we have used $e^t>1+t$,
$\kappa_j\geq\kappa_i\geq\kappa_1-\sqrt{\kappa_1}/n$, $t\leq
\sqrt{\kappa_1}/n$.

\end{proof}
Now, we are in the position to handle the Case (II).
\begin{lemm}\label{lemm15}
Assume $\kappa=(\kappa_1,\kappa_2,\cdots,\kappa_n)\in\Gamma_k$, $n<2k$ and $\kappa_1$ is the maximum
entry of $\kappa$. Then for any given index $1\leq i\leq n$, if
$\kappa_i>
 \kappa_1-\sqrt{\kappa_1}/n $, we have
\begin{align*}
A_i+B_i+C_i+D_i-E_i\geq 0,
\end{align*}
when the positive constant $K$ and  $\kappa_1$ both are sufficiently
large.
\end{lemm}
\begin{proof}
The proof is similar to Lemma 15 in \cite{RW}. We only need to show
that the following two inequalities hold,
\begin{align}\label{s3.05}
e^{\kappa_i}[K(\sigma_k(\kappa))_i^2-\sigma_k^{pp,qq}(\kappa)h_{ppi}h_{qqi}]+2\dsum_{l\neq
i}\dfrac{e^{\kappa_l}-e^{\kappa_i}}{\kappa_l-\kappa_i}\sigma_{k}^{ll}(\kappa)h_{lli}^2\geq\dfrac{1}{\log
P}e^{\kappa_i}\sigma_k^{ii}(\kappa)h_{iii}^2,
\end{align}
\begin{align}\label{s3.06}
2\dsum_{l\neq
i}e^{\kappa_l}\sigma_k^{ll,ii}(\kappa)h_{lli}^2-\dfrac{1}{\log
 P}\dsum_{l\neq i}e^{\kappa_l}\sigma_k^{ii}(\kappa)h_{lli}^2\geq 0,
\end{align}
which is corresponding to the inequalities (4.26) and (4.28) in
\cite{RW} respectively.

By the definition of $P$, we know $\log P>\kappa_1$.
Combing the inequality \eqref{s3.01} with Lemma \ref{lemm16}, we have the proof  of the inequality (\ref{s3.05}).

In order to prove the inequality \eqref{s3.06}, we only need to show the following inequality
$$2\kappa_1\sigma_{k}^{ii,ll}(\kappa)-\sigma_{k}^{ii}(\kappa)\geq 0.$$
It is clear that
\begin{align}\label{s3.21}
2\kappa_1\sigma_{k-2}(\kappa|il)-\sigma_{k-1}(\kappa|i)\geq
&2\kappa_1\sigma_{k-2}(\kappa|1i)-\sigma_{k-1}(\kappa|i)\\
=&2\kappa_1\sigma_{k-2}(\kappa|1i)-\kappa_1\sigma_{k-2}(\kappa|1i)-\sigma_{k-1}(\kappa|1i)\nonumber\\
=&\kappa_1\sigma_{k-2}(\kappa|1i)-\sigma_{k-1}(\kappa|1i).\nonumber
\end{align}
Thus, if $\sigma_{k-1}(\kappa|1i)\leq 0$,
(\ref{s3.06}) obviously
holds. If $\sigma_{k-1}(\kappa|1i)>0$, we have $(\kappa|1i)\in\Gamma_{k-1}$. Therefore, by Lemma \ref{lemm8}, we have
\begin{align}\label{s3.22}
\sigma_{k-1}(\kappa|1i)\leq
\dfrac{n-k}{k-1}\sigma_{k-2}(\kappa|1i)\kappa_1.
\end{align}
Thus if we have $n< 2k$ which implies $n\leq 2k-1$, combing (\ref{s3.21}) with
(\ref{s3.22}), we get
$$
2\kappa_1\sigma_{k-2}(\kappa|il)-\sigma_{k-1}(\kappa|i)\geq
\frac{2k-1-n}{k-1}\sigma_{k-2}(\kappa|1i)\kappa_1\geq 0,
$$
which gives the inequality (\ref{s3.06}).

\end{proof}

\section{An inequality}
In this section, we will establish an inequality that the left hand side of \eqref{s3.01} is bigger than the summation of four quadratic forms defined in the following.
The argument is similar to \cite{RW}, but will become
a little more complicated.

We define some notations. Let $\xi=(\xi_1,\xi_2,\cdots,\xi_{n})\in\mathbb{R}^n$ be
an $n$-dimensional vector. Suppose $1\leq i\leq n$ is some given index. We define four
quadratic forms, $\textbf{A}_{k;i},
\textbf{B}_{k;i}, \textbf{C}_{k;i},\textbf{D}_{k;i}$:
 \begin{eqnarray}
 \textbf{A}_{k;i}&=&\sum_{j\neq i}\sigma_{k-2}^2(\kappa|ij)\xi_j^2+\sum_{p\neq q; p,q\neq i}\left[\sigma_{k-2}^2(\kappa|ipq)
-\sigma_{k-1}(\kappa|ipq)\sigma_{k-3}(\kappa|ipq)\right]\xi_p\xi_q\nonumber;\\
\textbf{B}_{k;i}&=&\sum_{j\neq i}2\sigma_{k-2}(\kappa|ij)\xi_j^2-\sum_{p\neq q; p,q\neq i}\sigma_{k-2}(\kappa|ipq)\xi_p\xi_q;\nonumber
\end{eqnarray}
\begin{eqnarray}
\textbf{C}_{k;i}&=&\sum_{j\neq i}\left[\kappa_j^2\sigma_{k-2}^2(\kappa|ij)-2\sigma_{k}(\kappa|ij)\sigma_{k-2}(\kappa|ij)\right]\xi_j^2\nonumber\\
&&+\sum_{p,q\neq i, p\neq q}\left[\sigma_{k}(\kappa|ipq)\sigma_{k-2}(\kappa|ipq)-\sigma_{k-1}^2(\kappa|ipq)\right]\xi_p\xi_q;\nonumber\\
\textbf{D}_{k;i}&=&\sum_{j\neq
i}\sigma_{k-1}^2(\kappa|ij)\xi_j^2+\sum_{p\neq q; p,q\neq i,
}\sigma_{k-1}(\kappa|ip)\sigma_{k-1}(\kappa|iq)\xi_p\xi_q;\nonumber
 \end{eqnarray}
 We define a constant,
\begin{eqnarray}\label{ckK}
c_{k,K}=\frac{1}{K\kappa_i\sigma_{k-1}(\kappa|i)-1}.
\end{eqnarray}
It is easy to see that if $\kappa_i\geq
\kappa_1-\sqrt{\kappa_1}/n$, we have
$$\kappa_i\sigma_{k-1}(\kappa|i)\geq
\kappa_1\sigma_{k-1}(\kappa|1)/2\geq\theta\sigma_k(\kappa)/2>0,$$
when $\kappa_1$ is sufficiently large. Here $\theta$ is the constant
given in Lemma \ref{lemm11}. Therefore, $c_{k,K}$ is a positive
constant and can be very small if the constant $K$ is sufficiently
large. Thus, throughout the paper, we always assume $K$ is
sufficiently large and then $c_{k,K}$ is positive.

We mainly will prove the following lemma in this section.
\begin{lemm} \label{lemm17}
Assume $\kappa=(\kappa_1,\cdots,\kappa_n)\in\Gamma_{k}$,
$\kappa_1$ is the maximum entry of $\kappa$ and $\sigma_{k}(\kappa)$
has a positive lower bound $\sigma_{k}(\kappa)\geq N_0$. Then for
any given index $1\leq i\leq n$, if
$\kappa_i>\kappa_1-\sqrt{\kappa_1}/n$, for any $n$ dimensional vector $\xi=(\xi_1,\cdots,\xi_n)\in\mathbb{R}^n$, we have
\begin{eqnarray}\label{ex1}
&&\kappa_i\Big[K\Big(\sum_j\sigma_{k}^{jj}(\kappa)\xi_j\Big)^2-\sigma_{k}^{pp,qq}(\kappa)\xi_{p}\xi_{q}\Big]-\sigma^{ii}_{k}(\kappa)\xi_{i}^2+\sum_{j\neq
i}a_j\xi_{j}^2\\
&\geq&\frac{1}{c_{k,K}}\left[\textbf{A}_{k;i}+\sigma_{k}(\kappa)\textbf{B}_{k;i}+\textbf{C}_{k;i}-c_{k,K}\textbf{D}_{k;i}\right],\nonumber
\end{eqnarray}
when $\kappa_1$ and $K$ both are sufficiently large. Here $a_j$ and $c_{k,K}$ are defined by \eqref{aj} and \eqref{ckK}.
\end{lemm}
Before to prove the above lemma, we need some algebraic identities.
\begin{lemm}\label{lemm18}
Assume $\kappa=(\kappa_1,\cdots,\kappa_n)\in\Gamma_{k}$. Suppose $1\leq i,j,p,q\leq n$ are given indices.
$a_j$ and $c_{k,K}$ are defined by \eqref{aj} and \eqref{ckK}. We have the following five identities:

\noindent (1)
\begin{align*}
&\kappa_iK\sigma_{k}^{ii}(\kappa)\sigma_{k}^{jj}(\kappa)[-\sigma_{k}^{jj}(\kappa)+2\kappa_i\sigma_{k}^{ii,jj}(\kappa)]-\kappa_i^2[\sigma_{k}^{ii,jj}(\kappa)]^2
+a_j[\kappa_iK\sigma_{k}^{ii}(\kappa)-1]\sigma_{k}^{ii}(\kappa)\\
=&\frac{1}{c_{k,K}}[\sigma_{k}^{ii}(\kappa)+\sigma_{k}^{jj}(\kappa)](\kappa_i+\kappa_j)\sigma_{k-2}(\kappa|ij)-\sigma_{k-1}^2(\kappa|ij).
\end{align*}

\noindent (2)
\begin{align*}
&\kappa_i\left[\sigma_{k}^{pp}(\kappa)\sigma_{k}^{ii,qq}(\kappa)+\sigma_{k}^{qq}(\kappa)\sigma_{k}^{ii,pp}(\kappa)-\sigma_{k}^{ii}(\kappa)
\sigma_{k}^{pp,qq}(\kappa)\right]-\sigma_{k}^{pp}(\kappa)\sigma_{k}^{qq}(\kappa)\nonumber\\
&-\kappa_i^2\sigma_{k}^{ii,pp}(\kappa)\sigma_{k}^{ii,qq}(\kappa)+\kappa_i\sigma_{k}^{ii}(\kappa)\sigma_{k}^{pp,qq}(\kappa)\nonumber\\
=&-\sigma_{k-1}(\kappa|ip)\sigma_{k-1}(\kappa|iq).
\end{align*}

\noindent (3)
\begin{align*}
\left[\sigma_{k}^{ii}(\kappa)+\sigma_{k}^{jj}(\kappa)\right](\kappa_i+\kappa_j)
=&2\sigma_k(\kappa)-2\sigma_{k}(\kappa|ij)+\kappa_i^2\sigma_{k-2}(\kappa|ij)+\kappa_j^2\sigma_{k-2}(\kappa|ij).\nonumber
\end{align*}

\noindent (4)
\begin{align*}
\sigma_{k}^{qq}(\kappa)\sigma_{k}^{ii,pp}(\kappa)-\sigma_{k}^{ii}(\kappa)\sigma_{k}^{pp,qq}(\kappa)=&\kappa_i\sigma_{k-2}^2(\kappa|ipq)
-\kappa_i\sigma_{k-1}(\kappa|ipq)\sigma_{k-3}(\kappa|ipq)\nonumber\\
&+\kappa_q\sigma_{k-3}(\kappa|ipq)\sigma_{k-1}(\kappa|ipq)
-\kappa_q\sigma_{k-2}^2(\kappa|ipq).
\end{align*}

\noindent (5)
\begin{align*}
\sigma_{k}^{pp}(\kappa)\sigma_{k-1}(\kappa|iq)=&\sigma_k\sigma_{k-2}(\kappa|ipq)+\sigma_{k-1}^2(\kappa|ipq)-\sigma_{k}(\kappa|ipq)\sigma_{k-2}(\kappa|ipq)\nonumber\\
&-\kappa_q\kappa_i\sigma_{k-2}^2(\kappa|ipq)+\kappa_q\kappa_i\sigma_{k-3}(\kappa|ipq)\sigma_{k-1}(\kappa|ipq).
\end{align*}

\end{lemm}
\par
\begin{proof}For simplification purpose, we omit the $\kappa$ in our notations in the following argument, which means that we let
$$\sigma_k=\sigma_k(\kappa), \ \ \sigma_k^{pp}=\sigma_k^{pp}(\kappa), \ \ \sigma_k^{pp,qq}=\sigma_k^{pp,qq}(\kappa).$$

\noindent (1) Using the identity
\begin{align*}
-\sigma_{k}^{jj}+2\kappa_i\sigma_{k}^{ii,jj}+\sigma_{k}^{ii}=(\kappa_i+\kappa_j)\sigma_{k-2}(\kappa|ij),
\end{align*}
and $a_j=\sigma_{k}^{jj}+(\kappa_i+\kappa_j)\sigma_k^{ii,jj}$, we
have
\begin{align}
&\kappa_iK\sigma_{k}^{ii}\sigma_{k}^{jj}(-\sigma_{k}^{jj}+2\kappa_i\sigma_{k}^{ii,jj})-\kappa_i^2(\sigma_{k}^{ii,jj})^2
+a_j[\kappa_iK(\sigma_{k}^{ii})^2-\sigma_{k}^{ii}] \label{s4.02}\\
=&\kappa_iK\sigma_{k}^{ii}\sigma_{k}^{jj}(-\sigma_{k}^{jj}+2\kappa_i\sigma_{k}^{ii,jj}+\sigma_{k}^{ii})-\kappa_i^2(\sigma_{k}^{ii,jj})^2-\sigma_{k}^{ii}\sigma_{k}^{jj}\nonumber\\
&
+(\kappa_iK\sigma_{k}^{ii}-1)\sigma_{k}^{ii}(\kappa_i+\kappa_j)\sigma_k^{ii,jj} \nonumber\\
=&(\kappa_iK\sigma_{k}^{ii}-1)\sigma_{k}^{jj}(\kappa_i+\kappa_j)\sigma_{k-2}(\kappa|ij)+(\kappa_i+\kappa_j)\sigma_k^{jj}\sigma_{k-2}(\kappa|ij)\nonumber
\\ &
+(\kappa_iK\sigma_{k}^{ii}-1)\sigma_{k}^{ii}(\kappa_i+\kappa_j)\sigma_k^{ii,jj}-\kappa_i^2(\sigma_{k}^{ii,jj})^2-\sigma_{k}^{ii}\sigma_{k}^{jj} \nonumber \\
=&(\kappa_iK\sigma_{k}^{ii}-1)(\sigma_{k}^{ii}+\sigma_{k}^{jj})(\kappa_i+\kappa_j)\sigma_{k-2}(\kappa|ij)\nonumber\\
&+(\kappa_i+\kappa_j)\sigma_{k}^{jj}\sigma_{k-2}(\kappa|ij)-\kappa_i^2(\sigma_{k}^{ii,jj})^2-\sigma_{k}^{ii}\sigma_{k}^{jj}
\nonumber\\
=&(\kappa_iK\sigma_{k}^{ii}-1)(\sigma_{k}^{ii}+\sigma_{k}^{jj})(\kappa_i+\kappa_j)\sigma_{k-2}(\kappa|ij)-\sigma_{k-1}^2(\kappa|ij).\nonumber
\end{align}
Here in the above last equality we have used the following identity,
\begin{align*}
&(\kappa_i+\kappa_j)\sigma_{k}^{jj}\sigma_{k-2}(\kappa|ij)-\kappa_i^2(\sigma_{k}^{ii,jj})^2-\sigma_{k}^{ii}\sigma_{k}^{jj}\\
=&\kappa_i\sigma_{k-2}(\kappa|ij)[\sigma_{k-1}(\kappa|j)-\kappa_i\sigma_{k-2}(\kappa|ij)]+\sigma_{k}^{jj}[\kappa_j\sigma_{k-2}(\kappa|ij)-\sigma_{k-1}(\kappa|i)]\\
=&\kappa_i\sigma_{k-2}(\kappa|ij)\sigma_{k-1}(\kappa|ij)-\sigma_{k}^{jj}\sigma_{k-1}(\kappa|ij)\\
=&-\sigma_{k-1}^2(\kappa|ij).
\end{align*}
\par

\noindent (2) We have
\begin{align}
&\kappa_i(\sigma_{k}^{pp}\sigma_{k}^{ii,qq}+\sigma_{k}^{qq}\sigma_{k}^{ii,pp}-\sigma_{k}^{ii}\sigma_{k}^{pp,qq})-\sigma_{k}^{pp}\sigma_{k}^{qq}
-\kappa_i^2\sigma_{k}^{ii,pp}\sigma_{k}^{ii,qq}+\kappa_i\sigma_{k}^{ii}\sigma_{k}^{pp,qq}\nonumber\\
=&\kappa_i\sigma_{k}^{qq}\sigma_{k}^{ii,pp}-\sigma_{k}^{pp}\sigma_{k-1}(\kappa|iq)
-\kappa_i^2\sigma_{k}^{ii,pp}\sigma_{k}^{ii,qq}\nonumber\\
=&\kappa_i\sigma_{k-1}(\kappa|iq)\sigma_k^{ii,pp}-\sigma_{k}^{pp}\sigma_{k-1}(\kappa|iq)\nonumber\\
=&-\sigma_{k-1}(\kappa|ip)\sigma_{k-1}(\kappa|iq). \nonumber
\end{align}

\noindent (3) We have
\begin{align*}
(\sigma_{k}^{ii}+\sigma_{k}^{jj})(\kappa_i+\kappa_j)
=&\kappa_i\sigma_{k-1}(\kappa|i)+\kappa_j\sigma_{k-1}(\kappa|j)+\kappa_i\sigma_{k-1}(\kappa|j)+\kappa_j\sigma_{k-1}(\kappa|i)\\
=&2\sigma_k-\sigma_{k}(\kappa|i)-\sigma_{k}(\kappa|j)+\kappa_i^2\sigma_{k-2}(\kappa|ij)+\kappa_i\sigma_{k-1}(\kappa|ij)\\
&+\kappa_j^2\sigma_{k-2}(\kappa|ij)+\kappa_j\sigma_{k-1}(\kappa|ij)\\
=&2\sigma_k-2\sigma_{k}(\kappa|ij)+\kappa_i^2\sigma_{k-2}(\kappa|ij)+\kappa_j^2\sigma_{k-2}(\kappa|ij).
\end{align*}

\noindent (4) We further denote $\tilde\sigma_{m}=\sigma_{m}(\kappa|ipq)$ here. Thus, we have
\begin{align*}
&\sigma_{k}^{qq}\sigma_{k}^{ii,pp}-\sigma_{k}^{ii}\sigma_{k}^{pp,qq}\\
=&[\kappa_p\sigma_{k-2}(\kappa|pq)+\sigma_{k-1}(\kappa|pq)]\sigma_{k-2}(\kappa|ip)-\sigma_{k-2}(\kappa|pq)[\kappa_p\sigma_{k-2}(\kappa|ip)+\sigma_{k-1}(\kappa|ip)]\\
=&\sigma_{k-1}(\kappa|pq)\sigma_{k-2}(\kappa|ip)-\sigma_{k-2}(\kappa|pq)\sigma_{k-1}(\kappa|ip)\\
=&(\kappa_i\tilde\sigma_{k-2}+\tilde\sigma_{k-1})(\kappa_q\tilde\sigma_{k-3}+\tilde\sigma_{k-2})-(\kappa_i\tilde\sigma_{k-3}+\tilde\sigma_{k-2})(\kappa_q\tilde\sigma_{k-2}+\tilde\sigma_{k-1})\\
=&\kappa_i\tilde\sigma_{k-2}^2-\kappa_i\tilde\sigma_{k-1}\tilde\sigma_{k-3}+\kappa_q\tilde\sigma_{k-3}\tilde\sigma_{k-1}-\kappa_q\tilde\sigma_{k-2}^2.
\end{align*}

\noindent (5) We also denote $\tilde\sigma_{m}=\sigma_{m}(\kappa|ipq)$ here. We have
\begin{align*}
\sigma_{k}^{pp}\sigma_{k-1}(\kappa|iq)
=&\sigma_{k-1}(\kappa|p)[\kappa_p\sigma_{k-2}(\kappa|ipq)+\sigma_{k-1}(\kappa|ipq)]\\
=&\kappa_p\sigma_{k-1}(\kappa|p)\tilde\sigma_{k-2}+\sigma_{k-1}(\kappa|p)\tilde\sigma_{k-1}\\
=&[\sigma_k-\sigma_{k}(\kappa|p)]\tilde\sigma_{k-2}+[\kappa_q\kappa_i\tilde\sigma_{k-3}+(\kappa_q+\kappa_i)\tilde\sigma_{k-2}+\tilde\sigma_{k-1}]\tilde\sigma_{k-1}\\
=&\sigma_k\tilde\sigma_{k-2}-[\kappa_q\kappa_i\tilde\sigma_{k-2}+(\kappa_q+\kappa_i)\tilde\sigma_{k-1}+\tilde\sigma_{k}]
\tilde\sigma_{k-2}\\
&+[\kappa_q\kappa_i\tilde\sigma_{k-3}+(\kappa_q+\kappa_i)\tilde\sigma_{k-2}+\tilde\sigma_{k-1}]\tilde\sigma_{k-1}\\
=&\sigma_k\tilde\sigma_{k-2}+\tilde\sigma_{k-1}^2-\tilde\sigma_{k}\tilde\sigma_{k-2}-\kappa_q\kappa_i\tilde\sigma_{k-2}^2
+\kappa_q\kappa_i\tilde\sigma_{k-3}\tilde\sigma_{k-1}.
\end{align*}

\end{proof}

\noindent{\bf Proof of the Lemma \ref{lemm17}}:
For the sake of simplification, we still omit the $\kappa$ in the following calculation, which means that we still let
$$\sigma_k=\sigma_k(\kappa),\ \  \sigma_k^{pp}=\sigma_k^{pp}(\kappa), \ \ \sigma_k^{pp,qq}=\sigma_k^{pp,qq}(\kappa).$$
Let's calculate the left hand side of \eqref{s3.01}. By Lemma \ref{lemm11}, we have
$$K\kappa_i\sigma_{k}^{ii}-1\geq K\kappa_1\sigma_k^{11}/2-1\geq K\theta\sigma_{k}/2-1\geq 0,$$  if the positive constant $K$ and $\kappa_1$ both are sufficiently large.
A
straightforward calculation shows,
\begin{align}\label{s4.03}
&\kappa_i\Big[K\Big(\sum_{j}\sigma_{k}^{jj}\xi_j\Big)^2-\sigma_{k}^{pp,qq}\xi_p\xi_q\Big]-\sigma_{k}^{ii}\xi_i^2+\sum_{j\neq i}a_j\xi_j^2\\
=&\kappa_iK\Big(\sum_{j\neq i}\sigma_{k}^{jj}\xi_j\Big)^2+2\kappa_i\xi_i\Big[\sum_{j\neq i}(K\sigma_{k}^{ii}\sigma_{k}^{jj}-\sigma_{k}^{ii,jj})\xi_j\Big]\nonumber\\
&+[\kappa_iK(\sigma_{k}^{ii})^2-\sigma_{k}^{ii}]\xi_i^2+\sum_{j\neq i}a_j\xi_j^2-\kappa_i\sum_{p\neq i; q\neq i}\sigma_{k}^{pp,qq}\xi_p\xi_q\nonumber\\
\geq&\kappa_iK\Big(\sum_{j\neq i}\sigma_{k}^{jj}\xi_j\Big)^2-\frac{\kappa_i^2\Big[\sum_{j\neq i}(K\sigma_{k}^{ii}\sigma_{k}^{jj}-\sigma_{k}^{ii,jj})\xi_j\Big]^2}{\kappa_iK(\sigma_{k}^{ii})^2-\sigma_{k}^{ii}}\nonumber\\
&+\sum_{j\neq i}a_j\xi_j^2-\kappa_i\sum_{p\neq i; q\neq
i}\sigma_{k}^{pp,qq}\xi_p\xi_q\nonumber\\
=&\sum_{j\neq
i}\left[\kappa_iK(\sigma_{k}^{jj})^2-\frac{\kappa_i^2(K\sigma_{k}^{ii}\sigma_{k}^{jj}-\sigma_{k}^{ii,jj})^2}{\kappa_iK(\sigma_{k}^{ii})^2-\sigma_{k}^{ii}}+a_j\right]\xi_j^2\nonumber\\
&+\sum_{p,q\neq i;p\neq
q}\Big[\kappa_iK\sigma_{k}^{pp}\sigma_{k}^{qq}-\frac{\kappa_i^2(K\sigma_{k}^{ii}\sigma_{k}^{pp}
-\sigma_{k}^{ii,pp})(K\sigma_{k}^{ii}\sigma_{k}^{qq}-\sigma_{k}^{ii,qq})}{\kappa_iK(\sigma_{k}^{ii})^2
-\sigma_{k}^{ii}}-\kappa_i\sigma_{k}^{pp,qq}\Big]\xi_p\xi_q\nonumber,
\end{align}
where, in the second inequality, we have used,
\begin{align*}
&\frac{\kappa_i^2\Big[\sum_{j\neq
i}(K\sigma_{k}^{ii}\sigma_{k}^{jj}-\sigma_{k}^{ii,jj})\xi_j\Big]^2}{\kappa_iK(\sigma_{k}^{ii})^2-\sigma_{k}^{ii}}+2\kappa_i\xi_i\Big[\sum_{j\neq
i}(K\sigma_{k}^{ii}\sigma_{k}^{jj}-\sigma_{k}^{ii,jj})\xi_j\Big]\nonumber\\
&+[\kappa_iK(\sigma_{k}^{ii})^2-\sigma_{k}^{ii}]\xi_i^2\geq 0.
\nonumber
\end{align*}

Thus, we can multiple the term
$\kappa_iK(\sigma_{k}^{ii})^2-\sigma_{k}^{ii}$ in both sides of
\eqref{s4.03}. Then, we get
\begin{align}\label{s4.04}
&[\kappa_iK(\sigma_{k}^{ii})^2-\sigma_{k}^{ii}]\Big\{\kappa_i\Big[K\Big(\sum_{j}\sigma_{k}^{jj}\xi_j\Big)^2-\sigma_{k}^{pp,qq}\xi_p\xi_q\Big]-\sigma_{k}^{ii}\xi_i^2
+\sum_{j\neq i}a_j\xi_j^2\Big\}\\
\geq&\sum_{j\neq
i}\Big[\kappa_iK\sigma_{k}^{ii}\sigma_{k}^{jj}(-\sigma_{k}^{jj}+2\kappa_i\sigma_{k}^{ii,jj})-\kappa_i^2(\sigma_{k}^{ii,jj})^2
+a_j(\kappa_iK\sigma_{k}^{ii}-1)\sigma_{k}^{ii}\Big]\xi_j^2\nonumber\\
&+\sum_{p,q\neq i;p\neq
q}\Big[\kappa_iK\sigma_{k}^{ii}[\kappa_i(\sigma_{k}^{pp}\sigma_{k}^{ii,qq}+\sigma_{k}^{qq}\sigma_{k}^{ii,pp}-\sigma_{k}^{ii}\sigma_{k}^{pp,qq})-\sigma_{k}^{pp}\sigma_{k}^{qq}]\nonumber\\
&-\kappa_i^2\sigma_{k}^{ii,pp}\sigma_{k}^{ii,qq}+\kappa_i\sigma_{k}^{ii}\sigma_{k}^{pp,qq}\Big]\xi_p\xi_q\nonumber
\end{align}
\begin{align*}
=&\dsum_{j\neq
i}\Big[(\kappa_iK\sigma_{k}^{ii}-1)(\sigma_{k}^{ii}+\sigma_{k}^{jj})(\kappa_i+\kappa_j)\sigma_{k-2}(\kappa|ij)-\sigma_{k-1}^2(\kappa|ij)\Big]\xi_j^2\nonumber\\
&+\sum_{p,q\neq i;p\neq
q}\Big[(\kappa_iK\sigma_{k}^{ii}-1)[\kappa_i(\sigma_{k}^{pp}\sigma_{k}^{ii,qq}+\sigma_{k}^{qq}\sigma_{k}^{ii,pp}-\sigma_{k}^{ii}\sigma_{k}^{pp,qq})-\sigma_{k}^{pp}\sigma_{k}^{qq}]\nonumber\\
&-\sigma_{k-1}(\kappa|ip)\sigma_{k-1}(\kappa|iq)\Big]\xi_p\xi_q\nonumber\\
= &\dsum_{j\neq
i}\Big[(\kappa_iK\sigma_{k}^{ii}-1)(\sigma_{k}^{ii}+\sigma_{k}^{jj})(\kappa_i+\kappa_j)\sigma_{k-2}(\kappa|ij)-\sigma_{k-1}^2(\kappa|ij)\Big]\xi_j^2\nonumber\\
&+\sum_{p,q\neq i;p\neq
q}\Big[(\kappa_iK\sigma_{k}^{ii}-1)(\kappa_i\sigma_{k}^{qq}\sigma_{k}^{ii,pp}-\kappa_i\sigma_{k}^{ii}\sigma_{k}^{pp,qq}-\sigma_{k}^{pp}\sigma_{k-1}(\kappa|iq))\nonumber\\
&-\sigma_{k-1}(\kappa|ip)\sigma_{k-1}(\kappa|iq)\Big]\xi_p\xi_q\nonumber\\
=&(\kappa_iK\sigma_{k}^{ii}-1)\dsum_{j\neq
i}\Big[\kappa_i^2\sigma_{k-2}^2(\kappa|ij)+\kappa_j^2\sigma_{k-2}^2(\kappa|ij)-2\sigma_{k}(\kappa|ij)\sigma_{k-2}(\kappa|ij)\nonumber\\
&+2\sigma_{k}\sigma_{k-2}(\kappa|ij)
-c_{k,K}\sigma_{k-1}^2(\kappa|ij)\Big]\xi_j^2\nonumber\\
&+(\kappa_iK\sigma_{k}^{ii}-1)\sum_{p,q\neq i;p\neq
q}\Big[\kappa_i^2\sigma_{k-2}^2(\kappa|ipq)
-\kappa_i^2\sigma_{k-1}(\kappa|ipq)\sigma_{k-3}(\kappa|ipq)\nonumber\\
&-\sigma_k\sigma_{k-2}(\kappa|ipq) +\sigma_{k}(\kappa|ipq)\sigma_{k-2}(\kappa|ipq)-\sigma_{k-1}^2(\kappa|ipq)\nonumber\\
&-c_{k,K}\sigma_{k-1}(\kappa|ip)\sigma_{k-1}(\kappa|iq)\Big]\xi_p\xi_q\nonumber\\
=&\frac{1}{c_{k,K}}\left[\kappa_i^2\textbf{A}_{k;i}+\sigma_{k}\textbf{B}_{k;i}+\textbf{C}_{k;i}-c_{k,K}\textbf{D}_{k;i}\right],
\end{align*}
where in the second equality, we have used identities (1),(2) of Lemma \ref{lemm18} and in the forth equality, we have used identities (3),(4),(5) of  Lemma \ref{lemm18}.  We have completed our proof.
\bigskip

Now, we stat the main result of the following several sections.

\begin{theo}\label{maintheo} Conjecture \ref{con} holds when
$k=n-2$ and $n\geq 5$.
\end{theo}

The proof of the above theorem is very complicated. Thus, we first give an outline. For convenience, we let $\kappa_1\geq\cdots\geq\kappa_n$.
We divide into five cases to verify the above Theorem \ref{maintheo}
in section 6 and 7. These cases are

\bigskip

\noindent Case A: Suppose $\sigma_{n-2}(\kappa|i)\leq 0$ and
$\kappa_n\leq\kappa_{n-1}\leq 0$.

\par

\noindent  Case B: Suppose $\sigma_{n-2}(\kappa|i)\leq 0$ and $\kappa_n< 0,
\kappa_{n-1}>0$. This case needs to be divided into three sub-cases:

\par
Case B1: Suppose
$$\kappa_i\sigma_{n-3}(\kappa|i)\geq(1+\delta_0)\sigma_{n-2}(\kappa)\quad {\rm with} \quad \delta_0=\dfrac{1}{32n(n-2)}.$$

\par
Case B2: Suppose $\kappa_1\cdots\kappa_{n-2}\geq
2(n-2)\sigma_{n-2}(\kappa)$.

Case B3: For the given constant $\delta_0$ in Case B1, we suppose
$$\kappa_i\sigma_{n-3}(\kappa|i)\leq
(1+\delta_0)\sigma_{n-2}(\kappa)\text{ and } \kappa_1\cdots\kappa_{n-2}<
2(n-2)\sigma_{n-2}(\kappa).$$

\noindent Case C: Suppose $\sigma_{n-2}(\kappa|i)\geq 0$.

\par
In section 6, for the given index $i$, if
$\kappa_i\geq\kappa_1-\sqrt{\kappa_1}/n$, we will prove the
quadratic form
\begin{align}\label{s4.05}
\kappa_i^2
\textbf{A}_{n-2;i}+\sigma_{n-2}(\kappa)\textbf{B}_{n-2;i}+\textbf{C}_{n-2;i}-c_{n-2,K}\textbf{D}_{n-2;i}
\end{align}
is non-negative for the cases A, B1 and B2, when $\kappa_1$ is sufficiently large. Using Lemma \ref{lemm17}, we obtain the proof
of Theorem \ref{maintheo} for same cases. Note that the
conditions appearing in sub-cases B1 and B2 are only used  in Lemma
\ref{lemm30}. In section 7, we will give a straightforward proof of  Theorem \ref{maintheo}
for the cases B3 and C.

\section{More algebraic lemmas}
In this section, we list more algebraic results about $\sigma_k$ which will be used in section 6. For a $m$ dimensional vector $\kappa=(\kappa_1,\cdots,\kappa_m)$, in this section, we always denote $\kappa^2=(\kappa_1^2,\cdots,\kappa_m^2)$ which is also a $m$ dimensional vector.
\begin{lemm}\label{lemm20}
Assume $\kappa=(\kappa_1,\cdots,\kappa_m)\in \mathbb{R}^m$. For any  $0<k\leq m$, we have
\begin{align}\label{a5.1}
\sigma_{k}^2(\kappa)=\sigma_{k}(\kappa^2)+2\dsum_{i=1}^{k}(-1)^{i+1}\sigma_{k+i}(\kappa)\sigma_{k-i}(\kappa),
\end{align}
where by
(\ref{a2.1}), we have $\sigma_{k}(\kappa^2)=\dsum_{1\leq j_1<\cdots<j_k\leq
m}\kappa_{j_1}^2\cdots \kappa_{j_k}^2$.
\end{lemm}
\begin{proof}
We will prove \eqref{a5.1} by induction of the dimension $m$.
At first, for $m=1$, it is obvious. Now
we assume that our Lemma is true for $m-1$, namely that for any $0<k\leq m-1$, \eqref{a5.1} holds. Let's prove our Lemma for $m$. If $k=m$, \eqref{a5.1} is clear. Thus, we only need to consider the proof for $k<m$ in the following.

For a $m$ dimensional vector $\kappa$, we clearly have
\begin{align}\label{s5.04}
\sigma_k^2(\kappa)=&[\kappa_1\sigma_{k-1}(\kappa|1)+\sigma_k(\kappa|1)]^2\\
=&\kappa_1^2\sigma_{k-1}^2(\kappa|1)+2\kappa_1\sigma_{k-1}(\kappa|1)\sigma_k(\kappa|1)+\sigma_k^2(\kappa|1)\nonumber\\
=&\kappa_1^2\Big[\sigma_{k-1}(\kappa^2|1)+2\dsum_{i=1}^{k-1}(-1)^{i+1}\sigma_{k-1+i}(\kappa|1)\sigma_{k-1-i}(\kappa|1)\Big]\nonumber\\
&+2\sigma_{k-1}(\kappa|1)[\sigma_{k+1}(\kappa)-\sigma_{k+1}(\kappa|1)]+\sigma_{k}^2(\kappa|1)\nonumber\\
=&\kappa_1^2\sigma_{k-1}(\kappa^2|1)+2\kappa_1\dsum_{i=1}^{k-1}(-1)^{i+1}[\sigma_{k+i}(\kappa)-\sigma_{k+i}(\kappa|1)]\sigma_{k-1-i}(\kappa|1)\nonumber\\
&+2\sigma_{k-1}(\kappa|1)\sigma_{k+1}(\kappa)-2\sigma_{k-1}(\kappa|1)\sigma_{k+1}(\kappa|1)+\sigma_{k}^2(\kappa|1),\nonumber
\end{align}
where we have used the inductive assumption that \eqref{a5.1} holds for $k-1$ and the $m-1$ dimensional vector $(\kappa|1)$ in the third equality.
A straightforward calculation shows
\begin{align}\label{s5.05}
&\kappa_1\dsum_{i=1}^{k-1}(-1)^{i+1}[\sigma_{k+i}(\kappa)-\sigma_{k+i}(\kappa|1)]\sigma_{k-1-i}(\kappa|1)\\
=&\dsum_{i=1}^{k-1}(-1)^{i+1}\sigma_{k+i}(\kappa)[\sigma_{k-i}(\kappa)-\sigma_{k-i}(\kappa|1)]\nonumber\\
&-\dsum_{i=1}^{k-1}(-1)^{i+1}[\sigma_{k+1+i}(\kappa)-\sigma_{k+1+i}(\kappa|1)]\sigma_{k-1-i}(\kappa|1)\nonumber\\
=&\dsum_{i=1}^{k-1}(-1)^{i+1}\sigma_{k+i}(\kappa)\sigma_{k-i}(\kappa)-\dsum_{i=1}^{k-1}(-1)^{i+1}\sigma_{k+i}(\kappa)\sigma_{k-i}(\kappa|1)\nonumber\\
&-\dsum_{i=1}^{k-1}(-1)^{i+1}\sigma_{k+1+i}(\kappa)\sigma_{k-1-i}(\kappa|1)+\dsum_{i=1}^{k-1}(-1)^{i+1}\sigma_{k+1+i}(\kappa|1)\sigma_{k-1-i}(\kappa|1).\nonumber
\end{align}
It is clear that
\begin{align}\label{s5.06}
&-\sum_{i=1}^{k-1}(-1)^{i+1}\sigma_{k+i}(\kappa)\sigma_{k-i}(\kappa|1)-\dsum_{i=1}^{k-1}(-1)^{i+1}\sigma_{k+1+i}(\kappa)\sigma_{k-1-i}(\kappa|1)
\end{align}
\begin{align}
=&-\sigma_{k+1}(\kappa)\sigma_{k-1}(\kappa|1)-(-1)^k\sigma_{2k}(\kappa).\nonumber
\end{align}
Inserting (\ref{s5.05}), (\ref{s5.06}) into (\ref{s5.04}), we obtain
\begin{align}\label{s5.07}
\sigma_k^2(\kappa)=&\kappa_1^2\sigma_{k-1}(\kappa^2|1)+2\sum_{i=1}^{k-1}(-1)^{i+1}\sigma_{k+i}(\kappa)\sigma_{k-i}(\kappa)-2(-1)^k\sigma_{2k}(\kappa)\\
&+\sigma_{k}^2(\kappa|1)+2\sum_{i=1}^{k-1}(-1)^{i+1}\sigma_{k+1+i}(\kappa|1)\sigma_{k-1-i}(\kappa|1)-2\sigma_{k-1}(\kappa|1)\sigma_{k+1}(\kappa|1)\nonumber\\
=&\kappa_1^2\sigma_{k-1}(\kappa^2|1)+2\dsum_{i=1}^{k}(-1)^{i+1}\sigma_{k+i}(\kappa)\sigma_{k-i}(\kappa)\nonumber\\
&+\sigma_{k}^2(\kappa|1)-2\sum_{i=1}^{k}(-1)^{i+1}\sigma_{k+i}(\kappa|1)\sigma_{k-i}(\kappa|1)\nonumber\\
=&\kappa_1^2\sigma_{k-1}(\kappa^2|1)+2\dsum_{i=1}^{k}(-1)^{i+1}\sigma_{k+i}(\kappa)\sigma_{k-i}(\kappa)+\sigma_{k}(\kappa^2|1)\nonumber\\
=&\sigma_{k}(\kappa^2)+2\dsum_{i=1}^{k}(-1)^{i+1}\sigma_{k+i}(\kappa)\sigma_{k-i}(\kappa),\nonumber
\end{align}
where we have used inductive assumption \eqref{a5.1} for $k$ and the $m-1$ dimensional vector $(\kappa|1)$ in the third equality.
\end{proof}

\begin{lemm}\label{lemm21}
Assume $\kappa=(\kappa_1,\cdots,\kappa_{m})\in\bar{\Gamma}_m$. For
any vector $\xi=(\xi_1,\cdots,\xi_{m})\in \mathbb{R}^{m}$ and any
 $0\leq s\leq m$, the quadratic form
$$\sum_j\sigma_{s}(\kappa|j)\xi_j^2+\sum_{p\neq q}\sigma_{s}(\kappa|pq)\xi_p\xi_q$$ is
nonnegative.

\end{lemm}

\begin{proof}
Using (\ref{a2.1}),  we have
\begin{align*}
&\dsum_{j}\sigma_s(\kappa|j)\xi_j^2+\dsum_{p\neq
q}\sigma_s(\kappa|pq)\xi_p\xi_q\\
=&\dsum_{j}\xi_j^2\dsum_{\substack{j_1,\cdots,j_s\neq j\\
j_1<j_2<\cdots<j_{s}}}\kappa_{j_1}\kappa_{j_2}\cdots\kappa_{j_{s}}+\dsum_{p\neq
q}\xi_p\xi_q\dsum_{\substack{j_1,\cdots,j_s\neq p,q\\
j_1<j_2<\cdots<j_{s}}}\kappa_{j_1}\kappa_{j_2}\cdots\kappa_{j_{s}}\\
=&\dsum_{j_1<j_2<\cdots<j_{s}}\kappa_{j_1}\kappa_{j_2}\cdots\kappa_{j_{s}}\Big(\dsum_{j\neq
j_1,\cdots,j_s}\xi_j^2+\dsum_{\substack{p\neq q\\  p,q\neq
j_1,\cdots,j_s}}\xi_p\xi_q\Big)\\
=&\dsum_{j_1<j_2<\cdots<j_{s}}\kappa_{j_1}\kappa_{j_2}\cdots\kappa_{j_{s}}\sigma_1^2(\xi|j_1,j_2,\cdots
j_{s}),
\end{align*}
which is clearly nonnegative.
\end{proof}

\begin{lemm}\label{lemm21b}
Assume $\kappa=(\kappa_1,\cdots,\kappa_{m})\in\bar{\Gamma}_m$. For
any vector $\xi=(\xi_1,\cdots,\xi_{m})\in \mathbb{R}^{m}$, the
quadratic form
\begin{eqnarray}\label{newkappa}
\sum_{j}2\sigma_{m-3}(\kappa|j)\xi_j^2-\sum_{p\neq q}\sigma_{m-3}(\kappa|pq)\xi_p\xi_q
\end{eqnarray}
is
nonnegative.

\end{lemm}

\begin{proof}
It is clear that
\begin{align*}
&\sum_{j}2\sigma_{m-3}(\kappa|j)\xi_j^2-\sum_{p\neq
q}\sigma_{m-3}(\kappa|pq)\xi_p\xi_q\\
=&\sum_{j}\Big(\dsum_{s\neq
j}\sigma_{m-3}(\kappa|js)\Big)\xi_j^2-\sum_{p\neq
q}\Big(\dsum_{s\neq p,q}\sigma_{m-3}(\kappa|pqs)\Big)\xi_p\xi_q\\
=&\sum_{s}\Big(\dsum_{j\neq
s}\sigma_{m-3}(\kappa|js)\xi_j^2-\sum_{p\neq q;p,q\neq
s}\sigma_{m-3}(\kappa|pqs)\xi_p\xi_q\Big).
\end{align*}
If we have  $\kappa\in\bar\Gamma_m$, we obviously have $\kappa\in
\Gamma_{m-1}$. Thus, using Lemma 10 in \cite{RW}, the above
quadratic form is nonnegative.

\end{proof}

\begin{lemm}\label{lemm22}
Assume $\kappa=(\kappa_1,\cdots,\kappa_m)\in \mathbb{R}^m$. Then for
any $1\leq s\leq m$, we have
\begin{align*}
&\dsum_{i=1}^m
\sigma_{m-s}(\kappa|i)\sigma_{m-1}(\kappa|i)=\sigma_{m-s}(\kappa)\sigma_{m-1}(\kappa)-(s+1)\sigma_m(\kappa)\sigma_{m-(s+1)}(\kappa).
\end{align*}
\end{lemm}
\begin{proof}
By (iv) in section 2, we have
\begin{align}\label{s5.08}
s\sigma_{m-s}(\kappa)\sigma_{m-1}(\kappa)=&\Big(\dsum_{i=1}^m\sigma_{m-s}(\kappa|i)\Big)\Big(\dsum_{j=1}^m\sigma_{m-1}(\kappa|j)\Big)\\
=&\dsum_{i=1}^m \sigma_{m-s}(\kappa|i)\sigma_{m-1}(\kappa|i)+\dsum_{i\neq
j} \sigma_{m-s}(\kappa|i)\sigma_{m-1}(\kappa|j),\nonumber
\end{align}
and
\begin{align}\label{s5.09}
\dsum_{i\neq j} \sigma_{m-s}(\kappa|i)\sigma_{m-1}(\kappa|j)=&\dsum_{i\neq
j}
[\kappa_j\sigma_{m-(s+1)}(\kappa|ij)+\sigma_{m-s}(\kappa|ij)]\sigma_{m-1}(\kappa|j)\\
=&\sigma_m(\kappa)\dsum_{i\neq j} \sigma_{m-(s+1)}(\kappa|ij)+\dsum_{i\neq
j}\sigma_{m-s}(\kappa|ij)\sigma_{m-1}(\kappa|j)\nonumber\\
=&s(s+1)\sigma_m(\kappa)\sigma_{m-(s+1)}(\kappa)+(s-1)\dsum_{i=1}^m
\sigma_{m-s}(\kappa|i)\sigma_{m-1}(\kappa|i),\nonumber
\end{align}
inserting (\ref{s5.09}) into (\ref{s5.08}), we obtain our lemma.

\end{proof}

\begin{lemm}\label{lemm23}
Assume $\kappa=(\kappa_1,\cdots,\kappa_m)\in \mathbb{R}^m$. For any
given indices $1\leq j\leq m$, we have
\begin{align*}
 &\dsum_{s\neq j}\sigma_{m-4}^2(\kappa|js)\\
 =&3\sigma_{m-4}^2(\kappa|j)-2\sigma_{m-5}(\kappa|j)\sigma_{m-3}(\kappa|j)-4\sigma_{m-6}(\kappa|j)\sigma_{m-2}(\kappa|j)-6\sigma_{m-7}(\kappa|j)\sigma_{m-1}(\kappa|j).
\end{align*}
\end{lemm}

\begin{proof}
Since we have $\sigma_{m-4}(\kappa|js)=\sigma_{m-4}(\kappa|j)-\kappa_s\sigma_{m-5}(\kappa|js)$,
we have
\begin{align}\label{s5.10}
 &\dsum_{s\neq j}\sigma_{m-4}^2(\kappa|js)\\
 =&\dsum_{s\neq
 j}[\sigma_{m-4}^2(\kappa|j)-2\kappa_s\sigma_{m-5}(\kappa|js)\sigma_{m-4}(\kappa|j)+\kappa_s^2\sigma_{m-5}^2(\kappa|js)]\nonumber\\
 =&(m-1)\sigma_{m-4}^2(\kappa|j)-2\sigma_{m-4}(\kappa|j)\dsum_{s\neq
j}\kappa_s\sigma_{m-5}(\kappa|js)+\dsum_{s\neq
j}\kappa_s^2\sigma_{m-5}^2(\kappa|js).\nonumber
\end{align}
Using (v) and (iv) in section 2, we have the following four identities,
\begin{align}\label{s5.11}
 \dsum_{s\neq
 j}\kappa_s\sigma_{m-5}(\kappa|js)=(m-4)\sigma_{m-4}(\kappa|j),
\end{align}

\begin{align}\label{s5.12}
 &\dsum_{s\neq
 j}\kappa_s^2\sigma_{m-5}^2(\kappa|js)\\
=&\dsum_{s\neq
 j}\kappa_s\sigma_{m-5}(\kappa|js)[\sigma_{m-4}(\kappa|j)-\sigma_{m-4}(\kappa|js)]\nonumber\\
=&\sigma_{m-4}(\kappa|j)\dsum_{s\neq
 j}\kappa_s\sigma_{m-5}(\kappa|js)-\dsum_{s\neq
 j}\kappa_s\sigma_{m-5}(\kappa|js)\sigma_{m-4}(\kappa|js)\nonumber\\
=&(m-4)\sigma_{m-4}^2(\kappa|j)-\dsum_{s\neq
 j}\sigma_{m-5}(\kappa|js)[\sigma_{m-3}(\kappa|j)-\sigma_{m-3}(\kappa|js)]\nonumber\\
 =&(m-4)\sigma_{m-4}^2(\kappa|j)-4\sigma_{m-3}(\kappa|j)\sigma_{m-5}(\kappa|j)+\dsum_{s\neq
 j}\sigma_{m-5}(\kappa|js)\sigma_{m-3}(\kappa|js),\nonumber
\end{align}

\begin{align}\label{s5.13}
&\dsum_{s\neq
 j}\sigma_{m-5}(\kappa|js)\sigma_{m-3}(\kappa|js)\\
 =& \dsum_{s\neq
 j}[\sigma_{m-5}(\kappa|j)-\kappa_s\sigma_{m-6}(\kappa|js)]\sigma_{m-3}(\kappa|js)\nonumber\\
 =&\sigma_{m-5}(\kappa|j)\dsum_{s\neq
 j}\sigma_{m-3}(\kappa|js)-\dsum_{s\neq
 j}\kappa_s\sigma_{m-6}(\kappa|js)\sigma_{m-3}(\kappa|js)\nonumber\\
 =&2\sigma_{m-5}(\kappa|j)\sigma_{m-3}(\kappa|j) -\dsum_{s\neq
 j}\sigma_{m-6}(\kappa|js)[\sigma_{m-2}(\kappa|j)-\sigma_{m-2}(\kappa|js)]\nonumber\\
 =&2\sigma_{m-5}(\kappa|j)\sigma_{m-3}(\kappa|j)-5\sigma_{m-6}(\kappa|j)\sigma_{m-2}(\kappa|j)+\dsum_{s\neq
 j}\sigma_{m-6}(\kappa|js)\sigma_{m-2}(\kappa|js),\nonumber
\end{align}

\begin{align}\label{s5.14}
&\dsum_{s\neq
 j}\sigma_{m-6}(\kappa|js)\sigma_{m-2}(\kappa|js)\\
 =& \dsum_{s\neq
 j}[\sigma_{m-6}(\kappa|j)-\kappa_s\sigma_{m-7}(\kappa|js)]\sigma_{m-2}(\kappa|js)\nonumber\\
 =&\sigma_{m-6}(\kappa|j)\dsum_{s\neq
 j}\sigma_{m-2}(\kappa|js)-\dsum_{s\neq
 j}\kappa_s\sigma_{m-7}(\kappa|js)\sigma_{m-2}(\kappa|js)\nonumber\\
 =&\sigma_{m-6}(\kappa|j)\sigma_{m-2}(\kappa|j) -\dsum_{s\neq
 j}\sigma_{m-7}(\kappa|js)\sigma_{m-1}(\kappa|j)\nonumber\\
 =&\sigma_{m-6}(\kappa|j)\sigma_{m-2}(\kappa|j)-6\sigma_{m-7}(\kappa|j)\sigma_{m-1}(\kappa|j).\nonumber
\end{align}
Inserting (\ref{s5.11})-(\ref{s5.14}) into (\ref{s5.10}), we obtain our lemma.
\end{proof}

\begin{lemm}\label{lemm24}
Assume $\kappa=(\kappa_1,\cdots,\kappa_{m})\in \Gamma_{s}$ with $1\leq s\leq m$. For any $m$ dimensional vector
$\xi=(\xi_1,\xi_2,\cdots,\xi_m)\in \mathbb{R}^m$, the following
quadratic form
\begin{align*}
\dsum_{j} \sigma_{s-1}^2(\kappa|j)\xi_j^2+\dsum_{p\neq
q}[\sigma_{s-1}^2(\kappa| pq) -\sigma_{s}(\kappa|
pq)\sigma_{s-2}(\kappa| pq)]\xi_p\xi_q
\end{align*}
 is nonnegative.
\end{lemm}

\begin{proof}
Let $k=s, l=1$, $\delta=1$ and $w_{iih}=\xi_i$ in (\ref{s2.03}), then we
get
\begin{align}\label{1313}
\sum_{p,q}[\sigma_{s-1}(\kappa|p)\sigma_{s-1}(\kappa|q)-\sigma_{s}(\kappa)\sigma_{s}^{pp,qq}(\kappa)]\xi_p\xi_q\geq
0.
\end{align}
If $p=q=j$, we have
$$
\sigma_{s-1}(\kappa|p)\sigma_{s-1}(\kappa|q)-\sigma_{s}(\kappa)\sigma_{s}^{pp,qq}(\kappa)=\sigma_{s-1}^2(\kappa|j).
$$
If $p\neq q$, we have
$$
\sigma_{s-1}(\kappa|p)\sigma_{s-1}(\kappa|q)-\sigma_{s}(\kappa)\sigma_{s}^{pp,qq}(\kappa)=\sigma_{s-1}^2(\kappa|pq)-\sigma_{s}(\kappa|pq)\sigma_{s-2}(\kappa|pq).
$$
Using the above two identities and \eqref{1313}, we obtain our lemma.
\end{proof}

\begin{lemm}\label{lemm25}
Assume $\kappa=(\kappa_1,\cdots,\kappa_m)\in \Gamma_{m-2}$. For any
$\xi=(\xi_1,\cdots,\xi_{m})\in \mathbb{R}^{m}$, the following
quadratic form
\begin{align}\label{s5.15}
\sum_j 2\sigma_{m-3}(\kappa|j)\xi_j^2 -\sum_{p\neq
q}\sigma_{m-3}(\kappa|pq)\xi_p\xi_q
\end{align}
 is nonnegative.
\end{lemm}

\begin{proof}
Denote the quadratic form (\ref{s5.15}) by $L$. Assume
$\kappa_1\geq\cdots\geq\kappa_m$. If $\kappa_m\geq 0$, by Lemma
\ref{lemm21b}, we know $L\geq 0$. Thus, in the following, we will assume
$\kappa_m<0$.

We can rewrite $L$ to be
\begin{align}\label{a5.14}
L=&2\dsum_{j\neq m}\sigma_{m-3}(\kappa|j)\xi_j^2+2\sigma_{m-3}(\kappa|m)\xi_m^2\\
&-\dsum_{p\neq q;p,q\neq
m}\sigma_{m-3}(\kappa|pq)\xi_p\xi_q-2\dsum_{j\neq
m}\sigma_{m-3}(\kappa|jm)\xi_m\xi_j.\nonumber
\end{align}
If the indices $j,p,q\neq m$, using the identity
$\kappa_m=\dfrac{\sigma_{m-2}(\kappa)-\sigma_{m-2}(\kappa|m)}{\sigma_{m-3}(\kappa|m)}$,
we have
\begin{align*}
\sigma_{m-3}(\kappa|j)=&\kappa_m\sigma_{m-4}(\kappa|jm)+\sigma_{m-3}(\kappa|jm)\\
=&\dfrac{\sigma_{m-2}(\kappa)-\sigma_{m-2}(\kappa|m)}{\sigma_{m-3}(\kappa|m)}\sigma_{m-4}(\kappa|jm)+\sigma_{m-3}(\kappa|jm),
\end{align*}
and
\begin{align*}
\sigma_{m-3}(\kappa|pq)=&\kappa_m\sigma_{m-4}(\kappa|pqm)+\sigma_{m-3}(\kappa|pqm)\\
=&\dfrac{\sigma_{m-2}(\kappa)-\sigma_{m-2}(\kappa|m)}{\sigma_{m-3}(\kappa|m)}\sigma_{m-4}(\kappa|pqm)+\sigma_{m-3}(\kappa|pqm).
\end{align*}
Inserting the above two identities into (\ref{a5.14}), we get
\begin{align}\label{s5.16}
&\sigma_{m-3}(\kappa|m)\cdot L\\
=&\sigma_{m-2}(\kappa)\Big[\sum_{j\neq m} 2\sigma_{m-4}(\kappa|jm)\xi_j^2-\sum_{p\neq q;p,q\neq m}\sigma_{m-4}(\kappa|pqm)\xi_p\xi_q\Big]\nonumber\\
&+2\dsum_{j\neq
m}[-\sigma_{m-2}(\kappa|m)\sigma_{m-4}(\kappa|jm)+\sigma_{m-3}(\kappa|jm)\sigma_{m-3}(\kappa|m)]\xi_j^2 \nonumber\\
&-\sum_{p\neq q;p,q\neq
m}[-\sigma_{m-2}(\kappa|m)\sigma_{m-4}(\kappa|pqm)+\sigma_{m-3}(\kappa|pqm)\sigma_{m-3}(\kappa|m)]\xi_p\xi_q \nonumber\\
&+2\sigma_{m-3}^2(\kappa|m)\xi_m^2-2\dsum_{j\neq
m}\sigma_{m-3}(\kappa|jm)\sigma_{m-3}(\kappa|m)\xi_m\xi_j.\nonumber
\end{align}
Let's denote
$$L'=\sum_{j\neq m} 2\sigma_{m-4}(\kappa|jm)\xi_j^2-\sum_{p\neq q;p,q\neq m}\sigma_{m-4}(\kappa|pqm)\xi_p\xi_q.$$
By Cauchy-Schwarz inequality, we have
\begin{align}\label{s5.18}
&2\sigma_{m-3}^2(\kappa|m)\xi_m^2-2\dsum_{j\neq
m}\sigma_{m-3}(\kappa|jm)\sigma_{m-3}(\kappa|m)\xi_m\xi_j\geq
-\dfrac{1}{2}\Big(\dsum_{j\neq
m}\sigma_{m-3}(\kappa|jm)\xi_j\Big)^2.
\end{align}
Inserting (\ref{s5.18}) into (\ref{s5.16}), we get
\begin{align}\label{s5.19}
&\sigma_{m-3}(\kappa|m) L-\sigma_{m-2}(\kappa)L' \\
 \geq&2\dsum_{j\neq
m}[\sigma_{m-3}(\kappa|jm)\sigma_{m-3}(\kappa|m)-\sigma_{m-2}(\kappa|m)\sigma_{m-4}(\kappa|jm)]\xi_j^2\nonumber\\
&+\dsum_{p\neq q;p,q\neq
m}[\sigma_{m-2}(\kappa|m)\sigma_{m-4}(\kappa|pqm)-\sigma_{m-3}(\kappa|pqm)\sigma_{m-3}(\kappa|m)]\xi_p\xi_q \nonumber\\
&-\dfrac{1}{2}\Big(\dsum_{j\neq
m}\sigma_{m-3}(\kappa|jm)\xi_j\Big)^2.\nonumber
\end{align}

Since $\kappa\in\Gamma_{m-2}$, by (\ref{Gamma1}) we have
$$\kappa_{m-2}+\kappa_{m-1}+\kappa_m>0.$$
As we have assumed $\kappa_m<0$, we have $\kappa_{m-2}>0$. In fact, we further
have that the only possible zero entry is $\kappa_{m-1}$. If it
occurs, namely $\kappa_{m-1}=0$, it is easy to see that the $m-1$
dimensional vector $\tilde{\kappa}=(\kappa|m-1)\in\Gamma_{m-2}$.
Thus, in this case, the quadratic form $L$ can be rewritten as
\begin{eqnarray}
\\
L&=&2\sigma_{m-3}(\tilde{\kappa})\xi_{m-1}^2-2\sum_{j\neq m-1}\sigma_{m-3}(\tilde{\kappa}|j)\xi_{m-1}\xi_j+\sum_{j\neq m-1}\sigma_{m-3}(\tilde{\kappa}|j)\xi_j^2\nonumber\\
&&+\sum_{j\neq m-1}\sigma_{m-3}(\tilde{\kappa}|j)\xi_j^2 -\sum_{p,q\neq m-1, p\neq q}\sigma_{m-3}(\tilde{\kappa}|pq)\xi_p\xi_q\nonumber\\
&\geq&\sum_{j\neq m-1}\sigma_{m-3}(\tilde{\kappa}|j)\xi_{m-1}^2-2\sum_{j\neq m-1}\sigma_{m-3}(\tilde{\kappa}|j)\xi_{m-1}\xi_j+\sum_{j\neq m-1}\sigma_{m-3}(\tilde{\kappa}|j)\xi_j^2\nonumber\\
&=&\sum_{j\neq m-1}\sigma_{m-3}(\tilde{\kappa}|j)(\xi_{m-1}-\xi_j)^2,\nonumber
\end{eqnarray}
which implies our lemma. Here, in the first inequality, we have used
Lemma 10 in \cite{RW}. Thus, in the following, we always assume
$\kappa_{m-1}\neq 0$. On the other hand, $\kappa_{m-2}>0$ implies
$\kappa_i>0$ for $i\leq m-2$. Therefore, every entry of $\kappa$ is
non zero. Hence, we can let $\mu_i=\dfrac{1}{\kappa_i}$ for
$i=1,2,\cdots, m$. We further let $\tilde{\xi}_i=\mu_i\xi_i$. By
(\ref{s5.19}), using the notation $\mu_i,\tilde{\xi}_i$, we have
\begin{align}\label{s5.20}
&\frac{\sigma_{m-3}(\kappa|m)L-\sigma_{m-2}(\kappa)L'}{\sigma_{m-1}^2(\kappa|m)}\\
 \geq&2\dsum_{j\neq
m}[\mu_j\sigma_{1}(\mu|jm)\sigma_{2}(\mu|m)-\sigma_{1}(\mu|m)\mu_j\sigma_{2}(\mu|jm)]\xi_j^2\nonumber\\
&+\dsum_{p\neq q;p,q\neq
m}[\sigma_{1}(\mu|m)\mu_p\mu_q\sigma_{1}(\mu|pqm)-\sigma_{2}(\mu|m)\mu_p\mu_q]\xi_p\xi_q
-\dfrac{1}{2}\Big(\dsum_{j\neq
m}\mu_j\sigma_{1}(\mu|jm)\xi_j\Big)^2\nonumber\\
=&2\dsum_{j\neq
m}[\sigma_{1}^2(\mu|jm)-\sigma_{2}(\mu|jm)](\mu_j\xi_j)^2-\dfrac{1}{2}\Big(\dsum_{j\neq m}\sigma_{1}(\mu|jm)\mu_j\xi_j\Big)^2\nonumber\\
&+\dsum_{p\neq q;p,q\neq
m}[\sigma_{1}(\mu|m)\sigma_{1}(\mu|pqm)-\sigma_{2}(\mu|m)]\mu_p\xi_p\mu_q\xi_q\nonumber\\
=&\dfrac{1}{2}\dsum_{j\neq
m}\sigma_{1}^2(\mu|jm)\tilde\xi_j^2+\dsum_{j\neq
m}[\sigma_{1}^2(\mu|jm)-2\sigma_{2}(\mu|jm)]\tilde\xi_j^2\nonumber\\
&+\dsum_{p\neq q;p,q\neq
m}\Big(\dfrac{1}{2}\sigma_{1}^2(\mu|pqm)-\sigma_{2}(\mu|pqm)-\dfrac{3}{2}\mu_p\mu_q-\dfrac{1}{2}(\mu_p+\mu_q)\sigma_{1}(\mu|pqm)\Big)\tilde\xi_p\tilde\xi_q\nonumber\\
=&\dfrac{1}{2}\Big[\dsum_{j\neq
m}\sigma_{1}^2(\mu|jm)\tilde\xi_j^2+2\dsum_{j\neq m}\dsum_{s\neq
m,j}\mu_s^2\tilde\xi_j^2\nonumber\\
&+\dsum_{p\neq q;p,q\neq m}\Big(\dsum_{s\neq
m,p,q}\mu_s^2-\mu_p\mu_q-\mu_p\sigma_1(\mu|pm)-\mu_q\sigma_1(\mu|qm)\Big)\tilde\xi_p\tilde\xi_q\Big]\nonumber
\end{align}
\begin{align}
=&\dfrac{1}{2}\Big[\dsum_{j\neq
m}\sigma_{1}^2(\mu|jm)\tilde\xi_j^2+\Big(\sum_{j\neq m}\dsum_{s\neq
m,j}\mu_s^2\tilde\xi_j^2-\sum_{p\neq q;p,q\neq m}\mu_p\mu_q\tilde\xi_p\tilde\xi_q\Big)\nonumber\\
&+\Big(\sum_{j\neq m}\dsum_{s\neq
m,j}\mu_s^2\tilde\xi_j^2+\sum_{p\neq q;p,q\neq m}\sum_{s\neq
m,p,q}\mu_s^2\tilde{\xi}_p\tilde{\xi}_q\Big)-2\sum_{p\neq q;p,q\neq m}\mu_p\sigma_1(\mu|pm)\tilde\xi_p\tilde\xi_q\Big].\nonumber
\end{align}
A straightforward calculation shows
\begin{align}\label{s5.21}
&\dsum_{j\neq m}\dsum_{s\neq m,j}\mu_s^2\tilde\xi_j^2-\dsum_{p\neq
q;p,q\neq m}\mu_p\mu_q\tilde\xi_p\tilde\xi_q\\
=&\dfrac{1}{2}\dsum_{p\neq q;p,q\neq
m}\mu_p^2\tilde\xi_q^2+\dfrac{1}{2}\dsum_{p\neq q;p,q\neq
m}\mu_q^2\tilde\xi_p^2-\dsum_{p\neq q;p,q\neq
m}(\mu_p\tilde\xi_q)(\mu_q\tilde\xi_p)\nonumber\\
=&\dfrac{1}{2}\dsum_{p\neq q;p,q\neq
m}(\mu_p\tilde\xi_q-\mu_q\tilde\xi_p)^2,\nonumber
\end{align}
\begin{align}\label{s5.22}
&\dsum_{j\neq m}\dsum_{s\neq m,j}\mu_s^2\tilde\xi_j^2+\dsum_{p\neq
q;p,q\neq m}\dsum_{s\neq m,p,q}\mu_s^2\tilde\xi_p\tilde\xi_q\\
=&\dsum_{j,s\neq m;j\neq s}\mu_s^2\tilde\xi_j^2+\dsum_{p\neq
q;p,q\neq m}\dsum_{j\neq m,p,q}\mu_j^2\tilde\xi_p\tilde\xi_q\nonumber\\
=&\dsum_{j\neq m}\dsum_{s\neq m,j}\mu_j^2\tilde\xi_s^2+\dsum_{j\neq
m}\dsum_{p\neq q;p,q\neq
m,j}\mu_j^2\tilde\xi_p\tilde\xi_q=\dsum_{j\neq
m}\mu_j^2\Big(\dsum_{s\neq m,j}\tilde\xi_s\Big)^2,\nonumber
\end{align}
and
\begin{align}\label{s5.23}
&\dsum_{j\neq m}\sigma_{1}^2(\mu|jm)\tilde\xi_j^2+\dsum_{j\neq
m}\mu_j^2\Big(\dsum_{s\neq m,j}\tilde\xi_s\Big)^2-2\dsum_{j\neq
m}\mu_j\sigma_1(\mu|jm)\tilde\xi_j\dsum_{s\neq m,j}\tilde\xi_s\\
=&\dsum_{j\neq m}\Big(\sigma_1(\mu|jm)\tilde\xi_j-\mu_j\sum_{s\neq
m,j}\tilde\xi_s \Big)^2.\nonumber
\end{align}
Inserting  (\ref{s5.21}), (\ref{s5.22}), (\ref{s5.23}) into
(\ref{s5.20}), we obtain $$\sigma_{m-3}(\kappa|m)L- \sigma_{m-2}(\kappa)L'\geq 0.$$

Since $\kappa_m$ is the minimum entry, we have the $m-1$ dimensional
vector $\bar{\kappa}=(\kappa|m)\in\Gamma_{m-2}$. On the other hand, we can rewrite $L'$ to be
\begin{align}\label{s5.17}
L'=\sum_{j\neq m}2\sigma_{(m-1)-3}(\bar{\kappa}|j)\xi_j^2-\sum_{p\neq q;p,q\neq
m}\sigma_{(m-1)-3}(\bar{\kappa}|pq)\xi_p\xi_q.
\end{align}
Comparing $L'$ with $L$, we observe that $L$ and $L'$ possess the same type. Thus, we can use the induction argument to handle
 $L'$.
We repeat the above whole argument of handling $L$ to deal with $L'$, then we obtain
\begin{align*}&\sigma_{m-4}(\bar{\kappa}|m-1)L'\\
\geq &\sigma_{m-3}(\bar{\kappa})\Big[\sum_{j\neq m,
m-1}2\sigma_{m-5}(\bar{\kappa}|j,m-1)\xi_j^2-\sum_{p\neq q;p,q\neq
m,m-1}\sigma_{m-5}(\bar{\kappa}|p,q,m-1)\xi_p\xi_q\Big].
\end{align*}
Since $\kappa_{m-2}>0$, the $m-2$ dimensional vector
$(\kappa_1,\cdots,\kappa_{m-2})\in\Gamma_{m-2}$. Therefore, the above
quadratic form is nonnegative by Lemma \ref{lemm21b}.

\end{proof}

\begin{lemm}\label{lemm26}
Assume $\kappa=(\kappa_1,\cdots,\kappa_m)\in\Gamma_{m-2}$. For
any given index $1\leq j\leq m$, we have
\begin{align*}
4\sigma_{m-4}^2(\kappa)\geq \dsum_{s\neq
j}\sigma_{m-4}^2(\kappa|js).
\end{align*}
\end{lemm}

\begin{proof}
We may assume
$\kappa_1\geq\kappa_2\geq\cdots\geq\kappa_m$. Thus, for any index $s$ satisfying $1\leq s\leq m$ and $s\neq j,m$, we have
$$
0<\sigma_{m-4}(\kappa|js)\leq\sigma_{m-4}(\kappa|ms),
$$
which implies $$\sum_{s\neq j}\sigma^2_{m-4}(\kappa|js)\leq \sum_{s\neq m}\sigma_{m-4}^2(\kappa|ms).$$
Therefore, we only need to consider the case $j=m$.

For simplification purpose, in our proof, we will always denote
$\tilde\sigma_l=\sigma_l(\kappa|m)$ for $l=1,2,\cdots,m$.
 Using the identity
$\sigma_{m-2}(\kappa)=\kappa_m\sigma_{m-3}(\kappa|m)+\sigma_{m-2}(\kappa|m)>0$,
we get
\begin{eqnarray}\label{kappa}
\kappa_m>-\dfrac{\tilde\sigma_{m-2}}{\tilde\sigma_{m-3}}.
\end{eqnarray}
Then by (x) in section 2, we have
\begin{align*}
\sigma_{m-4}(\kappa)=&\kappa_m\tilde\sigma_{m-5}+\tilde\sigma_{m-4}
>\dfrac{1}{\tilde\sigma_{m-3}}[\tilde\sigma_{m-4}\tilde\sigma_{m-3}-\tilde\sigma_{m-5}\tilde\sigma_{m-2}]>0.
\end{align*}
Thus, squaring the both sides of the above inequality, we get
\begin{align}\label{s5.24}
4\sigma_{m-4}^2(\kappa)>&\dfrac{4}{\tilde\sigma_{m-3}^2}[\tilde\sigma_{m-4}^2\tilde\sigma_{m-3}^2
-2\tilde\sigma_{m-5}\tilde\sigma_{m-4}\tilde\sigma_{m-3}\tilde\sigma_{m-2}+\tilde\sigma_{m-5}^2\tilde\sigma_{m-2}^2].
\end{align}
Using Lemma \ref{lemm23} and (\ref{s5.24}), we have
\begin{align}\label{s5.26}
&\Big[4\sigma_{m-4}^2(\kappa)-\dsum_{s\neq m}\sigma_{m-4}^2(\kappa|ms)\Big]\tilde\sigma_{m-3}^2\\
>&\tilde\sigma_{m-4}^2 \tilde\sigma_{m-3}^2 +4\tilde\sigma_{m-5}^2 \tilde\sigma_{m-2}^2 +2\tilde\sigma_{m-5} \tilde\sigma_{m-3}^3
+4\tilde\sigma_{m-6} \tilde\sigma_{m-2} \tilde\sigma_{m-3}^2 \nonumber\\
&+6\tilde\sigma_{m-7} \tilde\sigma_{m-1} \tilde\sigma_{m-3}^2
-8\tilde\sigma_{m-5} \tilde\sigma_{m-4} \tilde\sigma_{m-3}
\tilde\sigma_{m-2} .\nonumber
\end{align}
 By Lemma \ref{lemm20}, we have
\begin{align}
\tilde\sigma_{m-4}^2
>&2\tilde\sigma_{m-3} \tilde\sigma_{m-5} -2\tilde\sigma_{m-2} \tilde\sigma_{m-6} +2\tilde\sigma_{m-1}
\tilde\sigma_{m-7},\label{s5.27}\\
\tilde\sigma_{m-3}^2
>&2\tilde\sigma_{m-2} \tilde\sigma_{m-4} -2\tilde\sigma_{m-1} \tilde\sigma_{m-5} .\label{s5.28}
\end{align}
We divide into two cases to deal with (\ref{s5.26}).

Case (a) Suppose $\tilde\sigma_{m-1}\geq 0$. Since $\kappa\in\Gamma_{m-2}$, the only possible negative entries of $\kappa$ are $\kappa_m,\kappa_{m-1}$. Uisng our assumption, we know $\kappa_{m-1}\geq 0$, which implies $(\kappa|m)\in\Gamma_{m-1}$.

In this case, inserting (\ref{s5.27}) and (\ref{s5.28}) into (\ref{s5.26}), we get
\begin{align}\label{new5.31}
&\Big[4\sigma_{m-4}^2(\kappa)-\dsum_{s\neq m}\sigma_{m-4}^2(\kappa|ms)\Big]\tilde\sigma_{m-3}^2\\
>&[2\tilde\sigma_{m-3}\tilde\sigma_{m-5}-2\tilde\sigma_{m-2}\tilde\sigma_{m-6}+2\tilde\sigma_{m-1}\tilde\sigma_{m-7}]
[2\tilde\sigma_{m-2}\tilde\sigma_{m-4}-2\tilde\sigma_{m-1}\tilde\sigma_{m-5}]\nonumber\\
&+4\tilde\sigma_{m-5}^2\tilde\sigma_{m-2}^2+2\tilde\sigma_{m-5}\tilde\sigma_{m-3}[2\tilde\sigma_{m-2}\tilde\sigma_{m-4}-2\tilde\sigma_{m-1}\tilde\sigma_{m-5}]\nonumber\\
&+4\tilde\sigma_{m-6}\tilde\sigma_{m-2}\tilde\sigma_{m-3}^2+6\tilde\sigma_{m-7}\tilde\sigma_{m-1}\tilde\sigma_{m-3}^2-8\tilde\sigma_{m-5}\tilde\sigma_{m-4}\tilde\sigma_{m-3}\tilde\sigma_{m-2}\nonumber\\
=&4(\tilde\sigma_{m-1}\tilde\sigma_{m-2}\tilde\sigma_{m-5}\tilde\sigma_{m-6}-\tilde\sigma_{m-1}^2\tilde\sigma_{m-5}\tilde\sigma_{m-7})\nonumber\\
&+4(\tilde\sigma_{m-6}\tilde\sigma_{m-2}\tilde\sigma_{m-3}^2-\tilde\sigma_{m-2}^2\tilde\sigma_{m-4}\tilde\sigma_{m-6})\nonumber\\
&+4(\tilde\sigma_{m-5}^2\tilde\sigma_{m-2}^2-2\tilde\sigma_{m-1}\tilde\sigma_{m-3}\tilde\sigma_{m-5}^2)\nonumber\\
&+6\tilde\sigma_{m-7}\tilde\sigma_{m-1}\tilde\sigma_{m-3}^2+4\tilde\sigma_{m-1}\tilde\sigma_{m-2}\tilde\sigma_{m-4}\tilde\sigma_{m-7}.\nonumber
\end{align}
By (x) in section 2, it is clear that
\begin{align*}
&\tilde\sigma_{m-1}\tilde\sigma_{m-2}\tilde\sigma_{m-5}\tilde\sigma_{m-6}-\tilde\sigma_{m-1}^2\tilde\sigma_{m-5}\tilde\sigma_{m-7}\geq
0,\\
&\tilde\sigma_{m-6}\tilde\sigma_{m-2}\tilde\sigma_{m-3}^2-\tilde\sigma_{m-2}^2\tilde\sigma_{m-4}\tilde\sigma_{m-6}\geq
0,\\
&\tilde\sigma_{m-5}^2\tilde\sigma_{m-2}^2-2\tilde\sigma_{m-1}\tilde\sigma_{m-3}\tilde\sigma_{m-5}^2\geq
0.
\end{align*}
Thus, combing the above three formulae with \eqref{new5.31}, we obtain
\begin{align*}
&\Big[4\sigma_{m-4}^2(\kappa)-\sum_{s\neq
m}\sigma_{m-4}^2(\kappa|ms)\Big]\tilde\sigma_{m-3}^2 \geq 0.
\end{align*}

Case (b) Suppose $\tilde\sigma_{m-1}< 0$. Same as the argument of Case (a), we have $\kappa_{m-1}<0$, which implies $\kappa_m<0$ Thus, using \eqref{kappa}, we know $\tilde{\sigma}_{m-2}>0$. Thus, we get $(\kappa|m)\in \Gamma_{m-2}$.

In this case, inserting (\ref{s5.28}) and (\ref{s5.27})  into (\ref{s5.26}), we get
\begin{align}\label{new5.32}
&\Big[4\sigma_{m-4}^2(\kappa)-\sum_{s\neq m}\sigma_{m-4}^2(\kappa|ms)\Big]\tilde\sigma_{m-3}^2\\
>&\tilde\sigma_{m-4}^2
[2\tilde\sigma_{m-2}\tilde\sigma_{m-4}-2\tilde\sigma_{m-1}\tilde\sigma_{m-5}]\nonumber\\
&+4\tilde\sigma_{m-5}^2\tilde\sigma_{m-2}^2+2\tilde\sigma_{m-5}\tilde\sigma_{m-3}[2\tilde\sigma_{m-2}\tilde\sigma_{m-4}-2\tilde\sigma_{m-1}\tilde\sigma_{m-5}]\nonumber\\
&+4\tilde\sigma_{m-6}\tilde\sigma_{m-2}\tilde\sigma_{m-3}^2+6\tilde\sigma_{m-7}\tilde\sigma_{m-1}\tilde\sigma_{m-3}^2-8\tilde\sigma_{m-5}\tilde\sigma_{m-4}\tilde\sigma_{m-3}\tilde\sigma_{m-2}\nonumber\\
=&2\tilde\sigma_{m-2}\tilde\sigma_{m-4}^3-2\tilde\sigma_{m-4}^2\tilde\sigma_{m-1}\tilde\sigma_{m-5}+4\tilde\sigma_{m-5}^2\tilde\sigma_{m-2}^2
-4\tilde\sigma_{m-2}\tilde\sigma_{m-3}\tilde\sigma_{m-4}\tilde\sigma_{m-5}\nonumber\\
&-4\tilde\sigma_{m-1}\tilde\sigma_{m-5}^2\tilde\sigma_{m-3}+4\tilde\sigma_{m-6}\tilde\sigma_{m-2}\tilde\sigma_{m-3}^2+6\tilde\sigma_{m-7}\tilde\sigma_{m-1}\tilde\sigma_{m-3}^2\nonumber\\
\geq&2\tilde\sigma_{m-2}\tilde\sigma_{m-4}[2\tilde\sigma_{m-3}\tilde\sigma_{m-5}-2\tilde\sigma_{m-2}\tilde\sigma_{m-6}+2\tilde\sigma_{m-1}\tilde\sigma_{m-7}]-2\tilde\sigma_{m-4}^2\tilde\sigma_{m-1}\tilde\sigma_{m-5}\nonumber\\
&+4\tilde\sigma_{m-5}^2\tilde\sigma_{m-2}^2
-4\tilde\sigma_{m-2}\tilde\sigma_{m-3}\tilde\sigma_{m-4}\tilde\sigma_{m-5}-4\tilde\sigma_{m-1}\tilde\sigma_{m-5}^2\tilde\sigma_{m-3}\nonumber\\
&+4\tilde\sigma_{m-6}\tilde\sigma_{m-2}[2\tilde\sigma_{m-2}\tilde\sigma_{m-4}-2\tilde\sigma_{m-1}\tilde\sigma_{m-5}]+6\tilde\sigma_{m-7}\tilde\sigma_{m-1}\tilde\sigma_{m-3}^2\nonumber
\end{align}
\begin{align}
=&\Big(4\tilde\sigma_{m-1}\tilde\sigma_{m-2}\tilde\sigma_{m-4}\tilde\sigma_{m-7}-8\tilde\sigma_{m-1}\tilde\sigma_{m-2}\tilde\sigma_{m-5}\tilde\sigma_{m-6}\Big)\nonumber\\
&+\Big(-2\tilde\sigma_{m-4}^2\tilde\sigma_{m-1}\tilde\sigma_{m-5}
-4\tilde\sigma_{m-1}\tilde\sigma_{m-5}^2\tilde\sigma_{m-3}+6\tilde\sigma_{m-7}\tilde\sigma_{m-1}\tilde\sigma_{m-3}^2\Big)\nonumber\\
&+4\tilde\sigma_{m-2}^2\tilde\sigma_{m-4}\tilde\sigma_{m-6}+4\tilde\sigma_{m-5}^2\tilde\sigma_{m-2}^2.\nonumber
\end{align}
Here, we also use (\ref{s5.28}) and (\ref{s5.27}) in the third
inequality. Using (x) in section 2, we have
\begin{align*}
&-8\tilde\sigma_{m-1}\tilde\sigma_{m-2}\tilde\sigma_{m-5}\tilde\sigma_{m-6}+4\tilde\sigma_{m-1}\tilde\sigma_{m-2}\tilde\sigma_{m-4}\tilde\sigma_{m-7}\geq 0,\\
&-2\tilde\sigma_{m-4}^2\tilde\sigma_{m-1}\tilde\sigma_{m-5}-4\tilde\sigma_{m-1}\tilde\sigma_{m-5}^2\tilde\sigma_{m-3}
+6\tilde\sigma_{m-7}\tilde\sigma_{m-1}\tilde\sigma_{m-3}^2\geq 0.
\end{align*}
In the above second inequality, we have used
$\tilde\sigma_{m-4}^2\tilde\sigma_{m-5}\geq\tilde\sigma_{m-5}^2\tilde\sigma_{m-3}$.
Thus, combing the above two inequalities with \eqref{new5.32}, we obtain
\begin{align*}
&\Big[4\sigma_{m-4}^2(\kappa)-\sum_{s\neq
m}\sigma_{m-4}^2(\kappa|ms)\Big]\tilde\sigma_{m-3}^2\geq 0.
\end{align*}
\end{proof}

\begin{lemm}\label{lemm27}
Assume $\kappa=(\kappa_1,\cdots,\kappa_{m})\in\Gamma_{m-2}$, $m\geq
5$ and $\kappa_1\geq\cdots\geq\kappa_m$. For any given small
constant $0<\delta<1$, if $-\kappa_m\geq\delta \kappa_1$, or $
\kappa_{m-1}\geq \delta \kappa_1$, then there exists a constant
$\delta'>0$ only depending on $\delta$, such that
$\sigma_{m-3}(\kappa|1)\geq\delta'\kappa_1^{m-3}$.
\end{lemm}
\begin{proof}
For the sake of simplification, we denote
$\tilde\sigma_s=\sigma_s(\kappa|1,m-1,m)$ in our proof. Since
$\kappa\in\Gamma_{m-2}$, by (\ref{Gamma1}) we have
$$\kappa_{m-2}+\kappa_{m-1}+\kappa_m>0.$$
Thus, our assumption $-\kappa_m\geq\delta \kappa_1$ or $ \kappa_{m-1}\geq \delta
\kappa_1$ implies
$\kappa_{m-2}\geq\dfrac{\delta}{2}\kappa_1$. We select
\begin{eqnarray}\label{new2}
\delta'=
\dmin\Big\{\frac{1}{2^{m-1}}\delta^{m-2},\delta^{m-1}\Big\}.
\end{eqnarray}
It is obvious
\begin{align}\label{a5.28}
\sigma_{m-3}(\kappa|1)-\kappa_m\sigma_{m-4}(\kappa|1m)=\sigma_{m-3}(\kappa|1m),
\end{align}
and
\begin{align}\label{s5.29}
\kappa_1\sigma_{m-3}(\kappa|1)+\kappa_m\sigma_{m-3}(\kappa|1m)+\sigma_{m-2}(\kappa|1m)=\sigma_{m-2}(\kappa)>0.
\end{align}

Now, we divide into three cases to prove our lemma.
\par
Case (a): Suppose $-\kappa_m\geq\delta \kappa_1$ and $\kappa_{m-1}\leq
0$. Using \eqref{a5.28} and our assumption, we have $\sigma_{m-3}(\kappa|1m)>0$.

Multiple $\kappa_m$ in both side of (\ref{a5.28}) and insert it to
(\ref{s5.29}), then we get
\begin{align}\label{s5.30}
(\kappa_1+\kappa_m)\sigma_{m-3}(\kappa|1)-\kappa_m^2\sigma_{m-4}(\kappa|1m)+\sigma_{m-2}(\kappa|1m)=\sigma_{m-2}(\kappa)>0.
\end{align}

 Since we have
$$\sigma_{m-3}(\kappa|1m)=\kappa_{m-1}\tilde\sigma_{m-4}+\tilde\sigma_{m-3}>0,$$ we get
$\kappa_{m-1}>-\dfrac{\tilde\sigma_{m-3}}{\tilde\sigma_{m-4}}$. Thus,
we have
\begin{align}\label{a5.31}
\sigma_{m-4}(\kappa|1m)=&\kappa_{m-1}\tilde\sigma_{m-5}+\tilde\sigma_{m-4}
>\dfrac{\tilde\sigma_{m-4}^2-\tilde\sigma_{m-5}\tilde\sigma_{m-3}}{\tilde\sigma_{m-4}}\\
\geq&\Big(1-\dfrac{C_{m-3}^{m-5}}{C_{m-3}^{m-4}C_{m-3}^{m-4}}\Big)\tilde\sigma_{m-4}\geq
\frac{m-2}{2^{m-3}}\delta^{m-4}\kappa_1^{m-4}.\nonumber
\end{align}
Here, we have used Newton's inequality in the third
inequality
 and used $\kappa_{m-2}\geq\dfrac{\delta}{2}\kappa_1$
in the last inequality.
\par
 Note that, in this case, $\kappa_m\leq \kappa_{m-1}\leq 0$, which implies $\sigma_{m-2}(\kappa|1m)=\kappa_2\cdots\kappa_{m-1}\leq 0$. On the other hand, by Lemma \ref{lemm9}, we have
 $$\kappa_1\geq\kappa_1+\kappa_m>\Big(1-\dfrac{2}{m-2}\Big)\kappa_1>0.$$
Thus, combing
 (\ref{s5.30}), (\ref{a5.31}) with $\sigma_{m-2}(\kappa|1m)\leq 0$,
we obtain
$$
\sigma_{m-3}(\kappa|1)\geq
\frac{m-2}{2^{m-3}}\delta^{m-2}\kappa_1^{m-2}/(\kappa_1+\kappa_m)\geq\delta'
\kappa_1^{m-3}.
$$
Here in the last inequality, we have used \eqref{new2}.

Case (b): Suppose $-\kappa_m\geq\delta \kappa_1,
0<\kappa_{m-1}\leq\delta\kappa_1$. Therefore, our assumption implies
$\kappa_m+\kappa_{m-1}\leq 0$. It is obvious that
\begin{align}\label{s5.31}
\kappa_1\sigma_{m-3}(\kappa|1)+\kappa_m\kappa_{m-1}\tilde\sigma_{m-4}+(\kappa_m+\kappa_{m-1})\tilde\sigma_{m-3}=\sigma_{m-2}(\kappa)>0.
\end{align}
If $\kappa_{m-1}\geq -\dfrac{\kappa_m}{2}$, by \eqref{s5.31}, we
have
\begin{align*}
\sigma_{m-3}(\kappa|1)>-\kappa_m\kappa_{m-1}\tilde\sigma_{m-4}/\kappa_1>\delta'\kappa_1^{m-3}.
\end{align*}
If $\kappa_{m-1}< -\dfrac{\kappa_m}{2}$, by \eqref{s5.31}, we have
\begin{align*}
\sigma_{m-3}(\kappa|1)>-(\kappa_m+\kappa_{m-1})\tilde\sigma_{m-3}/\kappa_1>\delta'\kappa_1^{m-3}.
\end{align*}
Here, in the above two inequalities, we both have used
$\kappa_{m-2}\geq\dfrac{\delta}{2}\kappa_1$ and \eqref{new2}.
\par
Case (c): Suppose $\kappa_{m-1}\geq \delta \kappa_1$. Using
(\ref{s5.29}), we have
$$\kappa_m>\dfrac{-\sigma_{m-2}(\kappa|1m)-\kappa_1\sigma_{m-3}(\kappa|1)}{\sigma_{m-3}(\kappa|1m)},$$
Thus, we get

\begin{align*}
\sigma_{m-3}(\kappa|1)=&\kappa_m\sigma_{m-4}(\kappa|1m)+\sigma_{m-3}(\kappa|1m)\\
>&\dfrac{-\sigma_{m-2}(\kappa|1m)\sigma_{m-4}(\kappa|1m)-\kappa_1\sigma_{m-3}(\kappa|1)\sigma_{m-4}(\kappa|1m)
+\sigma_{m-3}^2(\kappa|1m)}{\sigma_{m-3}(\kappa|1m)}.
\end{align*}
Solving the above inequality with respect to $\sigma_{m-3}(\kappa|1)$ and using
Newton's inequality, we obtain
\begin{align*}
\sigma_{m-3}(\kappa|1)>&\dfrac{\sigma_{m-3}^2(\kappa|1m)-\sigma_{m-2}(\kappa|1m)\sigma_{m-4}(\kappa|1m)}{\sigma_{m-3}(\kappa|1m)}
\Big/\Big(1+\dfrac{\kappa_1\sigma_{m-4}(\kappa|1m)}{\sigma_{m-3}(\kappa|1m)}\Big)\\
>&\Big(1-\dfrac{C_{m-2}^{m-4}}{C_{m-2}^{m-3}C_{m-2}^{m-3}}\Big)\sigma_{m-3}(\kappa|1m)\Big/\Big(1+\dfrac{\kappa_1\sigma_{m-4}(\kappa|1m)}{\sigma_{m-3}(\kappa|1m)}\Big)\\
>&\dfrac{m-1}{2}\kappa_3\cdots\kappa_{m-1}\Big/\Big(1+\dfrac{\kappa_1C_{m-2}^{m-4}\kappa_2\cdots\kappa_{m-3}}{C_{m-2}^{m-3}\kappa_3\cdots\kappa_{m-1}}\Big).
\end{align*}
Since $\kappa_1\geq\cdots\geq\kappa_{m-1}\geq\delta \kappa_1$, we
get
\begin{align*}
\sigma_{m-3}(\kappa|1)>&\dfrac{m-1}{2}\delta^{m-3}\kappa_1^{m-3}\Big/\Big(1+\dfrac{m-3}{2\delta^2}\Big)\geq\delta'\kappa_1^{m-3}.
\end{align*}
\end{proof}

\section{The Proof of Theorem \ref{maintheo} for cases A, B1 and B2}

In this section, we will divided into two steps to prove that (\ref{s4.05}) is a non negative form for cases A, B1 and B2. The main result of this section is the following theorem:
\begin{theo}\label{theorem}
Suppose $\kappa=(\kappa_{1},\cdots,\kappa_{n})\in\Gamma_{n-2}$, $n\geq 5$ and $\sigma_{n-2}(\kappa)$ has  positive upper and lower bounds $N_0\leq \sigma_{n-2}(\kappa)\leq N$. Further suppose $\kappa$ satisfies the condition of Case A, Case B1,
or Case B2 defined in section 4. For any given index $1\leq i\leq n$, if $\kappa_i>\kappa_1-\sqrt{\kappa_1}/n$, we have, for any $\xi=(\xi_1,\cdots,\xi_n)\in\mathbb{R}^n$,
\begin{align}\label{s6.01}
\dfrac{8\kappa_i^2}{9}\textbf{A}_{n-2;i}+\textbf{C}_{n-2;i}
\geq&\dfrac{1}{20} \sigma_{n-3}^2(\kappa|i)\sum_{j\neq i}\xi_j^2\geq
0,
\end{align}
and
\begin{align}\label{s6.02}
&\dfrac{\kappa_i^2}{9}\textbf{A}_{n-2;i}+\sigma_{n-2}(\kappa)\textbf{B}_{n-2;i}
-c_{n-2,K}\textbf{D}_{n-2;i}
+\dfrac{1}{20}\sigma_{n-3}^2(\kappa|i)\sum_{j\neq i}\xi_j^2\geq 0,
\end{align}
when $\kappa_1$ is sufficiently large. Here the definitions of $\textbf{A}_{n-2;i}, \textbf{B}_{n-2;i},
\textbf{C}_{n-2;i}, \textbf{D}_{n-2;i}$ are presented in section 4
and $c_{n-2,K}$ is defined by \eqref{ckK}. It is clear that
\eqref{s6.01} and \eqref{s6.02} give the non negativity of \eqref{s4.05}.
\end{theo}

 Throughout this section, we always use $\bar{\kappa}$ to denote $(\kappa|i)$. Thus, we have $(\bar{\kappa}|j)=(\kappa|ij)$ and $(\bar{\kappa}|pq)=(\kappa|ipq)$.

\begin{lemm}\label{lemm28}
Assume $\kappa=(\kappa_1,\kappa_2,\cdots,\kappa_n)\in\Gamma_{n-2}$,
 $n\geq 5$ and $\kappa_1\geq\cdots\geq\kappa_n$. We let
 $\delta=1/10$. Then, for any given index $1\leq i\leq n$, if
$\kappa_i>\kappa_1-\sqrt{\kappa_1}/n$, we have
\begin{align}\label{a6.4}
0<\dfrac{\sigma_{n-3}(\bar{\kappa})}{\sigma_{n-5}(\bar{\kappa})}\leq
(1+\delta)\kappa_1^2+\dfrac{\sigma_{n-2}(\kappa)}{\kappa_i},
\end{align}
when $\kappa_1$ is sufficiently large.
\end{lemm}

\begin{proof} In the cone $\Gamma_{n-2}$, the only possible non positive entries are $\kappa_{n-1},\kappa_n$. We divide into three cases to prove our lemma.
\par
Case (a): Suppose $\kappa_n\geq 0$. If
$\kappa_{n-3}\kappa_{n-4}C_{n-1}^{n-3}\leq\kappa_1^2$, by Lemma
\ref{lemm10}, we get
\begin{align*}
\dfrac{\sigma_{n-3}(\bar{\kappa})}{\sigma_{n-5}(\bar{\kappa})}\leq
\dfrac{\kappa_1\cdots\kappa_{n-3}C_{n-1}^{n-3}}{\kappa_1\cdots\kappa_{n-5}}=\kappa_{n-3}\kappa_{n-4}C_{n-1}^{n-3}\leq\kappa_1^2,
\end{align*}
which implies (\ref{a6.4}). If
$\kappa_{n-3}\kappa_{n-4}C_{n-1}^{n-3}>\kappa_1^2$, namely
$\kappa_{n-3}\kappa_{n-4}> \kappa_1^2/C_{n-1}^{n-3}$, by Lemma
\ref{lemm10}, we get
$$\sigma_{n-5}(\bar{\kappa})\geq\kappa_2\cdots\kappa_{n-4}>\frac{\kappa_1^{n-5}}{(C_{n-1}^{n-3})^{\frac{n-5}{2}}}>1,$$
when $\kappa_1$ is sufficiently large. Thus, we obtain
\begin{align*}
\dfrac{\sigma_{n-3}(\bar{\kappa})}{\sigma_{n-5}(\bar{\kappa})}=\dfrac{\kappa_i\sigma_{n-3}(\bar{\kappa})}{\kappa_i\sigma_{n-5}(\bar{\kappa})}<
\dfrac{\sigma_{n-2}(\kappa)}{\kappa_i},
\end{align*}
which implies (\ref{a6.4}).

 Case (b): Suppose $\kappa_n<0,\kappa_{n-1}\geq
0$. We need to prove
\begin{align*}
&\sigma_{n-3}(\bar{\kappa})-\Big((1+\delta)\kappa_1^2+\dfrac{\sigma_{n-2}(\kappa)}{\kappa_i}\Big)\sigma_{n-5}(\bar{\kappa})\leq
0.
\end{align*}
Using the identity,
\begin{align*}
\sigma_{n-2}(\kappa)=&\kappa_i\sigma_{n-3}(\bar{\kappa})+\sigma_{n-2}(\bar{\kappa})\\
=&\kappa_i[\kappa_n\sigma_{n-4}(\bar{\kappa}|n)+\sigma_{n-3}(\bar{\kappa}|n)]+\kappa_n\sigma_{n-3}(\bar{\kappa}|n)+\sigma_{n-2}(\bar{\kappa}|n),
\end{align*}
we get
$$\kappa_n=\dfrac{\sigma_{n-2}(\kappa)-\kappa_i\sigma_{n-3}(\bar{\kappa}|n)-\sigma_{n-2}(\bar{\kappa}|n)}{\kappa_i\sigma_{n-4}(\bar{\kappa}|n)+\sigma_{n-3}(\bar{\kappa}|n)}.$$
In this case, for the sake of simplification, we always denote $\tilde\sigma_m=\sigma_m(\bar{\kappa}|n)$. Inserting the above
identity into the following first equality, we have
\begin{align*}
&\sigma_{n-3}(\bar{\kappa})-\Big( (1+\delta)\kappa_1^2+\dfrac{\sigma_{n-2}(\kappa)}{\kappa_i}\Big)\sigma_{n-5}(\bar{\kappa})\\
=&\kappa_n\tilde\sigma_{n-4}+\tilde\sigma_{n-3}- (1+\delta)\kappa_1^2(\kappa_n\tilde\sigma_{n-6}+\tilde\sigma_{n-5})-\dfrac{\sigma_{n-2}(\kappa)\sigma_{n-5}(\bar{\kappa})}{\kappa_i}\\
=&\sigma_{n-2}(\kappa)G1+\dfrac{G2}{\kappa_i\tilde\sigma_{n-4}+\tilde\sigma_{n-3}},
\end{align*}
where $G1,G2$ are defined by
\begin{align*}
G1=&\dfrac{\tilde\sigma_{n-4}-(1+\delta)\kappa_1^2\tilde\sigma_{n-6}}{\kappa_i\tilde\sigma_{n-4}+\tilde\sigma_{n-3}}-\dfrac{\sigma_{n-5}(\bar{\kappa})}{\kappa_i}.\\
G2=&\tilde\sigma_{n-3}^2-\tilde\sigma_{n-2}\tilde\sigma_{n-4}-
(1+\delta)\kappa_i\kappa_1^2(\tilde\sigma_{n-4}\tilde\sigma_{n-5}-\tilde\sigma_{n-3}\tilde\sigma_{n-6})\\
&-(1+\delta)\kappa_1^2(\tilde\sigma_{n-3}\tilde\sigma_{n-5}-\tilde\sigma_{n-2}\tilde\sigma_{n-6}).
\end{align*}
Thus, we only need to show $G1\leq 0, G2\leq 0$.

Firstly, let's prove $G1\leq 0$. We divide into the cases of $n=5$ and $n>5$ to prove it respectively. Suppose $n=5$, then $\sigma_{n-5}(\bar{\kappa})=1$. By $\kappa_1>0$, we obviously have
$$
G1\leq\dfrac{\tilde\sigma_{n-4}}{\kappa_i\tilde\sigma_{n-4}}-\dfrac{1}{\kappa_i}=0.
$$
Suppose $n>5$. If we further assume  $\tilde\sigma_{n-4}\leq
\kappa_1^2\tilde\sigma_{n-6}$, we clearly have $G1\leq 0$. Therefore, we assume
$\tilde\sigma_{n-4}\geq \kappa_1^2\tilde\sigma_{n-6}$. With this assumption, we have
\begin{eqnarray}\label{new6.1}
\dfrac{\kappa_1\cdots\kappa_{n-3}}{\kappa_i}\geq\dfrac{\tilde\sigma_{n-4}}{C_{n-2}^{n-4}}\geq
\dfrac{\kappa_1^2\tilde\sigma_{n-6}}{C_{n-2}^{n-4}}\geq
\frac{1}{C_{n-2}^{n-4}}\kappa_1^2\kappa_2\cdots\kappa_{n-5}.
\end{eqnarray}
 By the condition
$\kappa_i>\kappa_1-\sqrt{\kappa_1}/n$, we know
$(1+\delta)\kappa_i>\kappa_1$ if $\kappa_1>\dfrac{1}{\delta^2}$.
Thus, \eqref{new6.1} gives $\kappa_{n-3}\kappa_{n-4}\geq
\dfrac{\kappa_1\kappa_i}{C_{n-2}^{n-4}}\geq
\dfrac{\kappa_1^2}{2C_{n-2}^{n-4}}$. Note that
$\kappa_1\geq\cdots\geq\kappa_n$, then, by Lemma \ref{lemm10} and $\bar{\kappa}\in\Gamma_{n-3}$, we have
\begin{align}\label{s6.05}
\sigma_{n-5}(\bar{\kappa})\geq
&\kappa_{2}\cdots\kappa_{n-4}\geq\dfrac{\kappa_1^{n-5}}{(2C_{n-2}^{n-4})^{(n-5)/2}}.
\end{align}
Thus we get
$$
G1\leq
\dfrac{\tilde\sigma_{n-4}}{\kappa_i\tilde\sigma_{n-4}+\tilde\sigma_{n-3}}-\dfrac{\sigma_{n-5}(\bar{\kappa})}{\kappa_i}<
\dfrac{1}{\kappa_i}-\dfrac{\kappa_1^{n-5}}{\kappa_i(2C_{n-2}^{n-4})^{(n-5)/2}}<0,
$$
if $\kappa_1$ is sufficiently large.

Secondly, let's prove $G2\leq 0$. Using
$(1+\delta)\kappa_i\geq\kappa_1$ and (x) in section 2, we have
\begin{align}\label{new6.2}
&G2\\
\leq&\tilde\sigma_{n-3}^2-\tilde\sigma_{n-2}\tilde\sigma_{n-4}-
(\tilde\sigma_{n-4}\tilde\sigma_{n-5}-\tilde\sigma_{n-3}\tilde\sigma_{n-6})\kappa_1^3-(\tilde\sigma_{n-3}\tilde\sigma_{n-5}-\tilde\sigma_{n-2}\tilde\sigma_{n-6})\kappa_1^2\nonumber\\
\leq&\tilde\sigma_{n-3}^2-\tilde\sigma_{n-2}\tilde\sigma_{n-4}-
\Big(1-\dfrac{C_{n-2}^{n-3}C_{n-2}^{n-6}}{C_{n-2}^{n-4}C_{n-2}^{n-5}}\Big)\kappa_1^3\tilde\sigma_{n-4}\tilde\sigma_{n-5}-\Big(1-\dfrac{C_{n-2}^{n-2}C_{n-2}^{n-6}}{C_{n-2}^{n-3}C_{n-2}^{n-5}}\Big)\kappa_1^2\tilde\sigma_{n-3}\tilde\sigma_{n-5}\nonumber\\
=&\tilde\sigma_{n-3}^2-\tilde\sigma_{n-2}\tilde\sigma_{n-4}-
\dfrac{n-1}{2(n-3)}\kappa_1^3\tilde\sigma_{n-4}\tilde\sigma_{n-5}-\dfrac{3n-3}{4n-8}\kappa_1^2\tilde\sigma_{n-3}\tilde\sigma_{n-5},\nonumber
\end{align}
We also need to divide into two cases of $n=5$ and $n>5$ to prove it respectively. For $n=5$, we may assume $(\bar{\kappa}|5)=(\kappa_a,\kappa_b,\kappa_c)$ and
$\kappa_1\geq\kappa_a\geq\kappa_b\geq\kappa_c\geq0$, then we have
$\tilde\sigma_{1}=\kappa_a+\kappa_b+\kappa_c,\tilde\sigma_{2}=\kappa_a\kappa_b+\kappa_a\kappa_c+\kappa_b\kappa_c,
\tilde\sigma_{3}=\kappa_a\kappa_b\kappa_c$. Thus, \eqref{new6.2} becomes
\begin{align*}
G2\leq&\tilde\sigma_{2}^2-\tilde\sigma_3\tilde\sigma_{1}-
\kappa_1^3\tilde\sigma_{1}-\kappa_1^2\tilde\sigma_{2}\\
=&\kappa_a^2\kappa_b^2+\kappa_a^2\kappa_c^2+\kappa_b^2\kappa_c^2+\kappa_a\kappa_b\kappa_c(\kappa_a+\kappa_b+\kappa_c)\\
&-\kappa_1^3(\kappa_a+\kappa_b+\kappa_c)-\kappa_1^2(\kappa_a\kappa_b+\kappa_a\kappa_c+\kappa_b\kappa_c)\leq
0.
\end{align*}
For $n\geq 6$, by Lemma \ref{lemm8}, we have
\begin{align*}
\kappa_1^3\tilde\sigma_{n-4}\tilde\sigma_{n-5}\geq
\dfrac{(n-3)^2(n-4)}{12}\tilde\sigma_{n-3}^2.
\end{align*}
Thus, we get
\begin{align*}
G2\leq&\tilde\sigma_{n-3}^2-
\dfrac{n-1}{2(n-3)}\kappa_1^3\tilde\sigma_{n-4}\tilde\sigma_{n-5}\\
\leq&\tilde\sigma_{n-3}^2-\dfrac{(n-1)(n-3)(n-4)}{24}\tilde\sigma_{n-3}^2\leq
0.
\end{align*}

Case (c): Suppose $\kappa_n\leq\kappa_{n-1}<0, \kappa_{n-2}>0$. For
simplification purpose, in the proof of this case, we denote
$\tilde\sigma_m=\sigma_m(\bar{\kappa}|n,n-1)$. Using the identity
\begin{align*}
\sigma_{n-2}(\kappa)=&\kappa_i\sigma_{n-3}(\bar{\kappa})+\sigma_{n-2}(\bar{\kappa})\\
=&\kappa_i[\tilde\sigma_{n-3}+(\kappa_n+\kappa_{n-1})\tilde\sigma_{n-4}+\kappa_n\kappa_{n-1}\tilde\sigma_{n-5}]\\
&+(\kappa_n+\kappa_{n-1})\tilde\sigma_{n-3}+\kappa_n\kappa_{n-1}\tilde\sigma_{n-4},
\end{align*}
we get
\begin{align}\label{a6.5}
\kappa_n+\kappa_{n-1}=\dfrac{\sigma_{n-2}(\kappa)-\kappa_i\tilde\sigma_{n-3}-\kappa_i\kappa_n\kappa_{n-1}\tilde\sigma_{n-5}-\kappa_n\kappa_{n-1}\tilde\sigma_{n-4}}{\kappa_i\tilde\sigma_{n-4}+\tilde\sigma_{n-3}}.
\end{align}
 Then inserting (\ref{a6.5}) into the following first equality , we have
\begin{align*}
&\sigma_{n-3}(\bar{\kappa})-\Big((1+\delta)\kappa_1^2+\dfrac{\sigma_{n-2}(\kappa)}{\kappa_i}\Big)\sigma_{n-5}(\bar{\kappa})\\
=&\kappa_n\kappa_{n-1}\tilde\sigma_{n-5}+(\kappa_n+\kappa_{n-1})\tilde\sigma_{n-4}+\tilde\sigma_{n-3}-\dfrac{\sigma_{n-2}(\kappa)\sigma_{n-5}(\bar{\kappa})}{\kappa_i}\\
&- (1+\delta)\kappa_1^2[\kappa_n\kappa_{n-1}\tilde\sigma_{n-7}+(\kappa_n+\kappa_{n-1})\tilde\sigma_{n-6}+\tilde\sigma_{n-5}]\\
=&\sigma_{n-2}(\kappa)G3+\dfrac{G4}{\kappa_i\tilde\sigma_{n-4}+\tilde\sigma_{n-3}},
\end{align*}
where $G3,G4$ are defined by
\begin{align*}
G3=&\dfrac{\tilde\sigma_{n-4}-(1+\delta)\kappa_1^2\tilde\sigma_{n-6}}{\kappa_i\tilde\sigma_{n-4}+\tilde\sigma_{n-3}}-\dfrac{\sigma_{n-5}(\bar{\kappa})}{\kappa_i}
,\\
 G4=&
-(1+\delta)\kappa_i\kappa_1^2(\tilde\sigma_{n-4}\tilde\sigma_{n-5}-\tilde\sigma_{n-3}\tilde\sigma_{n-6})-(1+\delta)\kappa_1^2\tilde\sigma_{n-3}\tilde\sigma_{n-5}+\tilde\sigma_{n-3}^2\\
&+
\kappa_n\kappa_{n-1}[(1+\delta)\kappa_i\kappa_1^2(\tilde\sigma_{n-5}\tilde\sigma_{n-6}-\tilde\sigma_{n-4}\tilde\sigma_{n-7})+(1+\delta)\kappa_1^2(\tilde\sigma_{n-4}\tilde\sigma_{n-6}-\tilde\sigma_{n-3}\tilde\sigma_{n-7})]\\
&-
\kappa_n\kappa_{n-1}(\tilde\sigma_{n-4}^2-\tilde\sigma_{n-3}\tilde\sigma_{n-5}).
\end{align*}
Obviously, same as the argument dealing with $G1$, we have $G3\leq
0$. Now we consider the term $G4$.  By Newton's inequality, we have
$\tilde\sigma_{n-4}^2-\tilde\sigma_{n-3}\tilde\sigma_{n-5}\geq 0$.
Then, we get
\begin{align}\label{s6.06}
\dfrac{G4}{(1+\delta)\kappa_1^2} \leq&
\dfrac{\tilde\sigma_{n-3}^2}{(1+\delta)\kappa_1^2}-\kappa_i(\tilde\sigma_{n-4}\tilde\sigma_{n-5}-\tilde\sigma_{n-3}\tilde\sigma_{n-6})-\tilde\sigma_{n-3}\tilde\sigma_{n-5}\\
&+
\kappa_n\kappa_{n-1}\kappa_i(\tilde\sigma_{n-5}\tilde\sigma_{n-6}-\tilde\sigma_{n-4}\tilde\sigma_{n-7})\nonumber\\
&+\kappa_n\kappa_{n-1}(\tilde\sigma_{n-4}\tilde\sigma_{n-6}-\tilde\sigma_{n-3}\tilde\sigma_{n-7})\nonumber\\
\leq&\dfrac{\tilde\sigma_{n-3}^2}{(1+\delta)\kappa_1^2}-\kappa_i(\tilde\sigma_{n-4}\tilde\sigma_{n-5}-\tilde\sigma_{n-3}\tilde\sigma_{n-6}-\kappa_n\kappa_{n-1}\tilde\sigma_{n-5}\tilde\sigma_{n-6})\nonumber\\
&-\tilde\sigma_{n-3}\tilde\sigma_{n-5}+\kappa_n\kappa_{n-1}\tilde\sigma_{n-4}\tilde\sigma_{n-6}.\nonumber
\end{align}
Since $\kappa_n\leq\kappa_{n-1}<0$, by Lemma \ref{lemm9} and $\bar{\kappa}\in\Gamma_{n-3}$, we get
\begin{align}\label{a6.7}
0<\kappa_n\kappa_{n-1}\leq\dfrac{(\kappa_n+\kappa_{n-1})^2}{4}\leq
\dfrac{\tilde\sigma_{n-3}^2}{\tilde\sigma_{n-4}^2}.
\end{align}

 Using \eqref{a6.7} and (x) in scetion 2, we get
\begin{align}\label{s6.07}
-\tilde\sigma_{n-3}\tilde\sigma_{n-5}+\kappa_n\kappa_{n-1}\tilde\sigma_{n-4}\tilde\sigma_{n-6}
\leq-\tilde\sigma_{n-3}\tilde\sigma_{n-5}+\dfrac{\tilde\sigma_{n-3}^2}{\tilde\sigma_{n-4}^2}\tilde\sigma_{n-4}\tilde\sigma_{n-6}\leq
0.
\end{align}
 Using (x) in section 2 again, $(1+\delta)\kappa_i\geq\kappa_1$ and (\ref{a6.7}), we get
\begin{align}\label{s6.08}
&\dfrac{\tilde\sigma_{n-3}^2}{(1+\delta)\kappa_1^2}-\kappa_i(\tilde\sigma_{n-4}\tilde\sigma_{n-5}-\tilde\sigma_{n-3}\tilde\sigma_{n-6}-\kappa_n\kappa_{n-1}\tilde\sigma_{n-5}\tilde\sigma_{n-6})\\
\leq&\dfrac{\tilde\sigma_{n-3}^2}{(1+\delta)\kappa_1^2}-\kappa_i\Big(\dfrac{2n-4}{3(n-3)}\tilde\sigma_{n-4}\tilde\sigma_{n-5}-\dfrac{\tilde\sigma_{n-3}^2}{\tilde\sigma_{n-4}^2}\tilde\sigma_{n-5}\tilde\sigma_{n-6}\Big)\nonumber\\
\leq&\dfrac{\tilde\sigma_{n-3}^2}{(1+\delta)\kappa_1^2}-\kappa_i\Big(\dfrac{2n-4}{3(n-3)}-\dfrac{(n-4)(n-5)}{6(n-3)^2}\Big)\tilde\sigma_{n-4}\tilde\sigma_{n-5}\nonumber\\
\leq&\dfrac{\tilde\sigma_{n-3}^2}{(1+\delta)\kappa_1^2}-\kappa_i\Big(\dfrac{2n-4}{3(n-3)}-\dfrac{(n-4)(n-5)}{6(n-3)^2}\Big)\dfrac{(n-3)^2(n-4)}{2\kappa_1^3}\tilde\sigma_{n-3}^2\nonumber\\
\leq&\Big(1-\frac{(n-2)(n-3)(n-4)}{3}+\frac{1}{12}(n-4)^2(n-5)\Big)\frac{\tilde{\sigma}_{n-3}^2}{(1+\delta)\kappa_1^2}\nonumber\\
\leq&\Big(1-\frac{(n-2)(n-3)(n-4)}{4}\Big)\frac{\tilde{\sigma}_{n-3}^2}{(1+\delta)\kappa_1^2}\nonumber\\
\leq&0\nonumber,
\end{align}
for $n\geq 5$. In the above second inequality, we have used
$$
\tilde\sigma_{n-4}^2\geq\dfrac{C_{n-3}^{n-4}C_{n-3}^{n-4}}{C_{n-3}^{n-3}C_{n-3}^{n-5}}\tilde\sigma_{n-3}\tilde\sigma_{n-5};\quad
\tilde\sigma_{n-4}\tilde\sigma_{n-5}\geq
\dfrac{C_{n-3}^{n-4}C_{n-3}^{n-5}}{C_{n-3}^{n-3}C_{n-3}^{n-6}}\tilde\sigma_{n-3}\tilde\sigma_{n-6},
$$
and in the third inequality of \eqref{s6.08}, we have used Lemma
\ref{lemm8}. Combing (\ref{s6.07}), (\ref{s6.08}) with
(\ref{s6.06}), we obtain $G4\leq 0$.

\end{proof}

Using Lemma \ref{lemm24}, Lemma \ref{lemm21} and Lemma \ref{lemm28},
we have:
\begin{lemm}\label{lemm29}
Assume $\kappa=(\kappa_1,\kappa_2,\cdots,\kappa_n)\in\Gamma_{n-2}$, $n\geq 5$ and $\kappa_1$ is the
maximum entry of $\kappa$. Suppose $\sigma_{n-2}(\kappa)$ has a positive upper bound $N$. Then for any given index $1\leq i\leq n$, if
$\kappa_i>\kappa_1-\sqrt{\kappa_1}/n$, we have
\begin{align*}
\dfrac{8}{9}\kappa_i^2\textbf{A}_{n-2;i}
 \geq&
\dfrac{2\sigma_{n-3}(\bar{\kappa})}{3\sigma_{n-5}(\bar{\kappa})}
\dsum_{j\neq
i}[2\sigma_{n-3}(\bar{\kappa}|j)\sigma_{n-5}(\bar{\kappa}|j)-2\sigma_{n-2}(\bar{\kappa}|j)\sigma_{n-6}(\bar{\kappa}|j)]\xi_j^2\\
&+\dfrac{2\sigma_{n-3}(\bar{\kappa})}{3\sigma_{n-5}(\bar{\kappa})}\dsum_{p,q\neq
i}\sigma_{n-3}(\bar{\kappa}|pq)\sigma_{n-5}(\bar{\kappa}|pq)\xi_p\xi_q,
\end{align*}
when $\kappa_1$ is sufficiently large.
\end{lemm}

\begin{proof} If $\kappa_1$ is sufficiently large, we clearly have
$$\frac{\sigma_{n-2}(\kappa)}{\kappa_i}\leq\frac{N}{\kappa_i}\leq \frac{1}{10}\kappa_1^2.$$
Thus, by Lemma \ref{lemm28}, if $\kappa_1$ is sufficiently large, we have
\begin{align*}
\dfrac{8}{9}\kappa_i^2>\dfrac{8}{9}(\kappa_1-\sqrt{\kappa_1}/n)^2>\dfrac{8}{9}\frac{9}{10}\kappa_1^2
=\dfrac{2}{3}\Big[\frac{1}{10}\kappa_1^2+\Big(1+\frac{1}{10}\Big)\kappa_1^2\Big]\geq
\dfrac{2\sigma_{n-3}(\bar{\kappa})}{3\sigma_{n-5}(\bar{\kappa})}.
\end{align*}
Using Lemma \ref{lemm20} and Lemma \ref{lemm21}, we obtain
\begin{align*}
\textbf{A}_{n-2;i}=&\dsum_{j\neq
i}\sigma_{n-4}(\bar{\kappa}^2|j)\xi_j^2+\dsum_{p,q\neq
i}\sigma_{n-4}(\bar{\kappa}^2|pq)\xi_p\xi_q\\
&+\dsum_{j\neq
i}[2\sigma_{n-3}(\bar{\kappa}|j)\sigma_{n-5}(\bar{\kappa}|j)-2\sigma_{n-2}(\bar{\kappa}|j)\sigma_{n-6}(\bar{\kappa}|j)]\xi_j^2\\
&+\dsum_{p,q\neq
i}\sigma_{n-3}(\bar{\kappa}|pq)\sigma_{n-5}(\bar{\kappa}|pq)\xi_p\xi_q\\
\geq&\dsum_{j\neq
i}[2\sigma_{n-3}(\bar{\kappa}|j)\sigma_{n-5}(\bar{\kappa}|j)-2\sigma_{n-2}(\bar{\kappa}|j)\sigma_{n-6}(\bar{\kappa}|j)]\xi_j^2\\
&+\dsum_{p,q\neq
i}\sigma_{n-3}(\bar{\kappa}|pq)\sigma_{n-5}(\bar{\kappa}|pq)\xi_p\xi_q.
\end{align*}
 Using the above two inequalities, we obtain our lemma.
\end{proof}

\begin{lemm}\label{lemm30}
Assume $\kappa=(\kappa_1,\kappa_2,\cdots,\kappa_n)\in\Gamma_{n-2}$,
$n\geq 5$,  $\kappa_1\geq\cdots\geq\kappa_n$ and
$\sigma_{n-2}(\kappa)$ has a upper bound $N$. Further suppose
$\kappa$ satisfies the condition of Case A, Case B1, or Case B2
defined in section 4. Then for any given indices $i,j$ satisfying $1\leq i,j\leq n,
$ and $j\neq i$, if $\kappa_i>\kappa_1-\sqrt{\kappa_1}/n$, we have
\begin{align*}
-\sigma_{n-2}(\bar{\kappa})\sigma_{n-6}(\bar{\kappa}|j)-\sigma_{n-2}(\bar{\kappa}|j)\sigma_{n-6}(\bar{\kappa}|j)+\dfrac{1}{40}\sigma_{n-3}(\bar{\kappa})\sigma_{n-5}(\bar{\kappa})\geq
0,
\end{align*}
when $\kappa_1$ is sufficiently large.
\end{lemm}

\begin{proof} We use the notation $L$ to denote the expression of the left hand side of the above inequality. Let's prove the non negativity of $L$ case by case.

{\bf Case A:} If the index $j=n$ or $n-1$, using $\kappa_n\leq \kappa_{n-1}\leq 0$, we have
$$\sigma_{n-2}(\bar{\kappa}|j)=\sigma_{n-2}(\kappa|ij)\leq 0.$$
Note
that $\sigma_{n-2}(\bar{\kappa})\leq 0$, we obtain $L\geq 0$
immediately.

If the index $j\leq n-2$, by Lemma \ref{lemm9} and $\bar{\kappa}\in\Gamma_{n-3}$, we have
\begin{align}\label{a6.10}
-(\kappa_n+\kappa_{n-1})<&\dfrac{2\sigma_{n-3}(\bar{\kappa}|n,n-1)}{\sigma_{n-4}(\bar{\kappa}|n,n-1)}.
\end{align}
Thus, we get
\begin{align*}
&-\sigma_{n-2}(\bar{\kappa})-\sigma_{n-2}(\bar{\kappa}|j)\\
=&-\kappa_n\kappa_{n-1}\sigma_{n-4}(\bar{\kappa}|n,n-1)-(\kappa_n+\kappa_{n-1})\sigma_{n-3}(\bar{\kappa}|n,n-1)\\
&-\kappa_n\kappa_{n-1}\sigma_{n-4}(\bar{\kappa}|n,n-1,j)\\
=&-\kappa_n\kappa_{n-1}[\sigma_{n-4}(\bar{\kappa}|n,n-1)+\sigma_{n-4}(\bar{\kappa}|n,n-1,j)]-(\kappa_n+\kappa_{n-1})\sigma_{n-3}(\bar{\kappa}|n,n-1)\\
\geq&-(\kappa_n+\kappa_{n-1})\Big[\dfrac{(\kappa_n+\kappa_{n-1})}{4}[\sigma_{n-4}(\bar{\kappa}|n,n-1)+\sigma_{n-4}(\bar{\kappa}|n,n-1,j)]\\
&+\sigma_{n-3}(\bar{\kappa}|n,n-1)\Big]\\
\geq&\dfrac{-(\kappa_n+\kappa_{n-1})}{4}\Big[-\dfrac{2\sigma_{n-3}(\bar{\kappa}|n,n-1)}{\sigma_{n-4}(\bar{\kappa}|n,n-1)}[\sigma_{n-4}(\bar{\kappa}|n,n-1)+\sigma_{n-4}(\bar{\kappa}|n,n-1,j)]\\
&+4\sigma_{n-3}(\bar{\kappa}|n,n-1)\Big]\\
=&\dfrac{-(\kappa_n+\kappa_{n-1})}{2}\sigma_{n-3}(\bar{\kappa}|n,n-1)\Big[2-\dfrac{\sigma_{n-4}(\bar{\kappa}|n,n-1)+\sigma_{n-4}(\bar{\kappa}|n,n-1,j)}{\sigma_{n-4}(\bar{\kappa}|n,n-1)}\Big]>0.
\end{align*}
Here, in the first inequality, we have used
$\kappa_n\kappa_{n-1}\leq\dfrac{(\kappa_n+\kappa_{n-1})^2}{4}$, in
the second inequality, we have used (\ref{a6.10}), and in the last
inequality, we have used
$$\sigma_{n-4}(\bar{\kappa}|n,n-1)=\kappa_j\sigma_{n-5}(\bar{\kappa}|n,n-1,j)+\sigma_{n-4}(\bar{\kappa}|n,n-1,j)>\sigma_{n-4}(\bar{\kappa}|n,n-1,j).$$
Thus $L\geq 0$.

\par

{\bf Case B1:} If the index $j<n$, we have
$\sigma_{n-2}(\bar{\kappa}|j)\leq 0$. Note that our assumption
$\sigma_{n-2}(\bar{\kappa})\leq 0$, then we obtain $L\geq 0$ immediately.
Thus, in the following, we only need to consider the non negativity of $L$ for $j=n$.

It is clear that
\begin{align}\label{L}
L=&[-\sigma_{n-2}(\bar{\kappa})-\sigma_{n-2}(\bar{\kappa}|n)]\sigma_{n-6}(\bar{\kappa}|n)+\dfrac{1}{40}\sigma_{n-3}(\bar{\kappa})\sigma_{n-5}(\bar{\kappa})\\
=&[-\kappa_n\sigma_{n-3}(\bar{\kappa}|n)-2\sigma_{n-2}(\bar{\kappa}|n)]\sigma_{n-6}(\bar{\kappa}|n)+\dfrac{1}{40}\sigma_{n-3}(\bar{\kappa})\sigma_{n-5}(\bar{\kappa}).\nonumber
\end{align}
For simplification purpose, we denote $\tilde\sigma_m=\sigma_m(\bar{\kappa}|n)$ in the proof of this lemma. Using the formula
\begin{align*}
\sigma_{n-2}(\kappa)=&\kappa_i\sigma_{n-3}(\bar{\kappa})+\sigma_{n-2}(\bar{\kappa})\\
=&\kappa_i(\kappa_n\tilde\sigma_{n-4}+\tilde\sigma_{n-3})+\kappa_n\tilde\sigma_{n-3}+\tilde\sigma_{n-2},
\end{align*}
we have
\begin{align*}
\kappa_n=\dfrac{\sigma_{n-2}(\kappa)-\tilde\sigma_{n-2}-\kappa_i\tilde\sigma_{n-3}}{\kappa_i\tilde\sigma_{n-4}+\tilde\sigma_{n-3}}.
\end{align*}
Inserting the above formula into \eqref{L}, we get
\begin{align}\label{s6.09}
&(\kappa_i\tilde\sigma_{n-4}+\tilde\sigma_{n-3})L\\
=&[\kappa_i\tilde\sigma_{n-3}^2-2\kappa_i\tilde\sigma_{n-2}\tilde\sigma_{n-4}-\tilde\sigma_{n-2}\tilde\sigma_{n-3}-\sigma_{n-2}(\kappa)\tilde\sigma_{n-3}]\tilde\sigma_{n-6}\nonumber\\
&+\dfrac{1}{40}\sigma_{n-3}(\bar{\kappa})\sigma_{n-5}(\bar{\kappa})(\kappa_i\tilde\sigma_{n-4}+\tilde\sigma_{n-3})\nonumber\\
=&[\kappa_i\sigma_{n-3}(\bar{\kappa}^2|n)-\tilde\sigma_{n-2}\tilde\sigma_{n-3}-\sigma_{n-2}(\kappa)\tilde\sigma_{n-3}]\tilde\sigma_{n-6}\nonumber\\
&+\dfrac{1}{40}\sigma_{n-3}(\bar{\kappa})\sigma_{n-5}(\bar{\kappa})(\kappa_i\tilde\sigma_{n-4}+\tilde\sigma_{n-3}).\nonumber
\end{align}
Here, in the second equality, we have used Lemma \ref{lemm20}.
\par
Let $\delta_0'=\dfrac{\delta_0}{2(n-2)^2}$. Firstly, we further suppose that
$-\kappa_n\geq\delta_0'\kappa_i$ or $\kappa_{n-1}\geq
\delta_0'\kappa_i$ holds, which implies
$-\kappa_n\geq\delta_0'\kappa_1/2$ or $\kappa_{n-1}\geq
\delta_0'\kappa_1/2$ holds, if $\kappa_1$ is sufficiently large.
 Using
$\kappa_{n-2}+\kappa_{n-1}+\kappa_{n}>0$ and
$\kappa_{n-2}\geq\kappa_{n-1}$, we have
\begin{eqnarray}\label{new6.3}
\kappa_{n-2}\geq\dfrac{\delta_0'}{4}\kappa_1,
\end{eqnarray}
if $\kappa_1$ is sufficiently large.
Then, we get
\begin{align}\label{new6.4}
&\kappa_i\sigma_{n-3}(\bar{\kappa}^2|n)-\tilde\sigma_{n-2}\tilde\sigma_{n-3}\\
=&\kappa_i\dsum_{j\neq i,n}\sigma_{n-3}(\kappa^2|ijn)-\sigma_{n-2}(\kappa|in)\dsum_{j\neq i,n}\sigma_{n-3}(\kappa|ijn)\nonumber\\
=&\dsum_{j\neq i,n}(\kappa_i-\kappa_j)\dfrac{\kappa_1^2\cdots\kappa_{n-1}^2}{\kappa_i^2\kappa_j^2}\nonumber\\
\geq&
-\dfrac{(n-2)\sqrt{\kappa_1}}{n}\dfrac{\kappa_1^2\cdots\kappa_{n-1}^2}{\kappa_i^2\kappa_j^2}
\geq-\kappa_1^{2n-5-1/2}.\nonumber
\end{align}
Here, in the first inequality, we have used
$\kappa_i\geq\kappa_1-\sqrt{\kappa_1}/n$ and $\kappa_j\leq\kappa_1$
and in the last inequality, we have used
$\kappa_{n-1}\leq\cdots\leq\kappa_2\leq\kappa_1$. It is easy to see
$\tilde\sigma_{n-6}\leq C_{n-2}^{n-6}\kappa_1^{n-6}$ and
$\tilde{\sigma}_{n-3}\leq (n-2)\kappa_1^{n-3}$, then combining with
\eqref{new6.4}, we get
\begin{align}\label{s6.10}
[\kappa_i\sigma_{n-3}(\bar{\kappa}^2|n)-\tilde\sigma_{n-2}\tilde\sigma_{n-3}-\sigma_{n-2}(\kappa)\tilde\sigma_{n-3}]\tilde\sigma_{n-6}\geq-2C_{n-2}^{n-6}\kappa_1^{3n-11-1/2},
\end{align}
if $\kappa_1$ is sufficiently large.
By Lemma \ref{lemm27} and the assumption that $-\kappa_n\geq\delta_0'\kappa_1/2$ or $\kappa_{n-1}\geq
\delta_0'\kappa_1/2$ holds, we have
$$\sigma_{n-3}(\bar{\kappa})\geq (\delta_1')^{n-3}\kappa_1^{n-3},$$ where $\delta'_1$ is some small positive constant only depending on $\delta_0'$. Using Lemma \ref{lemm10}, $\bar{\kappa}\in \Gamma_{n-3}$ and \eqref{new6.3}, we  have
$$\sigma_{n-5}(\bar{\kappa})\geq \kappa_2\cdots\kappa_{n-4}\geq (\delta_2')^{n-5}\kappa_1^{n-5},\text{ and }
\sigma_{n-4}(\bar{\kappa})\geq(\delta_2')^{n-4}\kappa_1^{n-4},$$
where $\delta_2'=\delta_0'/4$.
We select $$\delta_3'=\min\{\delta_1',\delta_2'\}.$$ Thus, we get
\begin{align}\label{s6.11}
\dfrac{1}{40}\sigma_{n-3}(\bar{\kappa})\sigma_{n-5}(\bar{\kappa})\kappa_i\tilde\sigma_{n-4}\geq\frac{(\delta_3')^{3n-11}}{40}\kappa_1^{3n-11},
\end{align}
if $\kappa_1$ is sufficiently large.
Inserting (\ref{s6.10}) and (\ref{s6.11}) into (\ref{s6.09}), we obtain
$L\geq 0$ if $\kappa_1$ is sufficiently large.
\par
Secondly, we consider the rest case, in which we suppose
 $-\kappa_n<\delta_0'\kappa_i$ and $\kappa_{n-1}<
\delta_0'\kappa_i$ both hold, which implies $i\leq n-2$. Then using \eqref{new6.4}, we have
\begin{align}\label{a6.14}
&\kappa_i\sigma_{n-3}(\bar{\kappa}^2|n)-\tilde\sigma_{n-2}\tilde\sigma_{n-3}\\
=&(\kappa_i-\kappa_{n-1})\dfrac{\kappa_1^2\cdots\kappa_{n-2}^2}{\kappa_i^2}
+\dsum_{j<n-1;j\neq i}(\kappa_i-\kappa_j)\dfrac{\kappa_1^2\cdots\kappa_{n-1}^2}{\kappa_i^2\kappa_j^2}\nonumber\\
\geq&(1-\delta_0')\kappa_i\dfrac{\kappa_1^2\cdots\kappa_{n-2}^2}{\kappa_i^2}
-\dfrac{(n-3)\sqrt{\kappa_1}}{n}\dfrac{\kappa_1^2\cdots\kappa_{n-2}^2}{\kappa_i^2}\nonumber\\
\geq&(1-2\delta_0')\kappa_i\dfrac{\kappa_1^2\cdots\kappa_{n-2}^2}{\kappa_i^2}\nonumber
\end{align}
if $\kappa_1$ is sufficiently large, where in the first inequality,
we have used $\kappa_{n-1}<\delta_0'\kappa_i,\kappa_i\geq
\kappa_1-\sqrt{\kappa_1}/n$ and $\kappa_j\leq \kappa_1,
\kappa_{n-1}\leq \kappa_j$. Using the assumption of Case B1, we
have
$$\kappa_i\sigma_{n-3}(\bar{\kappa})\geq(1+\delta_0)\sigma_{n-2}(\kappa)=(1+\delta_0)[\kappa_i\sigma_{n-3}(\bar{\kappa})+\sigma_{n-2}(\bar{\kappa})],$$
which implies
$$-\sigma_{n-2}(\bar\kappa)\geq \frac{\delta_0\kappa_i\sigma_{n-3}(\bar\kappa)}{1+\delta_0}\geq\delta_0\sigma_{n-2}(\kappa).$$
Then we get
$$\sigma_{n-2}(\kappa)\leq \frac{-\sigma_{n-2}(\bar{\kappa})}{\delta_0}\leq -\frac{n-2}{\delta_0}\frac{\kappa_1\cdots\kappa_{n-2}\kappa_{n}}{\kappa_i}
.$$
Note that
\begin{align}\label{a6.15}
\tilde\sigma_{n-3}\leq
(n-2)\dfrac{\kappa_1\cdots\kappa_{n-2}}{\kappa_i}.
\end{align}
 Using
(\ref{a6.14}) and the above two inequalities, we get
\begin{align}\label{s6.12}
\kappa_i\sigma_{n-3}(\bar{\kappa}^2|n)-\tilde\sigma_{n-2}\tilde\sigma_{n-3}-\sigma_{n-2}(\kappa)\tilde\sigma_{n-3}
\geq\Big[(1-2\delta_0')\kappa_i+\dfrac{(n-2)^2}{\delta_0}\kappa_n\Big]\dfrac{\kappa_1^2\cdots\kappa_{n-2}^2}{\kappa_i^2},
\end{align}
if $\kappa_1$ is sufficiently large.
Since
$-\kappa_n<\delta_0'\kappa_i=\dfrac{\delta_0}{2(n-2)^2}\kappa_i$,
(\ref{s6.12}) is nonnegative. Combing with (\ref{s6.09}), we obtain
$L\geq 0$, if $\kappa_1$ is sufficiently large.

\par
{\bf Case B2:} In view of the proof of  Case B1, the condition $\kappa_i\sigma_{n-3}(\bar{\kappa})\geq(1+\delta_0)\sigma_{n-2}(\kappa)$  is only used in the argument with the further assumption that
 $-\kappa_n<\delta_0'\kappa_i$ and $\kappa_{n-1}<
\delta_0'\kappa_i$ both hold. Thus, we only need to give a new proof for this case with the new condition
$$\kappa_1\cdots\kappa_{n-2}\geq2(n-2)\sigma_{n-2}(\kappa).$$
In fact, by
(\ref{a6.14}), (\ref{a6.15}) and the above inequality, we have
\begin{align}\label{s6.13}
&\kappa_i\sigma_{n-3}(\bar{\kappa}^2|n)-\tilde\sigma_{n-2}\tilde\sigma_{n-3}-\sigma_{n-2}(\kappa)\tilde\sigma_{n-3}\\
\geq&(1-2\delta_0')\kappa_i\dfrac{\kappa_1^2\cdots\kappa_{n-2}^2}{\kappa_i^2}-(n-2)\dfrac{\kappa_1\cdots\kappa_{n-2}}{\kappa_i}\sigma_{n-2}(\kappa)\nonumber\\
\geq&\dfrac{\kappa_1\cdots\kappa_{n-2}}{\kappa_i}[(1-2\delta_0')\kappa_1\cdots\kappa_{n-2}-(n-2)\sigma_{n-2}(\kappa)]\geq
0,\nonumber
\end{align}
if $\kappa_1$ is sufficiently large. Therefore we obtain $L\geq 0$, if $\kappa_1$ is sufficiently large.

\end{proof}

\begin{lemm}\label{lemm31} For any given index $1\leq i\leq n$, assume $\bar\kappa=(\kappa|i)\in\Gamma_{n-3}$.
For any vector $\xi=(\xi_1,\cdots,\xi_n)\in \mathbb{R}^n$, we define
a quadratic form:
\begin{align*}
\textbf{H}_{i}=&\dsum_{j\neq
i}\sigma_{n-3}(\bar{\kappa}^2|j)\xi_j^2+\dsum_{p,q\neq i;p\neq
q}\Big[\dfrac{2\sigma_{n-3}(\bar{\kappa})}{3\sigma_{n-5}(\bar{\kappa})}\sigma_{n-5}(\bar{\kappa}|pq)\sigma_{n-3}(\bar{\kappa}|pq)-\sigma_{n-3}^2(\bar{\kappa}|pq)\Big]\xi_p\xi_q.\nonumber
\end{align*}
Then we have
\begin{align}\label{s6.15}
\textbf{H}_{i}\geq&\dfrac{2\sigma_{n-3}(\bar{\kappa})}{3\sigma_{n-5}(\bar{\kappa})}\dsum_{j\neq
i}\Big[
\sigma_{n-5}(\bar{\kappa}|j)\sigma_{n-3}(\bar{\kappa}|j)-4\sigma_{n-6}(\bar{\kappa}|j)\sigma_{n-2}(\bar{\kappa}|j)
-\dfrac{4}{3}\sigma_{n-5}(\bar{\kappa})\sigma_{n-3}(\bar{\kappa})\Big]\xi_j^2.
\end{align}
\end{lemm}

\begin{proof}
We let
\begin{align*}
&\textbf{H1}_{i}\\
=&\dsum_{j\neq
i}\Big[\sigma_{n-3}(\bar{\kappa}^2|j)-\dfrac{2\sigma_{n-3}(\bar{\kappa})}{3\sigma_{n-5}(\bar{\kappa})}\dsum_{s\neq
j,i}\sigma_{n-5}(\bar{\kappa}|js)\sigma_{n-3}(\bar{\kappa}|js)
+\dfrac{\sigma_{n-3}^2(\bar{\kappa})}{\sigma_{n-5}^2(\bar{\kappa})}\dsum_{s\neq j,i}\dfrac{\sigma_{n-5}^2(\bar{\kappa}|js)}{9}\Big]\xi_j^2\\
&+\dsum_{p,q\neq i;p\neq
q}\Big[\dfrac{2\sigma_{n-3}(\bar{\kappa})}{3\sigma_{n-5}(\bar{\kappa})}\sigma_{n-5}(\bar{\kappa}|pq)\sigma_{n-3}(\bar{\kappa}|pq)
-\sigma_{n-3}(\bar{\kappa}^2|pq)-\dfrac{\sigma_{n-3}^2(\bar{\kappa})}{\sigma_{n-5}^2(\bar{\kappa})}\dfrac{\sigma_{n-5}^2(\bar{\kappa}|pq)}{9}\Big]\xi_p\xi_q,
\end{align*}
and
\begin{align*}
\textbf{H2}_{i}=&\dsum_{j\neq
i}\Big[\dfrac{2\sigma_{n-3}(\bar{\kappa})}{3\sigma_{n-5}(\bar{\kappa})}\dsum_{s\neq
j,i}\sigma_{n-5}(\bar{\kappa}|js)\sigma_{n-3}(\bar{\kappa}|js)
-\dfrac{\sigma_{n-3}^2(\bar{\kappa})}{\sigma_{n-5}^2(\bar{\kappa})}\dsum_{s\neq j,i}\dfrac{\sigma_{n-5}^2(\bar{\kappa}|js)}{9}\Big]\xi_j^2\\
&+\dsum_{p,q\neq i;p\neq
q}\dfrac{\sigma_{n-3}^2(\bar{\kappa})}{\sigma_{n-5}^2(\bar{\kappa})}\dfrac{\sigma_{n-5}^2(\bar{\kappa}|pq)}{9}
\xi_p\xi_q.
\end{align*}
Obviously, we have
\begin{align*}
\textbf{H}_{i}=\textbf{H1}_{i}+\textbf{H2}_{i}.
\end{align*}
Note that
$$
\sigma_{n-3}(\bar{\kappa}^2|pq)=\sigma_{n-3}^2(\bar{\kappa}|pq),
$$
and
$$
\sigma_{n-3}(\bar{\kappa}^2|j)=\dsum_{s\neq
j,i}\sigma_{n-3}(\bar{\kappa}^2|js)=\dsum_{s\neq
j,i}\sigma_{n-3}^2(\bar{\kappa}|js).
$$
Thus, using the above two identities, we can rewrite $\textbf{H1}_{i}$ to be
\begin{align*}
\textbf{H1}_{i}=&\dsum_{j\neq i}\Big[\dsum_{s\neq
j,i}\Big(\sigma_{n-3}(\bar{\kappa}|js)-\dfrac{\sigma_{n-3}(\bar{\kappa})}{3\sigma_{n-5}(\bar{\kappa})}\sigma_{n-5}(\bar{\kappa}|js)\Big)^2\Big]\xi_j^2\\
&-\dsum_{p,q\neq i;p\neq
q}\Big[\Big(\sigma_{n-3}(\bar{\kappa}|pq)-\dfrac{\sigma_{n-3}(\bar{\kappa})}{3\sigma_{n-5}(\bar{\kappa})}\sigma_{n-5}(\bar{\kappa}|pq)\Big)^2\Big]\xi_p\xi_q\\
=&\frac{1}{2}\sum_{p,q\neq i;p\neq
q}\Big[\Big(\sigma_{n-3}(\bar{\kappa}|pq)-\dfrac{\sigma_{n-3}(\bar{\kappa})}{3\sigma_{n-5}(\bar{\kappa})}\sigma_{n-5}(\bar{\kappa}|pq)\Big)^2\Big](\xi_p-\xi_q)^2
\geq 0.
\end{align*}
On the other hand, it is clear that
\begin{align*}
\textbf{H2}_{i} =&\dsum_{j\neq
i}\Big[\dfrac{\sigma_{n-3}^2(\bar{\kappa})}{\sigma_{n-5}^2(\bar{\kappa})}\dsum_{s\neq
j,i}\dfrac{\sigma_{n-5}^2(\bar{\kappa}|js)}{9}\Big]\xi_j^2+\dsum_{p,q\neq
i;p\neq q}\Big[\dfrac{\sigma_{n-3}^2(\bar{\kappa})}{\sigma_{n-5}^2(\bar{\kappa})}\dfrac{\sigma_{n-5}^2(\bar{\kappa}|pq)}{9}\Big]\xi_p\xi_q\\
&+\dsum_{j\neq
i}\Big[\dfrac{2\sigma_{n-3}(\bar{\kappa})}{3\sigma_{n-5}(\bar{\kappa})}\dsum_{s\neq
j,i}\sigma_{n-5}(\bar{\kappa}|js)\sigma_{n-3}(\bar{\kappa}|js)
-\dfrac{2\sigma_{n-3}^2(\bar{\kappa})}{9\sigma_{n-5}^2(\bar{\kappa})}\dsum_{s\neq
j,i}\sigma_{n-5}^2(\bar{\kappa}|js)\Big]\xi_j^2\\
=&\frac{1}{2}\sum_{p,q\neq
i;p\neq q}\Big[\dfrac{\sigma_{n-3}^2(\bar{\kappa})}{\sigma_{n-5}^2(\bar{\kappa})}\dfrac{\sigma_{n-5}^2(\bar{\kappa}|pq)}{9}\Big](\xi_p+\xi_q)^2\\
&+\sum_{j\neq
i}\Big[\dfrac{2\sigma_{n-3}(\bar{\kappa})}{3\sigma_{n-5}(\bar{\kappa})}\dsum_{s\neq
j,i}\sigma_{n-5}(\bar{\kappa}|js)\sigma_{n-3}(\bar{\kappa}|js)
-\dfrac{2\sigma_{n-3}^2(\bar{\kappa})}{9\sigma_{n-5}^2(\bar{\kappa})}\dsum_{s\neq
j,i}\sigma_{n-5}^2(\bar{\kappa}|js)\Big]\xi_j^2\\
\geq& \frac{2\sigma_{n-3}(\bar{\kappa})}{3\sigma_{n-5}(\bar{\kappa})}\sum_{j\neq
i}\Big[\dsum_{s\neq
j,i}\sigma_{n-5}(\bar{\kappa}|js)\sigma_{n-3}(\bar{\kappa}|js)
-\dfrac{\sigma_{n-3}(\bar{\kappa})}{3\sigma_{n-5}(\bar{\kappa})}\dsum_{s\neq
j,i}\sigma_{n-5}^2(\bar{\kappa}|js)\Big]\xi_j^2.
\end{align*}

Using Lemma \ref{lemm22} with $m=n-2, s=3$ and Lemma \ref{lemm26}
with $m=n-1$ respectively, we have
\begin{align*}
\dsum_{s\neq
j,i}\sigma_{n-5}(\bar{\kappa}|js)\sigma_{n-3}(\bar{\kappa}|js)=&\sigma_{n-5}(\bar{\kappa}|j)\sigma_{n-3}(\bar{\kappa}|j)-4\sigma_{n-6}(\bar{\kappa}|j)\sigma_{n-2}(\bar{\kappa}|j),\\
\dsum_{s\neq j,i}\sigma_{n-5}^2(\bar{\kappa}|js)\leq &4
\sigma_{n-5}^2(\bar{\kappa}).
\end{align*}
Thus, inserting the above two formulae into the last expression of $\textbf{H2}_i$,  we obtain (\ref{s6.15}).
\end{proof}

Now, we are in the position to prove Theorem \ref{theorem}.
\par

\noindent {\bf Proof of Theorem \ref{theorem}:} Firstly, let's prove (\ref{s6.01}).
 Using Lemma \ref{lemm20}, we
have
\begin{align}\label{a6.16}
&\kappa_j^2\sigma_{n-4}^2(\bar{\kappa}|j)-2\sigma_{n-4}(\bar{\kappa}|j)\sigma_{n-2}(\bar{\kappa}|j)\\
=&[\sigma_{n-3}(\bar{\kappa})-\sigma_{n-3}(\bar{\kappa}|j)]^2-2\sigma_{n-4}(\bar{\kappa}|j)\sigma_{n-2}(\bar{\kappa}|j)\nonumber\\
=&\sigma_{n-3}^2(\bar{\kappa})-2\sigma_{n-3}(\bar{\kappa})\sigma_{n-3}(\bar{\kappa}|j)+\sigma_{n-3}^2(\bar{\kappa}|j)-2\sigma_{n-4}(\bar{\kappa}|j)\sigma_{n-2}(\bar{\kappa}|j)\nonumber\\
=&\sigma_{n-3}^2(\bar{\kappa})-2\sigma_{n-3}(\bar{\kappa})\sigma_{n-3}(\bar{\kappa}|j)+\sigma_{n-3}(\bar{\kappa}^2|j).\nonumber
\end{align}

By Lemma \ref{lemm29}, we have
\begin{align}\label{s6.14}
&\dfrac{8\kappa_i^2}{9}\textbf{A}_{n-2;i}+\textbf{C}_{n-2;i}\\
\geq&\dfrac{2\sigma_{n-3}(\bar{\kappa})}{3\sigma_{n-5}(\bar{\kappa})}\dsum_{j\neq
i}[2\sigma_{n-5}(\bar{\kappa}|j)\sigma_{n-3}(\bar{\kappa}|j)-2\sigma_{n-6}(\bar{\kappa}|j)\sigma_{n-2}(\bar{\kappa}|j)]\xi_j^2\nonumber\\
&+\dfrac{2\sigma_{n-3}(\bar{\kappa})}{3\sigma_{n-5}(\bar{\kappa})}\dsum_{p,q\neq
i;p\neq q}\sigma_{n-5}(\bar{\kappa}|pq)\sigma_{n-3}(\bar{\kappa}|pq)\xi_p\xi_q\nonumber\\
&+\dsum_{j\neq
i}[\sigma_{n-3}^2(\bar{\kappa})-2\sigma_{n-3}(\bar{\kappa})\sigma_{n-3}(\bar{\kappa}|j)+\sigma_{n-3}(\bar{\kappa}^2|j)]\xi_j^2
-\dsum_{p,q\neq
i;p\neq q}\sigma_{n-3}^2(\bar{\kappa}|pq) \xi_p\xi_q\nonumber\\
=&\dfrac{2\sigma_{n-3}(\bar{\kappa})}{3\sigma_{n-5}(\bar{\kappa})}\dsum_{j\neq
i}[2\sigma_{n-5}(\bar{\kappa}|j)\sigma_{n-3}(\bar{\kappa}|j)-2\sigma_{n-6}(\bar{\kappa}|j)\sigma_{n-2}(\bar{\kappa}|j)]\xi_j^2\nonumber\\
&+\dsum_{j\neq
i}[\sigma_{n-3}^2(\bar{\kappa})-2\sigma_{n-3}(\bar{\kappa})\sigma_{n-3}(\bar{\kappa}|j)]\xi_j^2+\textbf{H}_{i},
\nonumber
\end{align}
when $\kappa_1$ is sufficiently large.
Here, in the first inequality, we have used
(\ref{a6.16}), the definition of $\textbf{C}_{n-2;i}$ and $\sigma_{n-2}(\kappa|ipq)=0$. Then inserting (\ref{s6.15}) into (\ref{s6.14}), we
obtain

\begin{align*}
&\dfrac{8\kappa_i^2}{9}\textbf{A}_{n-2;i}+\textbf{C}_{n-2;i}\\
\geq&\dfrac{2\sigma_{n-3}(\bar{\kappa})}{3\sigma_{n-5}(\bar{\kappa})}\dsum_{j\neq
i}[2\sigma_{n-5}(\bar{\kappa}|j)\sigma_{n-3}(\bar{\kappa}|j)-2\sigma_{n-6}(\bar{\kappa}|j)\sigma_{n-2}(\bar{\kappa}|j)]\xi_j^2\nonumber\\
&+\dsum_{j\neq
i}[\sigma_{n-3}^2(\bar{\kappa})-2\sigma_{n-3}(\bar{\kappa})\sigma_{n-3}(\bar{\kappa}|j)]\xi_j^2\\
&+\dfrac{2\sigma_{n-3}(\bar{\kappa})}{3\sigma_{n-5}(\bar{\kappa})}\dsum_{j\neq
i}\Big[
\sigma_{n-5}(\bar{\kappa}|j)\sigma_{n-3}(\bar{\kappa}|j)-4\sigma_{n-6}(\bar{\kappa}|j)\sigma_{n-2}(\bar{\kappa}|j)
-\dfrac{4}{3}\sigma_{n-5}(\bar{\kappa})\sigma_{n-3}(\bar{\kappa})\Big]\xi_j^2
\nonumber\\
=&\dfrac{\sigma_{n-3}(\bar{\kappa})}{\sigma_{n-5}(\bar{\kappa})}\dsum_{j\neq
i}\Big[
2\sigma_{n-5}(\bar{\kappa}|j)\sigma_{n-3}(\bar{\kappa}|j)-4\sigma_{n-6}(\bar{\kappa}|j)\sigma_{n-2}(\bar{\kappa}|j)\\
&
-2\sigma_{n-5}(\bar{\kappa})\sigma_{n-3}(\bar{\kappa}|j)+\dfrac{1}{9}\sigma_{n-5}(\bar{\kappa})\sigma_{n-3}(\bar{\kappa})
\Big]\xi_j^2\\
=&\dfrac{\sigma_{n-3}(\bar{\kappa})}{\sigma_{n-5}(\bar{\kappa})}\dsum_{j\neq
i}\Big[
-2\sigma_{n-6}(\bar{\kappa}|j)\sigma_{n-2}(\bar{\kappa})-2\sigma_{n-6}(\bar{\kappa}|j)\sigma_{n-2}(\bar{\kappa}|j)+\dfrac{1}{9}\sigma_{n-5}(\bar{\kappa})\sigma_{n-3}(\bar{\kappa})
\Big]\xi_j^2.
\end{align*}
Using Lemma \ref{lemm30}, we know, if $\kappa_1$ is sufficiently large,
$$-2\sigma_{n-6}(\bar{\kappa}|j)\sigma_{n-2}(\bar{\kappa})
-2\sigma_{n-6}(\bar{\kappa}|j)\sigma_{n-2}(\bar{\kappa}|j)+\dfrac{1}{20}\sigma_{n-5}(\bar{\kappa})\sigma_{n-3}(\bar{\kappa})\geq
0,$$ which implies (\ref{s6.01}).\\

Secondly, let's prove  (\ref{s6.02}). The inequality (\ref{s6.02}) is equivalent to
\begin{align}\label{s6.16}
&[\kappa_i
K\sigma_{n-3}(\bar{\kappa})-1]\Big[\dfrac{\kappa_i^2}{9}\textbf{A}_{n-2;i}+\sigma_{n-2}(\kappa)\textbf{B}_{n-2;i}+\dfrac{1}{20}\dsum_{j\neq
i} \sigma_{n-3}^2(\bar{\kappa})\xi_j^2\Big]-\textbf{D}_{n-2;i}\geq
0.
\end{align}

Since $\kappa_i\geq \kappa_1-\sqrt{\kappa_1}/n$, we have
$$\Big[\frac{1}{2}\kappa_i K\sigma_{n-3}(\bar{\kappa})-1\Big]\geq \Big(\frac{K}{4}\kappa_1 \sigma_{n-3}(\kappa|1)-1\Big)\geq \Big(\frac{K\theta}{4}N_0-1\Big),$$ if $\kappa_1$ is sufficiently large. Here $\theta$ is the constant defined in Lemma \ref{lemm11}.
Thus, $[\kappa_i K\sigma_{n-3}(\bar{\kappa})-1]$ can be sufficiently
large if the constant $K$ is large enough. Lemma \ref{lemm24} and
Lemma \ref{lemm25} imply that quadratic forms $\textbf{A}_{n-2;i}$,
$\textbf{B}_{n-2;i}$ are nonnegative respectively. Thus, if $K$ is sufficiently large, using $\kappa\in\Gamma_{n-2}$, we have
\begin{align*}
[\kappa_i
K\sigma_{n-3}(\bar{\kappa})-1]\dfrac{\kappa_i^2}{9}\textbf{A}_{n-2;i}
\geq &\dfrac{8\kappa_i^2}{9}\textbf{A}_{n-2;i},
\end{align*}
\begin{align*}
[\kappa_i
K\sigma_{n-3}(\bar{\kappa})-1]\sigma_{n-2}(\kappa)\textbf{B}_{n-2;i}&\geq
\dfrac{1}{2}\kappa_i
K\sigma_{n-3}(\bar{\kappa})\sigma_{n-2}(\kappa)\textbf{B}_{n-2;i}\geq
\kappa_i\sigma_{n-3}(\bar{\kappa})\textbf{B}_{n-2;i}\\
&=[\sigma_{n-2}(\kappa)-\sigma_{n-2}(\bar{\kappa})]\textbf{B}_{n-2;i}\geq
- \sigma_{n-2}(\bar{\kappa})\textbf{B}_{n-2;i},
\end{align*}
and
$$\dfrac{1}{20}[\kappa_i
K\sigma_{n-3}(\bar{\kappa})-1]\dsum_{j\neq i}
\sigma_{n-3}^2(\bar{\kappa})\xi_j^2\geq \dsum_{j\neq i}
\sigma_{n-3}^2(\bar{\kappa})\xi_j^2.$$

Therefore, in view of (\ref{s6.16}), we only need to show
\begin{align}\label{s6.17}
&\dfrac{8\kappa_i^2}{9}\textbf{A}_{n-2;i}-\sigma_{n-2}(\bar{\kappa})\textbf{B}_{n-2;i}+\dsum_{j\neq
i} \sigma_{n-3}^2(\bar{\kappa})\xi_j^2-\textbf{D}_{n-2;i}\geq 0,
\end{align}
if $\kappa_1$ is sufficiently large.

Using Lemma \ref{lemm20}, we have the equalities:

\begin{align*}
&-2\sigma_{n-2}(\bar{\kappa})\sigma_{n-4}(\bar{\kappa}|j)+\sigma_{n-3}^2(\bar{\kappa})-\sigma_{n-3}^2(\bar{\kappa}|j)\\
=&-2[\kappa_j\sigma_{n-3}(\bar{\kappa}|j)+\sigma_{n-2}(\bar{\kappa}|j)]\sigma_{n-4}(\bar{\kappa}|j)+\sigma_{n-3}^2(\bar{\kappa})-\sigma_{n-3}^2(\bar{\kappa}|j)\\
=&-2[\kappa_j\sigma_{n-4}(\bar{\kappa}|j)]\sigma_{n-3}(\bar{\kappa}|j)-2\sigma_{n-2}(\bar{\kappa}|j)\sigma_{n-4}(\bar{\kappa}|j)+\sigma_{n-3}^2(\bar{\kappa})-\sigma_{n-3}^2(\bar{\kappa}|j)\\
=&\sigma_{n-3}^2(\bar{\kappa}|j)-2\sigma_{n-4}(\bar{\kappa}|j)\sigma_{n-2}(\bar{\kappa}|j)-2\sigma_{n-3}(\bar{\kappa})\sigma_{n-3}(\bar{\kappa}|j)+\sigma_{n-3}^2(\bar{\kappa})\\
=&[\sigma_{n-3}(\bar{\kappa})-\sigma_{n-3}(\bar{\kappa}|j)]^2-2\sigma_{n-4}(\bar{\kappa}|j)\sigma_{n-2}(\bar{\kappa}|j)\\
=&\kappa_j^2\sigma_{n-4}^2(\bar{\kappa}|j)-2\sigma_{n-4}(\bar{\kappa}|j)\sigma_{n-2}(\bar{\kappa}|j),
\end{align*}
and
\begin{align*}
&\sigma_{n-2}(\bar{\kappa})\sigma_{n-4}(\bar{\kappa}|pq)-\sigma_{n-3}(\bar{\kappa}|p)\sigma_{n-3}(\bar{\kappa}|q)\\
=&[\kappa_p\kappa_q\sigma_{n-4}(\bar{\kappa}|pq)+(\kappa_p+\kappa_q)\sigma_{n-3}(\bar{\kappa}|pq)]\sigma_{n-4}(\bar{\kappa}|pq)\\
&-[\kappa_q\sigma_{n-4}(\bar{\kappa}|pq)+\sigma_{n-3}(\bar{\kappa}|pq)][\kappa_p\sigma_{n-4}(\bar{\kappa}|pq)+\sigma_{n-3}(\bar{\kappa}|pq)]\\
=&-\sigma_{n-3}^2(\bar{\kappa}|pq).
\end{align*}
Using the above two identities, we obtain
\begin{align*}
& \text{ LHS of }(\ref{s6.17})\\
=&\dfrac{8\kappa_i^2}{9}\textbf{A}_{n-2;i}-\sigma_{n-2}(\bar{\kappa})\Big[\sum_{j\neq
i}2\sigma_{n-4}(\bar{\kappa}|j)\xi_j^2-\sum_{p,q\neq
i}\sigma_{n-4}(\bar{\kappa}|pq)\xi_p\xi_q\Big]+\dsum_{j\neq
i}\sigma_{n-3}^2(\bar{\kappa})\xi_j^2\\
&-\Big[\sum_{j\neq
i}\sigma_{n-3}^2(\bar{\kappa}|j)\xi_j^2+\dsum_{p,q\neq
i}\sigma_{n-3}(\bar{\kappa}|p)\sigma_{n-3}(\bar{\kappa}|q)\xi_p\xi_q\Big]\\
=&\dfrac{8\kappa_i^2}{9}\textbf{A}_{n-2;i}+\dsum_{j\neq i}[
\kappa_j^2\sigma_{n-4}^2(\bar{\kappa}|j)-2\sigma_{n-4}(\bar{\kappa}|j)\sigma_{n-2}(\bar{\kappa}|j)]\xi_j^2
-\dsum_{p,q\neq i}\sigma_{n-3}^2(\bar{\kappa}|pq)\xi_p\xi_q\\
=&\dfrac{8\kappa_i^2}{9}\textbf{A}_{n-2;i}+\textbf{C}_{n-2;i},
\end{align*}
which is exactly the left hand side of (\ref{s6.01}). Thus, \eqref{s6.01} implies \eqref{s6.02}.

\section{The Proof of Theorem \ref{maintheo} for cases B3 and C}
This section will focus on  Cases B3 and C. We will use the idea and techniques developing in \cite{GRW,LRW}. Since the argument of the following lemma is similar to Lemma 9 in \cite{LRW}, we will omit some details in its proof.

\begin{lemm}\label{lem32} For any given index $1\leq i\leq n-3$ with $n\geq 5$,
assume $\kappa=(\kappa_1,\cdots,\kappa_n)\in\Gamma_{n-2}$,
$\kappa_1\geq\cdots\geq\kappa_n$,
$\kappa_i>\kappa_1-\sqrt{\kappa_1}/n$ and
$$\kappa_i\sigma_{n-3}(\kappa|i)\leq
(1+\delta_0)\sigma_{n-2}(\kappa),$$ where
$\delta_0=\dfrac{1}{32n(n-2)}$ defined in section 4. Further suppose that $\sigma_{n-2}(\kappa)$ has positive upper and lower bounds
$N_0\leq \sigma_{n-2}(\kappa)\leq N$. For $i\leq \mu \leq n-3$,
suppose there exists some positive constant $\delta<1$, satisfying
$\kappa_{\mu}/\kappa_1\geq\delta$. Then there exits another
sufficiently small positive constant $\delta'$ only depending on
$\delta$, such that, for any $\xi=(\xi_1,\xi_2,\cdots,\xi_n)\in\mathbb{R}^n$, if $\kappa_{\mu+1}/\kappa_1\leq\delta'$, we have
\begin{equation}\label{s7.1}
\kappa_i\Big[K\Big(\sum_j\sigma_{n-2}^{jj}(\kappa)\xi_j\Big)^2-\sigma_{n-2}^{pp,qq}(\kappa)\xi_p\xi_q\Big]-\sigma_{n-2}^{ii}(\kappa)\xi_i^2+\sum_{j\neq
i}a_j\xi_j^2\geq 0,
\end{equation}
when
$\kappa_1$ and $K$ are sufficiently large.
Here $a_j$ is defined by \eqref{aj}.
\end{lemm}
\begin{proof}
For the sake of simplification, we
let
$$\sigma_s=\sigma_s(\kappa),\ \  \sigma_s^{pp}=\sigma_s^{pp}(\kappa), \ \ \sigma_s^{pp,qq}=\sigma_s^{pp,qq}(\kappa), s=1,2\cdots,n,$$ in the proof. We further denote
$$F=K\Big(\sum_j\sigma_{n-2}^{jj}\xi_j\Big)^2-\sigma_{n-2}^{pp,qq}\xi_p\xi_q.$$

Same as (3.20) in \cite{LRW}, by Lemma \ref{Guan}, we have,
\begin{align}\label{s7.2}
&F\\
\geq&\frac{\sigma_{n-2}}{\sigma_{\mu}^2}\Big[\Big(1+\frac{\alpha}{2}\Big)\sum_{a}(\sigma_{\mu}^{aa}\xi_a)^2+\frac{\alpha}{2}\sum_{a\neq
b}\sigma_{\mu}^{aa}\sigma_{\mu}^{bb}\xi_a\xi_b+\sum_{a\neq
b}(\sigma_{\mu}^{aa}\sigma_{\mu}^{bb}-\sigma_{\mu}\sigma_{\mu}^{aa,bb})\xi_a\xi_b\Big].\nonumber
\end{align}
Here we let $\alpha=\dfrac{1}{2n(n-2)}<1$. For $\mu=1$, which implies that $i=1$, same as (3.21) in \cite{LRW}, we have
\begin{align}\label{s7.3}
\Big(1+\frac{\alpha}{2}\Big)\sum_{a,b} \xi_a\xi_b
\geq&\Big(1+\frac{\alpha}{4}\Big)\xi_i^2-C_{\alpha}\sum_{a\neq i}\xi_a^2,
\end{align}
where $C_{\alpha}$ is some positive constant only depending on $\alpha$.
Then, we get
\begin{align}\label{s7.4}
F\geq&\frac{\sigma_{n-2}}{\sigma_1^2}\Big(1+\frac{\alpha}{4}\Big)\xi_i^2
-\frac{C_{\alpha}}{\sigma_1^2}\sum_{a\neq i}\xi_a^2\\
\geq&\frac{\sigma_{n-2}^{ii}}{\kappa_i(1+\delta_0)(1+\sum_{j\neq
i}\kappa_j/\kappa_i)^2}\Big(1+\frac{\alpha}{4}\Big)\xi_i^2
-\frac{C_{\alpha}}{\sigma_1^2}\sum_{a\neq i}\xi_a^2\nonumber\\
\geq&\dfrac{\sigma_{n-2}^{ii}}{\kappa_i}\xi_i^2 -\frac{C_{\alpha}
}{\sigma_1^2}\sum_{a\neq i}\xi_a^2.\nonumber
\end{align}
 In the last two inequalities we have used $(1+\delta_0)\sigma_{n-2}\geq \kappa_i\sigma_{n-2}^{ii}$ and let $\delta'$ be sufficiently small to satisfy
$$ 1+\frac{\alpha}{4}\geq (1+\delta_0)[1+(n-1)\delta']^2.$$
For $\mu>1$, using Lemma \ref{lemm10} and
$|\kappa_{n-1}|+|\kappa_n|\leq 3\kappa_{n-2}$,
according to the argument of (3.25)-(3.27) in \cite{LRW}, we get
 for indices $1\leq a,b\leq \mu$,
\begin{equation}\label{s7.7}
\sigma_{\mu-1}^2(\kappa|ab)-\sigma_{\mu}(\kappa|ab)\sigma_{\mu-2}(\kappa|ab)\leq
C_1\Big(\frac{\kappa_{\mu+1}}{\kappa_b}\sigma_{\mu}^{aa}\Big)^2.
\end{equation}
Here $C_1$ is an absolutely constant. Same as (3.28) in \cite{LRW}, by (\ref{s7.7}), for any undetermined
 positive constant $\epsilon$, we have
\begin{align}\label{s7.8}
&\sum_{a\neq b; a,b\leq
\mu}(\sigma_{\mu}^{aa}\sigma_{\mu}^{bb}-\sigma_{\mu}\sigma_{\mu}^{aa,bb})\xi_a\xi_b
\geq -\epsilon\sum_{a\leq\mu}(\sigma_{\mu}^{aa}\xi_a)^2
\end{align}
if  $\delta'$ is sufficiently small and its upper bound depends on $\epsilon$ and $\delta$. Same as (3.30) and
(3.31) in \cite{LRW}, by (\ref{s7.7}), we also have
\begin{align}\label{s7.10}
&2\sum_{a\leq \mu; b> \mu}\Big[\Big(1+\dfrac{\al}{2}\Big)\sigma_{\mu}^{aa}\sigma_{\mu}^{bb}-\sigma_{\mu}\sigma_{\mu}^{aa,bb}\Big]\xi_a\xi_b\\
\geq&-\epsilon\sum_{a\leq \mu;
b>\mu}(\sigma_{\mu}^{aa}\xi_a)^2-\frac{2}{\epsilon}\sum_{a\leq \mu;
b>\mu}(\sigma_{\mu}^{bb}\xi_b)^2\nonumber,
\end{align}
and
\begin{align}\label{s7.11}
\sum_{a\neq b; a,b> \mu}\Big[\Big(1+\frac{\al}{2}\Big)\sigma_{\mu}^{aa}\sigma_{\mu}^{bb}-\sigma_{\mu}\sigma_{\mu}^{aa,bb}\Big]\xi_a\xi_b
\geq& -\sum_{a\neq b; a,b>\mu}\Big(1+\dfrac{\al}{2}\Big)(\sigma_{\mu}^{aa}\xi_a)^2.
\end{align}
On the other hand, we have
\begin{align}\label{s7.9}
&\frac{\alpha}{2}\sum_{a\neq
b;a,b\leq\mu}\sigma_{\mu}^{aa}\sigma_{\mu}^{bb}\xi_a\xi_b\\
=&\alpha\sum_{a\neq
i;a\leq\mu}\sigma_{\mu}^{ii}\sigma_{\mu}^{aa}\xi_i\xi_a+\frac{\alpha}{2}\sum_{a\neq
b;a,b\neq i;a,b\leq\mu}\sigma_{\mu}^{aa}\sigma_{\mu}^{bb}\xi_a\xi_b\nonumber\\
\geq&-\dfrac{\alpha}{4}(\sigma_{\mu}^{ii}\xi_i)^2-n\al\sum_{a\leq\mu;a\neq
i}(\sigma_{\mu}^{aa}\xi_a)^2.\nonumber
\end{align}
Hence, combing (\ref{s7.2}), (\ref{s7.8}),
(\ref{s7.10}),(\ref{s7.11}) with (\ref{s7.9}),  we get
\begin{align*}
F
\geq&\frac{\sigma_{n-2}}{\sigma_{\mu}^2}\Big[\Big(1+\dfrac{\al}{4}-n\epsilon\Big)(\sigma_{\mu}^{ii}\xi_i)^2+(1-n\al-n\epsilon)\sum_{a\leq
\mu;a\neq
i}(\sigma_{\mu}^{aa}\xi_a)^2-C_{\epsilon}\sum_{a>\mu}(\sigma_{\mu}^{aa}\xi_a)^2\Big],
\end{align*}
where $C_{\epsilon}$ is a constant only depending on $\epsilon$, $\alpha$. For $a>\mu$, by Lemma \ref{lemm10} and
$|\kappa_{n-1}|+|\kappa_n|\leq 3\kappa_{n-2}$, we have
\begin{align}\label{s7.12}
\sigma_{\mu}^{aa}\leq C_2\kappa_1\cdots\kappa_{\mu-1},\ \ \text{ and
} \sigma_{\mu}\geq \kappa_1\cdots\kappa_{\mu}.
\end{align}
For $a\leq \mu$, again using
$|\kappa_{n-1}|+|\kappa_n|\leq 3\kappa_{n-2}$, we have
\begin{eqnarray}\label{new7}
\sigma_{\mu}(\kappa|a)\leq C_3\frac{\kappa_1\cdots\kappa_{\mu+1}}{\kappa_a}.
\end{eqnarray}
Here, $C_2,C_3$ are two absolutely constants.
We let $\epsilon$ satisfy
$
\al/8\geq n\epsilon,
$
which implies $1-n\al-n\epsilon\geq 0$.
Then, we have
\begin{align}\label{s7.13}
F
\geq&\frac{\kappa_i\sigma_{n-2}^{ii}}{(1+\delta_0)\sigma_{\mu}^2}\Big(1+\dfrac{\al}{4}-n\epsilon\Big)(\sigma_{\mu}^{ii}\xi_i)^2
-\frac{\sigma_{n-2}C_{\epsilon}}{\sigma_{\mu}^2}\sum_{a>\mu}(\sigma_{\mu}^{aa}\xi_a)^2\\
\geq&\frac{\Big(1+\dfrac{\al}{4}-n\epsilon\Big)\sigma_{n-2}^{ii}}{(1+\delta_0)\kappa_i}\Big(\dfrac{\kappa_i\sigma_{\mu}^{ii}}{\sigma_{\mu}}\xi_i\Big)^2
-\frac{NC_{\epsilon}}{\sigma_{\mu}^2}\sum_{a>\mu}(\sigma_{\mu}^{aa}\xi_a)^2\nonumber\\
\geq&\frac{\Big(1+\dfrac{\al}{4}-n\epsilon\Big)\sigma_{n-2}^{ii}}{(1+\delta_0)\kappa_i}\Big(1-\dfrac{C_3\kappa_{\mu+1}}{\kappa_i}\Big)^2\xi_i^2
-\frac{NC_{\epsilon}}{\sigma_{\mu}^2}\sum_{a>\mu}(\sigma_{\mu}^{aa}\xi_a)^2\nonumber\\
\geq&\frac{\sigma_{n-2}^{ii}}{\kappa_i}\xi_i^2-\frac{NC_{\epsilon}}{\sigma_{\mu}^2}\sum_{a>\mu}(\sigma_{\mu}^{aa}\xi_a)^2.\nonumber
\end{align}
Here, in the third inequality, we have used $\sigma_{\mu}=\kappa_i\sigma_{\mu}^{ii}+\sigma_{\mu}(\kappa|i)$, \eqref{s7.12}, \eqref{new7}, and in the last inequality, we select constants
$\delta'$ and $\epsilon$ satisfying
\begin{align*}
 \delta'C_3\leq
\epsilon\delta ,\quad \Big(1+\dfrac{\al}{8}\Big)(1-\epsilon)^2\geq
1+\frac{\al}{16}=1+\delta_0.
\end{align*}

In view of Lemma \ref{lemm9}, we have $-\kappa_n\leq2\kappa_1/(n-2)$, if $\kappa_n\leq 0$. Thus, we get, for $j=n$ or $n-1$,
$$a_j=\sigma^{jj}_{n-2}+(\kappa_i+\kappa_j)\sigma^{jj,ii}_{n-2}\geq \sigma_{n-2}^{jj},$$ if $\kappa_1$ is sufficiently large. It is obvious that
$a_j\geq\sigma_{n-2}^{jj}$ for any index $j<n-1$.

Thus, at last, by (\ref{s7.12}) and Lemma \ref{lemm11},
 if the index $j>\mu$, we have
\begin{align}\label{s7.14}
a_j-\kappa_i\frac{NC_{\epsilon}}{\sigma_{\mu}^2}(\sigma_{\mu}^{jj})^2\geq&\sigma_{n-2}^{jj}-\frac{\kappa_iNC_{\epsilon}C_2^2}{\kappa_{\mu}^2}
\geq\sigma_{n-3}(\kappa|\mu+1)-\frac{NC_{\epsilon}C_2^2}{\delta^2\kappa_{1}}\\
\geq&\dfrac{\theta\sigma_{n-2}}{\kappa_{\mu+1}}-\frac{NC_{\epsilon}C_2^2}{\delta^2\kappa_{1}}\geq
0,\nonumber
\end{align}
when the constant $\delta'$ is selected to be sufficiently small and $\kappa_1$ is sufficiently large. Here $\theta$ is the constant defined in Lemma \ref{lemm11}.

Combing (\ref{s7.13}), (\ref{s7.14}) with (\ref{s7.1}), we have proved our
lemma.
\end{proof}

Based on Lemma \ref{lem32}, an induction argument being similar to the proof of
Corollary 10 in \cite{LRW} can be applied here. Thus, we have
\begin{coro}\label{cor1} With the same assumptions as Lemma \ref{lem32},
there exists a finite sequences of positive numbers
$\{\delta_j\}_{j=i}^{n-2}$, such
that, if the following two inequalities hold for some index $1\leq r\leq
n-3$,
\begin{eqnarray}\label{last}
\dfrac{\kappa_r}{\kappa_1}\geq\delta_r,\quad and \quad
\dfrac{\kappa_{r+1}}{\kappa_1}\leq\delta_{r+1},
\end{eqnarray}
then, when  $K$ and $\kappa_1$ is sufficiently large,
(\ref{s7.1}) holds.
\end{coro}

{\bf Proof of the Theorem \ref{maintheo} for cases B3 and C:} Firstly, suppose $1\leq i\leq n-3$. In
view of Corollary \ref{cor1} and $\kappa_i\geq\kappa_1-\sqrt{\kappa_1}/n$, if we can
show that $\kappa_i\sigma_{n-3}(\kappa|i)\leq(1+\delta_0)\sigma_{n-2}(\kappa)$
and $\kappa_{n-2}$ has an upper bound, by an inductive argument, there exists some index
$i\leq r\leq n-3$ satisfying \eqref{last}, which implies our Theorem. Thus, in the following, we will prove needed two requirements case by case.

\textbf{Case B3}: By our assumption $\kappa_1\cdots\kappa_{n-2}<
2(n-2)\sigma_{n-2}(\kappa)$ and $\kappa_1\geq
\cdots\geq\kappa_{n-2}$, it is clear that $\kappa_{n-2}$ has an upper bound.

\textbf{Case C}: Since $\sigma_{n-2}(\kappa|i)\geq 0$, we have
\begin{align}\label{a7.13}
\sigma_{n-2}(\kappa)=\kappa_i\sigma_{n-3}(\kappa|i)+\sigma_{n-2}(\kappa|i)\geq\kappa_i\sigma_{n-3}(\kappa|i).
\end{align}
Combing $\sigma_{n-2}(\kappa|i)\geq 0$ with $\kappa\in\Gamma_{n-2}$,
we have $(\kappa|i)\in\bar\Gamma_{n-2}$. Using Lemma \ref{lemm10} in
(\ref{a7.13}), we get
$$
\kappa_1\cdots\kappa_{n-2}\leq\kappa_i\sigma_{n-3}(\kappa|i)\leq
\sigma_{n-2}(\kappa),
$$
which also implies that $\kappa_{n-2}$ has an upper bound.
\par
Secondly, suppose $n-2\leq i\leq n$. Thus, we have
$$\kappa_{n-2}\geq \kappa_i\geq\kappa_1-\sqrt{\kappa_1}/n.$$ However, we have proved that $\kappa_{n-2}$ always has an upper bound for both Case B3 and Case C.
Therefore, if $\kappa_1$ is sufficiently large, we have a contradiction.

\bigskip

\noindent {\it Acknowledgement:} The authors wish to thank  Professor Pengfei Guan for his valuable suggestions and comments. Part of the work was done while the authors were visiting McGill University. They would like to
 thank  their support and hospitality.

\end{document}